\documentclass[reqno, 12pt]{amsart}
\usepackage{amsmath,amssymb,amscd,mathtools}
\usepackage[centering,textwidth=6.5in,textheight=9in]{geometry}
\usepackage{esvect}
\usepackage[scr]{rsfso}
\usepackage[abs]{overpic}
\usepackage{enumitem}
\usepackage{comment}

\usepackage{tikz}
\usetikzlibrary{math,angles,quotes,intersections,calc,through}

\usepackage{pgfplots}

\usepackage[colorlinks,hyperindex,linkcolor=blue,citecolor=blue,urlcolor=blue]{hyperref}

\newtheorem{theorem}{Theorem}[section]
\newtheorem{corollary}[theorem]{Corollary}
\newtheorem{lemma}[theorem]{Lemma}
\newtheorem{proposition}[theorem]{Proposition}

\theoremstyle{definition}
\newtheorem{definition}[theorem]{Definition}
\newtheorem{remark}[theorem]{Remark}

\newtheorem{conjecture}[theorem]{Conjecture}

\numberwithin{equation}{section}

\newcommand{\bC}{\mathbb{C}}
\newcommand{\bD}{\mathbb{D}}

\newcommand{\bR}{\mathbb{R}}

\newcommand{\bZ}{\mathbb{Z}}

\newcommand{\cC}{\mathcal{C}}
\newcommand{\cD}{\mathcal{D}}

\newcommand{\cI}{\mathcal{I}}
\newcommand{\cJ}{\mathcal{J}}

\newcommand{\cL}{\mathcal{L}}
\newcommand{\cM}{\mathcal{M}}

\newcommand{\cO}{\mathcal{O}}
\newcommand{\cP}{\mathcal{P}}
\newcommand{\cQ}{\mathcal{Q}}
\newcommand{\cR}{\mathcal{R}}
\newcommand{\cS}{\mathcal{S}}
\newcommand{\cT}{\mathcal{T}}

\newcommand{\cX}{\mathcal{X}}
\newcommand{\cY}{\mathcal{Y}}

\newcommand{\fb}{\mathfrak{b}}

\newcommand{\prl}{\mathrm{prl}}

\newcommand{\crit}{\mathrm{crit}}

\newcommand{\sing}{\mathrm{sing}}
\newcommand{\sym}{\mathrm{sym}}

\let\emptyset\varnothing

\newcommand{\eps}{\epsilon}

\newcommand{\pa}{\partial}

\newcommand{\Fix}{\mathrm{Fix}}

\newcommand{\ds}{\displaystyle}

\newcommand{\ccM}{\scalebox{0.8}{\rotatebox[origin=c]{90}{$\mathsf{M}$}}}

\newcommand{\crM}{\scalebox{0.8}{\rotatebox[origin=c]{270}{$\mathsf{M}$}}}

\newcommand{\sccM}{\scalebox{0.6}{\rotatebox[origin=c]{90}{$\mathsf{M}$}}}

\newcommand{\scrM}{\scalebox{0.6}{\rotatebox[origin=c]{270}{$\mathsf{M}$}}}

\begin{document}

\title{Homoclinic and heteroclinic intersections for lemon billiards}

\author[X. Jin]{Xin Jin}
\address{Math Department, Boston College, Chestnut Hill, MA 02467.}
\email{xin.jin@bc.edu}
\thanks{X. J. is supported in part by  the NSF Grant DMS-1854232.}

\author[P. Zhang]{Pengfei Zhang}
\address{Department of Mathematics, University of Oklahoma, Norman, OK 73019.}
\email{pengfei.zhang@ou.edu}



\begin{abstract}
We study the dynamical billiards on a symmetric lemon table $\cQ(b)$,
where $\cQ(b)$ is the intersection of two unit disks with center distance $b$. 
We show that there exists $\delta_0>0$ such that for all $b\in(1.5, 1.5+\delta_0)$ (except possibly a discrete subset),
the billiard map $F_b$ on the lemon table $\cQ(b)$ admits crossing homoclinic and heteroclinic intersections. In particular, such lemon billiards have positive topological entropy. 
\end{abstract}

\maketitle

\tableofcontents

\section{Introduction}\label{sec.int}

Topological entropy of a dynamical system measures the uncertainty of the system, 
where positive topological entropy usually means chaotic behaviors and zero topological entropy is associated
with integrable  behaviors. 
There have been many results on computing the topological entropy of a dynamical system,
see \cite{Kat07} for a survey of related results.
In this paper we will study the topological entropy of lemon billiards.

Let $\cQ(b)$ be the planar domain obtained as the intersection of two unit disks $D(O_{\ell},1)$ and $D(O_{r},1)$, where $b=|O_{\ell} O_{r}| \in (0,2)$ measures the distance between their centers, see Fig.~\ref{fig.ell.po}.  
The lemon shaped billiards have been studied by Heller and Tomsovic \cite{HeTo}, in which they demonstrated a clear connection between the classical mechanics and the quantum mechanics, a phenomenon also known as quantum scarring: the eigenstates of the quantum billiards are large only at where the periodic trajectories of the classical billiards go. Numerical results in \cite[Section IV]{CMZZ} suggest that

\begin{conjecture}\label{con.main}
Let $0<b<2$ and $F_b: M_b \to M_b$ be the billiard map on the lemon table $\cQ(b)$.
Then  $h_{top}(F_b)>0$.
\end{conjecture}

There are many results about the topological entropy and the metric entropy for dynamical systems with certain geometric or topological structure of the systems.
For example, the existence of transverse homoclinic intersections for some hyperbolic fixed point implies that the system has positive topological entropy \cite[Section 6.5]{KH95}, and the existence of an invariant and eventually strictly invariant cone-field implies that the system has positive metric entropy \cite{Woj85}. 
On the other hand, it is a rather difficult task to estimate the entropy of a specific dynamical system beyond aforementioned examples. 
One of the simplest models with chaotic behaviors is the standard map (also called Chirikov--Taylor map)
\begin{align}
f_{\lambda}: \bR/2\pi \times \bR \to \bR/2\pi \times \bR, 
\quad (x,y) \mapsto (x+y+\lambda \sin x, y+\lambda \sin x), \label{eq.stand}
\end{align}
where $\lambda\in \bR$ is a parameter.
It is conjectured that the map $f_{\lambda}$ has positive metric entropy for all $\lambda\neq 0$, see  \cite[Lecture 13]{Sin94}.
Despite its simplicity and much effort for the past thirty years, this conjecture is still open   
and is considered to be one of most resistant problems in dynamical systems \cite{Kat04}.
The metric entropy of the standard map is positive for sufficiently large $\lambda$ after a super-exponentially small random perturbation \cite{BXY17}
or after certain non-generic perturbations \cite{BeTu19}.

In this paper we obtain a partial result for Conjecture \ref{con.main}.
\begin{theorem}\label{thm.main}
Let $F_b: M_b \to M_b$ be the billiard map on the lemon table $\cQ(b)$.
Then there exists $\delta_0>0$ such that the set $\{b\in (1.5, 1.5+\delta_0): h_{top}(F_b)=0\}$
has no limiting point in $(1.5, 1.5+\delta_0)$.
\end{theorem}
In other words, $h_{top}(F_b)>0$ for most $b\in (1.5, 1.5+\delta_0)$. See Fig.~\ref{phase.pt}
for phase portraits of the lemon billiards with $b=1.51$ and $b=1.54$, respectively. Only the part on the right arc of the lemon table $\cQ(b)$ is shown here, since the phase portrait is the same on the left arc. See \eqref{eq.delta0} for the choice of the number $\delta_0$.

\begin{figure}[htbp]
\includegraphics[width=2in]{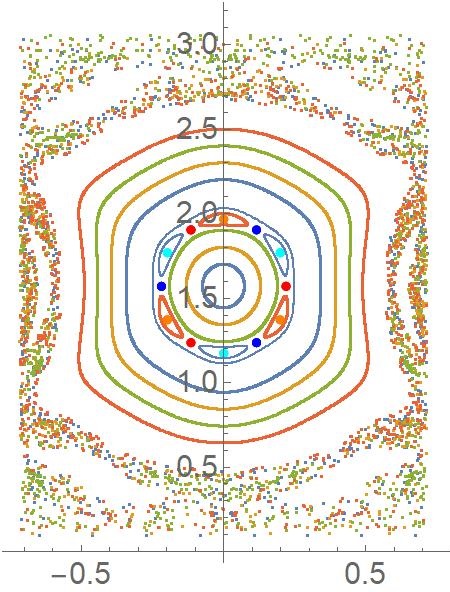}
\quad
\includegraphics[width=2in]{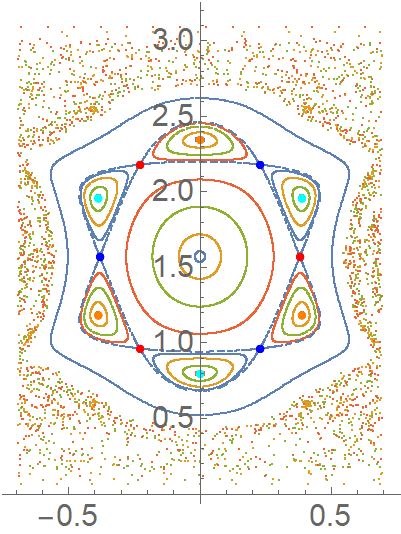}
\caption{Phase portraits of the lemon billiards for $b=1.51$ (left) and $b=1.54$ (right).}\label{phase.pt}
\end{figure}

What we actually prove is the existence of crossing  homoclinic intersections of hyperbolic periodic points of period $6$ for such lemon billiards.
Since the existence of crossing homoclinic intersections is an open property,
it follows that the  topological entropy for lemon-like billiards  is also positive after we make small perturbation of the boundary $\cQ(b)$.

There are several classes of billiards that are hyperbolic and have positive metric entropy, see \cite{CM06} for more details. The study of strictly convex billiards is rather limited.
For example, it is only proved recently \cite{XZ14} that $C^r$-generic convex planar billiards have positive topological entropy, and the first class of hyperbolic strictly convex billiards is obtained in \cite{BZZ, JZ21}.  
Unlike  \cite{XZ14}, here the object, the lemon billiards, is given, and \emph{no perturbation is allowed}.
With limited tools, we have to combine both algebraic and geometric aspects of the lemon
billiards to prove the existence of crossing homoclinic intersections. In \S~\ref{sec.gen.b.map} we introduce a new coordinate system under which the symmetries of the lemon billiards become more transparent. There are several places in later sections that we have to lean on the geometric meaning of certain sets of points that separate the phase space, since the algebraic formula for such sets are quite involved.

For comparison, it is well known that the standard map \eqref{eq.stand} has positive topological entropy.
In fact, Lazutkin \cite{Laz88} showed that the separatrices split for all hyperbolic periodic points of a globally defined complex analytic diffeomorphism. However, Lazutkin's result does not apply to the lemon billiards since the billiard map is not globally defined.
A closely related problem is the Cohen map $F(x,y)=(y, -x+\sqrt{y^2+1})$. The question about the integrability of the Cohen map was first conjectured by H. Cohen and communicated by Y. Colin de Verdiere to J. Moser in 1993. Rychlik and Torgerson \cite{RyTo98} showed that  the Cohen map $F$ does not have an algebraic integral. Based on Lowther's post \cite{Low12},  Cima, Gasull and Manosa \cite{CGM15} proved that the Cohen map is actually not $C^6$-locally integrable. It is still open whether the Cohen map has positive topological or metric entropy.

An advantage in the study of lemon billiards is the simplicity of its geometry, even though their dynamics can be rather complicated, as shown in Fig.~\ref{phase.pt}.
One could consider the asymmetric lemons with mixed phase portraits (see \cite[\S~IV]{CMZZ})
and more generally ellipse-hyperbola lens billiards (see \cite[\S~1.4]{JZ22}).
Many of the ideas in this paper could work, but with more involved analysis.

\noindent{\bf Organization of the paper.}
In Section~\ref{sec.pre}, we provide some preliminary results on homoclinic and heteroclinic intersections,
Moser's Twist Mapping Theorem, Brouwer's Fixed Point Theorem, and Mather's results on Prime-ends. Some elementary results about lemon billiards are also given.
In Section~\ref{sec.el.hy}, we give explicit formulas for  periodic orbits of period $6$ bifurcating
from the periodic orbit $\cO_2(b)$ of period $2$ of the lemon billiards $\cQ(b)$ when the parameter $b$ crosses $b=1.5$. A priori, there might be other such periodic orbits.
In Section~\ref{sec.topo} we obtain a characterization of the configurations of the periodic orbits of period $6$ that bifurcate from  $\cO_2(b)$. 
This leads to the geometric construction in Section~\ref{sec.prl} of a family of (not necessarily periodic) orbits that can be  used to connect the orbits in Section~\ref{sec.el.hy}. In Section~\ref{sec.gen.b.map}, we introduce a new coordinate system $(d_{\ell}, d_{r})$ for the lemon billiards under which one can take advantage of the various symmetries of the lemon billiards. Finally in Section~\ref{sec.homoclinic} we show that the periodic orbits found in Section \ref{sec.el.hy} are the only ones in a neighborhood of the orbit $\cO_2(b)$ 
for $b$ sufficiently close to $1.5$, and use this to prove the main theorem that $h_{top}(F_b)>0$ for most $b$ close to $1.5$.

\section{Preliminaries}\label{sec.pre}

Let $U\subseteq \bR^2$ be an open neighborhood of the origin $P=(0,0) \in \bR^2$,
$f:U\to \bR^2$ be  a symplectic embedding fixing the origin. 
The tangent map $D_{P}f$ is a $2\times 2$ matrix with determinant $1$ with two eigenvalues  $\lambda$ and  $1/\lambda$. Then the fixed point $P$ is said to be 
hyperbolic (parabolic or elliptic, respectively) if $|\lambda|>1$ ($\lambda=\pm 1$ or $\lambda\in S^1\backslash \{\pm 1\}$, respectively). If $p$ is either hyperbolic or elliptic, we say it is nondegenerate. If $p$ is hyperbolic with $\lambda>1$, then we say it is positive hyperbolic.

\subsection{Homoclinic intersections and topological entropy}
Let $p$ be a hyperbolic fixed point of $f$. Then the stable manifold $W^s(p)$ and the unstable manifold $W^u(p)$ are immersed curves in $M$ and are as smooth as the map $f$.
In particular, if $f$ is analytic, then both  $W^s(p)$ and $W^u(p)$ are analytic immersed curves.
The set $W^s(p)\backslash \{p\}$ has two components, each of which is called a stable branch of $p$. Consider the case that 
a stable branch $\Gamma^s(p)$ of $p$ intersects an unstable branch $\Gamma^u(q)$ of 
another hyperbolic fixed point $q$.
When $f$ is analytic, we have the following dichotomy for the intersection $\Gamma^s(p) \cap \Gamma^u(q)$:
\begin{enumerate}
\item either $\Gamma^s(p)=\Gamma^u(q)$: it is called a heteroclinic connection between $p$ and $q$;

\item or $\Gamma^s(p)\cap \Gamma^u(q)$ is countable: each point in $\Gamma^s(p)\cap \Gamma^u(q)$ is called a heteroclinic intersection between $p$ and $q$.
\end{enumerate}
A similar distinction can be made for the case that $q=p$, in which case we have a homoclinic
loop and a homoclinic intersection, respectively.
For brevity we will call it a saddle connection if it is either a homoclinic loop or a heteroclinic connection.

Poincar\'e discovered that the existence of a  transverse homoclinic intersection of a hyperbolic
fixed point leads to complicated dynamical behaviors. Later Smale \cite{Sma65} showed that it actually leads to the existence of an isolated hyperbolic set, the so-called horseshoe. In particular, such a dynamical system has positive topological entropy. The transversality condition can be replaced by a much weaker condition: topological crossing.
\begin{proposition}\label{pro.BW.crossing}  \cite[Theorem 2.1]{BW95}
Let $f$ be a diffeomorphism on a surface $M$ with a hyperbolic periodic point $p$.
Assume a stable branch of $p$ and an unstable branch of $p$ has a topological crossing. Then some power of $f$ admits a full-shift factor. 
\end{proposition}
Note that $h_{top}(f)>0$ whenever the map $f$ admits a full-shift factor.

\noindent{\bf Brouwer’s Fixed Point Theorem.}
Let $f: \bR^2 \to \bR^2$ be an orientation-preserving homeomorphism. Brouwer \cite{Brou12} proved the following dichotomy: either $f$ has a fixed point, or $f$ has no periodic point at all.
A slightly stronger result is the following:
\begin{proposition}\label{pro.franks} \cite[Corollary 1.3]{Fra92}
Suppose $f: \bR^2 \to \bR^2$ is a fixed point free orientation-preserving homeomorphism of the plane
and $D$ is an open topological disk such that $f(D)\cap D=\emptyset$.
Then $D$ is a wandering domain: $f^i(D) \cap f^j(D)=\emptyset$ whenever $i\neq j$.
\end{proposition}

\vskip.1in

\subsection{Twist coefficients and nonlinear stability}
An elliptic fixed point $P$ is said to be {\it non-resonant}
if its eigenvalue $\lambda$ satisfies $\lambda^n \neq 1$ for any $n \ge 3$.
More generally, an elliptic fixed point $P$ is non-resonant up to order $N$
if $\lambda^{n}\neq 1$ for each $3\le n \le N$.
If $P$ is non-resonant  up to order $(2N+2)$, then there exist an open neighborhood
$U_N\subset U$ of $P$ and a coordinate transformation
\begin{align}
h_N: U_N \to \bR^2, \begin{bmatrix} x \\ y \end{bmatrix}
\mapsto 
\begin{bmatrix} x+p_2(x,y) + \cdots + p_{2N+1}(x,y) \\ y+ q_2(x,y) + \cdots + q_{2N+1}(x,y) \end{bmatrix}
+O(r^{2N+2}), \label{formal}
\end{align}
where $p_n$ and $q_n$ are polynomials of degree $n$ for each $2\le n\le 2N+1$, 
such that
\begin{align}
h_N^{-1}\circ f\circ  h_N \Big(\begin{bmatrix} x \\ y \end{bmatrix} \Big)
=\begin{bmatrix} \cos \Theta(r^2) & -\sin\Theta(r^2) \\ \sin \Theta(r^2) & \cos\Theta(r^2)\end{bmatrix}\begin{bmatrix} x \\ y \end{bmatrix}
+ O(r^{2N+2}), \label{BNF}
\end{align} 
where  $r^2=x^2+y^2$,
$\Theta(r^2)=\theta+\tau_1 r^2+\tau_2 r^4+\cdots +\tau_N r^{2N}$,
and $\theta$ satisfies $\lambda=e^{i\theta}$. 
The function $\Theta(r^2)$ measures the amount of rotations 
of points around the fixed point $P$. 
The coordinate transformation \eqref{formal} is called Birkhoff transformation, 
the resulting form \eqref{BNF} is called Birkhoff Normal Form,
and the coefficient $\tau_k$ in the function $\Theta(r^2)$  is called
the $k$-th twist coefficient of $f$ at $P$ for each $1\le k\le N$.
See \cite{SiMo, Mos73} for more details.
Recall that an elliptic fixed point $P$ is said to be nonlinearly stable if there are nesting invariant circles accumulating on $P$.

\vskip.1in

\noindent{\bf Moser's Twist Mapping Theorem.} \cite[Theorem 2.13]{Mos73} 
Let $P$ be an elliptic fixed point of $f$ and is non-resonant up to order $(2N+2)$.
If $\tau_k \neq 0$ for some $1\le k\le N$, 
then $P$ is nonlinearly stable.

\vskip.1in

\subsection{Prime ends}
The concept of prime ends is introduced by Carath\'eodory to describe the boundary behavior of conformal maps in the complex plane in geometric terms.
Let $U\subset \bR^2=\bC$ be a bounded, simply connected domain.
The boundary $\pa U$ can be very complicated.
The prime-end compactification gives one way to compactify $U$ by adding a circle to it.
More precisely, pick a conformal map $h:\bD\to U$ from the unit disk $\bD\subset \bC$ to $U$. 
This induces a topology on the space  $\hat{U}\triangleq U\sqcup S^1$
via the extended homeomorphism  $\hat{h}:\overline{\mathbb{D}}\to\hat{U}$,
such that $\hat{h}|_{\mathbb{D}}=h$ and $\hat{h}|_{S^1}=\text{Id}$.
The topology on $\hat{U}$ is independent of the choice of $h$, and 
the space  $\hat{U}$ is called the prime-end compactification of $U$.
Let $f:U\to U$ be a homeomorphism. Then there exists  a unique 
extension $\hat{f}:\hat{U}\to \hat{U}$, which is called the prime-end extension of $f$. 
If $f$ is orientation-preserving,
then the restriction $\hat{f}|_{S^1}$ is an orientation-preserving circle homeomorphism.
Therefore, the rotation number of $\hat{f}|_{S^1}$ is well defined.
Such a number is denoted by  $\rho(f, U, \pa U)$ and is 
called the {\it Carath\'eodory rotation number} associated to the triple $(f,U, \pa U)$.

The construction of prime ends has been extended to open sets of finite type on general surfaces,
see \cite{Mat81, Mat82}. While Theorem 5.1 and Theorem 5.2 \cite{Mat81} are formulated for global surface homeomorphisms, it's worth noting that numerous propositions in \cite{Mat81} are cast within a broader context. The introductory paragraphs of both Section 8 and Section 9 shed light on this generality. More precisely, let $S$ be a closed surface,
$S_0 \subset S$ be an open subset and $f: S_0 \to S$ be an injective area-preserving homeomorphism
of $S_0$ onto an open subset $f(S_0)$ of $S$. 
Let $A \subset S_0$ be a continuum (that is, a compact and connected subset) in $S_0$ such that $f(A)=A$ and  $\widehat{S\backslash A}$ be the prime-end compactification of $S\backslash A$.  
Let $q \in A$ be a sectorial fixed point\footnote{Recall that sectorial fixed points for homeomorphisms can be viewed as a topological version of positive hyperbolic fixed points for diffeomorphisms that may allow more than 4 elementary sectors.}, $V$ be an open sector for $q$, $e\in \widehat{S\backslash A}$ a prime end.
Then $e$ is said to be a sector end associated to $V$ if there is a continuous local coordinate system
$(x, y)$ for $S$ centered at $q$ with the germ of $V$ at $q$ being defined by $y>0$, such that $V_i=\{(x,y) \in V: x^2 + y^2 \le i^{-1}\}$, $i\ge i_0$,
form a chain defining $e$ for $i_0$ sufficiently large. See the definition on Page 554 in \cite{Mat81}.

\begin{proposition}\label{pro.sec.end} \cite[Proposition 9.2]{Mat81}
Let $S_0 \subset S$ be an open subset, $f: S_0 \to S$ be an injective area-preserving homeomorphism
of $S_0$ onto its image, $A \subset S_0$ be an invariant continuum,  $e\in \widehat{S\backslash A}$ a prime end such that $\hat f(e) =e$, $q\in A$ a principal point of $e$.
Then $e$ is a sector end.
\end{proposition}

Let $L$ be an unstable branch of a positive hyperbolic fixed point $p$.
For each point $x\in L$, consider the unstable arc $[x,fx]_u\subset L$ and the limit set $\omega(L, f)$ of the sequence $f^n[x,fx]_u$, $n\ge 1$. Note that this limit set $\omega(L,f)$ is independent of the choices of 
$x\in L$ and is called the omega-set of the branch $L$. Then $L$ is said to be recurrent if $L\subset \omega(L, f)$.
For a stable branch $L$, we can define the limit set $\alpha(L,f)=\omega(L, f^{-1})$.
Then a stable branch $L$ is said to be recurrent if $L\subset \alpha(L, f)$.
A connection is a branch which is contained in the intersection of two invariant
manifolds (possibly of two different hyperbolic fixed points).

\begin{proposition}\label{pro.no.conn} \cite[Theorem 1.2. (2)]{OlCo}
Let $f: S\to S$ be an orientation-preserving and area-preserving homeomorphism.
Suppose $L$ is an invariant branch of $f$ and all fixed points contained in the closure $cl_S L$
are nondegenerate. Then either $L$ is a connection, or $L$ is recurrent and accumulates on both adjacent branches through the adjacent sectors.  
\end{proposition}

\subsection{Dynamical billiards}\label{sec.Bmap}
Let $Q\subset \bR^2$ be a connected and compact domain with (piecewise) smooth 
boundary $\pa Q$, $|\pa Q|$ be the arc-length of $\pa Q$ and
$s$ be an arc-length parameter of $\pa Q$, $0\le s < |\pa Q|$ (oriented in such a way
that $Q$ is on the left side of $\pa Q$).
Let $\gamma(s) \in \pa Q$ and $\dot\gamma(s)$ be the positive unit tangent vector of $\pa Q$
at $\gamma(s)$. Then for each $0<\theta <\pi$, $(s,\theta)$ determines a trajectory 
on $Q$ with initial position $\gamma(s)$ and initial velocity $R_{\theta}(\dot\gamma(s))$,
where $R_{\theta}$ is the rotation matrix of rotating $\theta$ counterclockwise.
Let $\gamma(s_1) \in \pa Q$ be the first intersection of the trajectory $(s,\theta)$ with the boundary $\pa Q$,
$\theta_1\in [0,\pi]$ be the angle from  $\dot\gamma(s_1)$ to the direction of the trajectory.
This defines a map $F:(s, \theta) \mapsto (s_1, \theta_1)$ on the space  $M=\pa Q\times [0, \pi]$, which is called a billiard map.  See \cite{CM06} for more details. 
Let $L(s, s_1)=|\gamma(s_1)-\gamma(s)|$ be the Euclidean distance between two points on the boundary $\pa Q$, which turns out to be  a generating function of the billiard map  \cite[Chapter 3]{Tab05}:
\begin{align*}
dL=-\cos\theta\ ds + \cos\theta_1\ ds_1.
\end{align*}
Since $d(dL)=0$, it follows that $\sin\theta_1 ds_1\wedge d\theta_1=\sin\theta ds\wedge d\theta$. That is, the billiard map 
$F$ preserves the $2$-form $\sin\theta ds\wedge d\theta$ and is  symplectic.
There is a convenient expression of the tangent map $DF$ when the domain $Q$ is convex \cite{LiWo95}. More precisely, let $L(s,\theta)= L(s, s_1(s,\theta))$ be the distance depending on $(s, \theta)$,
$R(s)$ the radius of curvature of $\pa Q$ at $\gamma(s)$, and $d(s,\theta)=R(s)\sin\theta$.
Then the tangent map $DF$ of the billiard map $F: (s,\theta)\mapsto (s_1, \theta_1)$ is given  by
\begin{align}
D_{(s,\theta)}F=\frac{1}{d(s_1, \theta_1) }
\begin{bmatrix} L(s,\theta) -d(s,\theta) & L(s,\theta) \\
L(s,\theta) -d(s,\theta) -d(s_1, \theta_1) & L(s,\theta) -d(s_1, \theta_1)
\end{bmatrix}. \label{eq.DF}
\end{align}

A tangent vector $v\in T_{x} M$ can be interpreted as a  beam of trajectories $\eta: (-\eps, \eps) \to M$ with $\eta(0)=p$ and $\dot \eta(0)=v$. 
Each line $\eta(t)$ intersects $\eta(0)$ at a point $P(t)$ (could be at infinity), $-\eps <t < \eps$.
Let $f^{+}(v)$ be the limit of the signed distance from the base point of $\eta(0)$ to $P(t)$ as $t\to 0$. Note that $f^{+}(v)$  is independent of the choice of the curve $\eta$
and is called the forward focusing distance of the tangent vector $v$.
Similarly one can define the backward focusing distance $f^{-}(v)$.
The mirror equation in optics states that
\begin{align}
\frac{1}{f^{+}(v)} + \frac{1}{f^{-}(v)} = \frac{2}{d(x)}. \label{eq.mirror}
\end{align}

\subsection{Lemon billiards}\label{pre.lemon}
Let $\cQ(b)$ be the lemon table, $M_b$ be the phase space and $F_b$ be the billiard map.
There exists a unique periodic orbit $\cO_2(b)=\{P, F_b(P)\}$ of period $2$, the orbit  bouncing back and forth on the table $\cQ(b)$ along the line $O_{\ell} O_r$.
Note that $D_P F_b =D_{F_b(P)}F_b=\begin{bmatrix} 1-b & 2-b \\ -b & 1-b \end{bmatrix}$.
In particular, the  orbit $\cO_2(b)$ is an elliptic periodic orbit for each $b\in (0,2)\backslash\{1\}$.
Moreover, for $1<b<2$, the tangent matrix $D_{P}F_b^2$ is conjugate to the rotation matrix $R_{\theta}$,
where $\theta=\arccos(2(1-b)^2 -1)$, which decreases from $\pi$ to $0$. 
It follows from \cite{KP05} that the first twist coefficient of the billiard map $F_{b}$ at the elliptic periodic orbit $\cO_2(b)$ satisfies $\tau_1(F_{b}^2,P)=\frac{1}{4}$. Note that
$\theta(1.5)=\frac{2\pi}{3}$, or equally, $\lambda^3=1$. So the periodic orbit $\cO_2(1.5)$ is $3$-resonant. Moser's theorem is still applicable to this resonance case due to the symmetry of the lemon table, see \cite{JZ22} for more details.
Therefore,  the periodic point $\cO_2(b)$ is nonlinear stable for every $1<b<2$. 
Pick a Diophantine number $\rho>\frac{1}{3}$ that is sufficient close to $\frac{1}{3}$ such that
there exists an $F_{1.5}^2$-invariant curve of with rotation number $\rho$. Then there exists a positive number $\delta_{\rho}>0$
such that for each $1.5\le b \le 1.5+\delta_{\rho}$, there exists an $F_b^2$-invariant curve $C_{\rho}(b)$ of the same rotation number $\rho$ surrounding the periodic point $P$. Moreover, the invariant curve $C_{\rho}(b)$ depends smoothly on the parameter $b$. 
Denote by $D_{\rho}(b)$ the disk bounded by $C_{\rho}(b)$, which also depends smoothly on the parameter $1.5\le b \le 1.5+\delta_{\rho}$.
For convenience, we formulate this fact as a proposition:
\begin{proposition}\label{pro.delta.rho}
Given a Diophantine number $\rho>\frac{1}{3}$ that is sufficiently close to $\frac{1}{3}$,
there exists a positive number $\delta_{\rho}>0$ 
such that for any $b\in (1.5, 1.5+\delta_{\rho})$, there exists an
$F_b^2$-invariant curve $C_{\rho}(b)$ of rotation number $\rho$ surrounding the elliptic periodic point $P\in \cO_2(b)$.
\end{proposition}

\section{Elliptic and hyperbolic periodic orbits of period 6}\label{sec.el.hy}

In this section we show the existence of elliptic and hyperbolic periodic orbits of period 6
bifurcated from the periodic orbit $\cO_2(b)$ of period $2$ when $b$ passes the value $b=1.5$.
Note that such orbits bounce alternatively between the left and right arcs of the table $\cQ(b)$.
Let $\phi_A(b)$ be the position angle of the 
corner point $A$ of the table $\cQ(b)$ with respect to the center $O_{\ell}$. It is easy to see that $\cos\phi_A(b)=\frac{b}{2}$.

\subsection{The elliptic periodic orbit of period 6}\label{sec.ell.po}
Consider the periodic orbit of period $6$ given in Fig.~\ref{fig.ell.po}. 

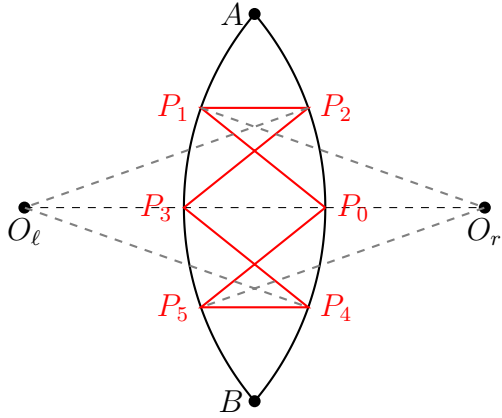
\begin{figure}[htbp]
\tikzmath{
\r=1;
\b=1.53;
\p=0.699755; 
\q=0.338071; 
}
\begin{tikzpicture}[scale=4]
\coordinate (O) at (0,0);
\coordinate (O1) at (\b, 0);
\fill (O) node[below]{$O_{\ell}$} circle[radius=0.02];
\fill (O1) node[below]{$O_{r}$} circle[radius=0.02];
\draw[dashed] (O) -- (O1);
\draw  [domain=-\p:\p, samples=100, thick] plot ({\r*cos(\x r)}, {\r*sin(\x r)});
\draw  [domain=-\p:\p, samples=100, thick] plot ({\b-\r*cos(\x r)}, {\r*sin(\x r)});
\coordinate (A) at  ({\r*cos(\p r)}, {\r*sin(\p r)});
\fill (A) node[left]{$A$} circle[radius=0.02];
\coordinate (B) at  ({\r*cos(\p r)}, {-\r*sin(\p r)});
\fill (B) node[left]{$B$} circle[radius=0.02];
\coordinate (P0) at ({\r}, {0});
\coordinate (P1) at  ({\b-\r*cos(\q r)}, {\r*sin(\q r)});
\coordinate (P2) at  ({\r*cos(\q r)}, {\r*sin(\q r)});
\coordinate (P3) at ({\b - \r}, {0});
\coordinate (P4) at ({\r*cos(\q r)}, {-\r*sin(\q r)});
\coordinate (P5) at  ({\b-\r*cos(\q r)}, {-\r*sin(\q r)});
\draw[thick, red] (P0)  node[right]{$P_0$} -- (P1) node[left]{$P_1$} -- (P2)  node[right]{$P_2$} -- (P3) node[left]{$P_3$} -- (P4)  node[right]{$P_4$} -- (P5) node[left]{$P_5$} -- (P0);
\draw[thick, gray, dashed] (P2) -- (O) -- (P4) (P1) -- (O1) -- (P5);
\end{tikzpicture}
\caption{A periodic orbit of period $6$ on the lemon table $\cQ(b)$ for $b=1.53$.}\label{fig.ell.po}
\end{figure}

To find an explicit formula for the periodic orbit showing in Fig.~\ref{fig.ell.po}, 
we set $P_2(\cos\phi, \sin\phi)$, where $\phi=\phi(b)$ is to be determined. 
Because of symmetry, $P_1P_2$ is parallel to the axis $O_{\ell} O_{r}$. Then  
$\angle P_2 O_{\ell} P_3 = \angle O_{\ell} P_2 P_1$.
Since $\angle O_{\ell} P_2 P_1 = \angle O_{\ell} P_2 P_3$, we have  $|P_2P_3|=|O_{\ell}P_3|=b-1$.
That is,
\begin{align*}
(b-1)^2 = (b-1 -\cos\phi)^2 + \sin^2\phi &=(b-1)^2 -2(b-1)\cos\phi +1.
\end{align*} 
It follows that $\cos\phi =\frac{1}{2(b-1)}$.
Note that $\cos\phi<1$ holds only if  $b>1.5$.
On the other hand, $\cos\phi>\cos\phi_A(b)=\frac{b}{2}$ holds only if
$b<\frac{1+\sqrt{5}}{2}$.
Therefore, the domain of existence  for this periodic orbit is 
\begin{align}
1.5 < b < \frac{1+\sqrt{5}}{2}.\label{eq.ell.dom}
\end{align}
Let $e_j(b)\in M_b$ be the point in the phase space corresponding to the trajectory
starting at $P_j$ and pointing to $P_{j+1}$, $0\le j \le 5$. 
This leads to a periodic orbit $\cO_6(b)=\{e_{j}(b): 0\le j \le 5\}$.
One can get another periodic orbit with exactly the same properties by reversing the direction of the trajectory. We only need to consider the orbit $\cO_{6}(b)$.
Now we show that the orbit $\cO_{6}(b)$ is elliptic. For this reason we will rename it as
$\cO_{6}^{e}(b)$.

\begin{remark}\label{rem.iterate}
Note that both the lemon table $\cQ(b)$ and the periodic orbit $\cO_{6}^{e}(b)$ are invariant with respect to the rotation of the table by angle $\pi$ around its center. It follows that the billiard maps $F_b: e_{j}(b) \to e_{j+1}(b)$ and $F_b: e_{j+3}(b) \to e_{j+4}(b)$ are the same for each $j=0, 1, 2$ (with respect to the local coordinates around these points). Therefore, the two nonlinear maps $F_b^3: e_{0}(b) \to e_{3}(b)$ and  $F_b^3: e_{3}(b) \to e_{0}(b)$ are the same, and the decomposition  $F_b^6 =(F_b^3)^2$ along the periodic orbit $\cO_{6}^{e}(b)$ can be viewed as an iterate of the nonlinear map $F_b^3$ around the point $e_{0}(b)$. It follows that a bifurcation at $e_{0}(b)$ happens under $F_b^6$ if and only if it happens under $F_b^3$. 
From this point of view, it is more natural to classify the periodic point $e_{0}(b)$ with respect to the iterate $F_b^3$. The same reasoning applies to the hyperbolic periodic orbit in \S~\ref{sec.hyp.po}
and to the periodic orbit of generalized maps in Lemma \ref{lem.E0.eph}.
\end{remark}

Let $L_j=|P_jP_{j+1}|$ and
$d_j=d(e_j(b))$ be the corresponding terms given in Section~\ref{sec.Bmap}. 
Note that $d_1=d_2=d_4=d_5=\cos\phi=\frac{1}{2(b-1)}$, 
$L_0=L_2=L_3=L_5= b-1$,
\begin{align}
d_0=d_3=\cos2\phi=\frac{1}{2(b-1)^2}-1,
\quad
L_1=L_4=2\cos\phi-b=\frac{1}{b-1}-b.
\end{align}
It follows that, as $2\times 2$-matrices, $D_{e_0(b)}F_b^6=(D_{e_0(b)}F_b^3)^2$, and
\begin{align*}
D_{e_0(b)}F_b^3&=\frac{1}{d_0^2d_1}
\begin{bmatrix}L_0 - d_1 & L_0 \\ L_0 - d_1-d_0 & L_0 -d_0\end{bmatrix}
\begin{bmatrix}L_1 - d_1 & L_1 \\ L_1 - 2d_1 & L_1 - d_1\end{bmatrix}
\begin{bmatrix}L_0 - d_0 & L_0 \\ L_0 - d_0 - d_1 & L_0 -d_1\end{bmatrix}.
\end{align*}
As mentioned in Remark~\ref{rem.iterate}, it is only necessary to consider the tangent matrix $D_{e_0(b)}F_b^3$, for which the following holds:
\begin{align}
f(b):=\frac{1}{2}\text{Tr}(D_{e_0(b)}F_b^3)
=16 b^5-48 b^4+40 b^3-4 b^2-b+1 + \frac{b}{2 b^2-4 b+1}.\label{eq.tr.e3}
\end{align}
Note that $|f(b)|<1$ for $1.5 < b < \frac{1+\sqrt{5}}{2}\approx 1.618$. It follows that the periodic orbit
is elliptic. For reasons that will be clear later, we will allow $b$ to go beyond $\frac{1+\sqrt{5}}{2}$
and consider the corresponding periodic orbit of some algebraically defined (generalized billiard) map in Section \ref{sec.gen.b.map}.
Let  $b_{\crit}\approx 1.63477$ be the solution of $f(b)=-1$, which is the critical parameter 
when the generalized orbit becomes parabolic. This parameter will come up frequently in the following sections, see Lemma \ref{lem.b.crit} and Lemma \ref{lem.E0.eph}.

\begin{figure}[htbp]
\begin{tikzpicture}
\begin{axis}[xmin=1.5, xmax=1.65, ymin=-1.2, ymax=1.2,  ytick={-1, 0, 1}]
\addplot[domain=1.5:1.64, color=red, smooth] {16*x^5 - 48*x^4 +40*x^3-4*x^2-x+1 +x/(2*x^2-4*x+1)};
\addplot[domain=1.5:1.65, color=black] {0};
\end{axis}
\end{tikzpicture}
\caption{The graph of the function $f(b)=\text{Tr}(D_{e_0(b)}F_b^3)/2$. }
\end{figure}
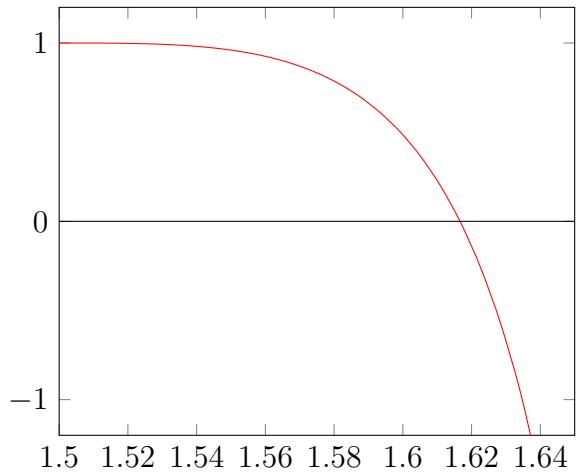

\subsection{The hyperbolic periodic orbit of period 6}\label{sec.hyp.po}
Consider the periodic orbit of period $6$ showing in Fig.~\ref{fig.hyp.po}. 

\begin{figure}[htbp]
\tikzmath{
\r=1;
\b=1.51;
\p=0.715142; 
\q=0.218837; 
\s=0.529245; 
}
\begin{tikzpicture}[scale=4]
\coordinate (O) at (0,0);
\coordinate (O1) at (\b, 0);
\fill (O) node[below]{$O_{\ell}$} circle[radius=0.02];
\fill (O1) node[below]{$O_{r}$} circle[radius=0.02];
\draw  [domain=-\p:\p, samples=100, thick] plot ({\r*cos(\x r)}, {\r*sin(\x r)});
\draw  [domain=-\p:\p, samples=100, thick] plot ({\b-\r*cos(\x r)}, {\r*sin(\x r)});
\coordinate (A) at  ({\r*cos(\p r)}, {\r*sin(\p r)});
\fill (A) node[left]{$A$} circle[radius=0.02];
\coordinate (B) at  ({\r*cos(\p r)}, {-\r*sin(\p r)});
\fill (B) node[left]{$B$} circle[radius=0.02];
\coordinate (P0) at  ({\r*cos(\q r)}, {\r*sin(\q r)});
\coordinate (P1) at ({\s*\r*cos(\q r)}, {\s*\r*sin(\q r)});
\coordinate (P2) at ({\b-\s*\r*cos(\q r)}, {-\s*\r*sin(\q r)});
\coordinate (P3) at  ({\b-\r*cos(\q r)}, {-\r*sin(\q r)});
\draw[thick, red] (P0)  node[right]{$P_0$} -- (P1) node[left]{$P_1$} -- (P2)  node[right]{$P_2$} -- (P3) node[left]{$P_3$};
\draw[thick, gray, dashed] (P1) -- (O) -- (P2) -- (O1) -- (P1);
\end{tikzpicture}
\caption{The trajectory of a hyperbolic orbit for $b=1.51$.}\label{fig.hyp.po}
\end{figure}
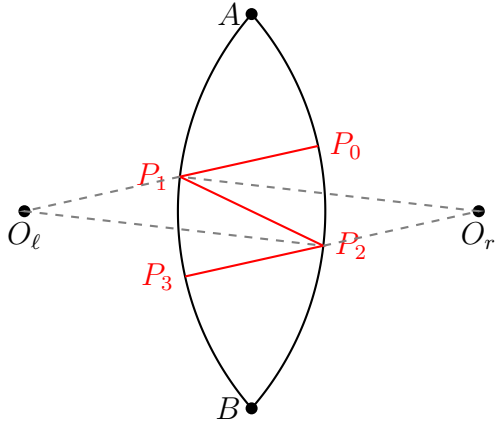

To find an explicit formula for the periodic orbit showing in Fig.~\ref{fig.hyp.po}, 
we set $P_0(\cos\phi, \sin\phi)$ for some $\phi\in (0, \phi_A(b))$, 
and $P_1(r\cos\phi, r\sin\phi)$ for some $r\in(0,1)$.
Then the remaining two points are
$P_2(b-r\cos\phi, -r\sin\phi)$ and $P_3(b-\cos\phi, -\sin\phi)$. 
The point $P_1$  satisfies $|O_r P_1|=1$, or equally,
\begin{align}
(r\cos\phi-b)^2 + r^2\sin^2\phi =r^2 - 2br\cos\phi +b^2 =1. \label{rP1}
\end{align}
Note that $\angle P_1 O_{\ell} P_2 = \angle P_1 P_2 O_{\ell}$ and hence
$|O_{\ell} P_1|= |P_1P_2|$, which leads to
\begin{align}
3r^2-4br\cos\phi + b^2=0. \label{rP1b}
\end{align} 
Eliminating $\cos\phi$ in \eqref{rP1} using \eqref{rP1b}, we have $r^2=b^2 - 2$.
Plugging $r=\sqrt{b^2 -1}$ back into  \eqref{rP1}, we have
\begin{align}
\cos\phi = \frac{r^2+b^2 -1}{2br} = \frac{2b^2 -3}{2b\sqrt{b^2 -2}}.
\end{align}

Note that $\cos\phi>\frac{b}{2}$ when $1.5<b<\sqrt{3}$.
So the domain of existence of this periodic orbit is
\begin{align}
1.5<b<\sqrt{3}. \label{eq.hyp.dom}
\end{align}
Denote $h_j(b) \in M_b$ the point starting at $P_j$ and pointing at $P_{j+1}$,
 and $h_{3+j}(b) \in M_b$ the point starting at $P_{3-j}$ and pointing at $P_{2-j}$,
for $0\le j \le 2$. This leads to a periodic orbit $\cO_{6}(b)=\{h_{j}(b): 0\le j \le 5\}$.
One can get another periodic orbit with exactly the same properties by reflecting the trajectory about the axis $AB$ through the two corners $A$ and $B$.
We only need to consider the orbit $\cO_{6}(b)$.

Next we will show that this periodic orbit $\cO_h(b)$ is hyperbolic.
For this reason we will rename it as $\cO_{6}^{h}(b)$.
Since $P_1=r(\cos\phi, \sin\phi)=(b-\cos\phi_1,\sin\phi_1)$, we have
\begin{align}
\sin\phi_1=r\sin\phi=\sqrt{b^2 - 2}\cdot \frac{\sqrt{4b^2 -9}}{2b\sqrt{b^2 - 2}}
=\sqrt{1 -\frac{9}{4b^2}}.
\end{align}
That is, $\cos\phi_1=\frac{3}{2b}$.
The orbit lengths are $L_0:= |P_0 P_1| =1-r= 1-\sqrt{b^2 - 2}$ and
$L_1 :=|P_1P_2| = |O_{r} P_1| =r = \sqrt{b^2 - 2}$.
Moreover, $d_0=d_3=1$ while 
\begin{align}
d_1 & =d_2=d_4=d_5=\cos\angle P_0 O_{\ell} P_2
=\frac{2-|P_0P_2|^2}{2} \nonumber \\
&=1-\frac{1}{2}(b^2-2(1+r)b\cos\phi +(1+r)^2)
=\frac{1}{2\sqrt{b^2 - 2}}. \label{eq.hangle.O}
\end{align}
It follows that the tangent matrix $D_{h_0(b)}F_b^6=(D_{h_0(b)}F_b^3)^2$, where
\begin{align}
D_{h_0(b)}F_b^3&=\frac{1}{d_1^2}
\begin{bmatrix}L_0 - d_1 & L_0 \\ L_0 - 1 - d_1 & L_0 - 1\end{bmatrix}
\begin{bmatrix}L_1 - d_1 & L_1 \\ L_1 - 2d_1 & L_1 -d_1\end{bmatrix}
\begin{bmatrix}L_0 - 1 & L_0 \\ L_0 - 1 - d_1 & L_0 -d_1\end{bmatrix}.
\label{eq.tan.hyp}
\end{align}
After simplifications, we have
\begin{align}
\frac{1}{2}\text{Tr}(D_{h_0(b)}F_b^3) -1
&=-16 b^4+68 b^2-72 + (16 b^4-64 b^2+63) \sqrt{b^2-2} \nonumber \\
&=(2 b-3) (2 b+3)\sqrt{b^2-2}(4 b^2-7 -4\sqrt{b^2-2})>0 \label{eq.hyp.tr}
\end{align}
on the domain $1.5 < b < \sqrt{3}$.
It follows that the periodic orbit $\cO_6^h(b)$ is hyperbolic.
In fact, it says something a little bit stronger. 
Note that the hyperbolic periodic orbit
$\cO_6^h(b)$ can only be positive hyperbolic because of the symmetry $D_{h_0(b)}F_b^6=(D_{h_0(b)}F_b^3)^2$.
Eq.~\eqref{eq.hyp.tr} says that, modulo the symmetry, $h_0(b)$ is already positive hyperbolic
with respect to $F_b^3$.

We will show in Proposition \ref{pro.no.new.po6} that there exists $\delta>0$ such that
the periodic orbits we have obtained in this section are the  only periodic orbits of period $6$ that are contained in a small neighborhood of the periodic orbit $\cO_2(b)=\{P, F_b(P)\}$ for $b\in (1.5, 1.5 +\delta_0)$. 
In \cite{Mey70} Meyer studied the generic $k$-bifurcations of periodic points for a smooth family of maps.  
We will only state the case for $k=3$. More precisely, let $f_s:U\to \bR^2$ be a family of symplectic maps 
fixing the origin $P=(0,0)$, $\lambda(s)$ be an eigenvalue of $D_{P}f_s$. Suppose $\lambda(0)=e^{2\pi i l/3}$, where $l=1$ or $2$.
Then $P$ is a $3$-bifurcation point for $f_s$ at $s=0$ if there are constants $\alpha\neq 0$ and $\gamma\neq 0$,
and symplectic action-angle coordinates $(I, \phi)$ such that
\begin{align}
f_s(I, \phi) = (I-\frac{2\gamma}{3} I^{3/2}\sin 3\phi+\cdots, \phi+\frac{2\pi l}{3}+\alpha s + \gamma I^{1/2}\cos 3\phi+\cdots). \label{eq.Mey}
\end{align}
Meyer showed that for a $3$-bifurcation point $P$, there exists exactly one  periodic orbit of period $3$ that bifurcates from it.
Moreover, this periodic orbit is hyperbolic. See also \cite[Section 11.1]{MeOf}.
Consider the periodic orbit  $\cO_2(b)=\{P, F_b(P)\}$ for the lemon billiards.
Note that $d=1$ and $L=2-b$. It follows from \eqref{eq.DF} that 
\[D_{P}F_b^2=\begin{bmatrix}2(b-1)^2 -1 & 2(b-1)(b-2) \\ 2b(b-1)  & 2(b-1)^2 -1 \end{bmatrix}.\]
When $b=1.5$, the trace $4(b-1)^2 -2=-1$ and hence the eigenvalue $\lambda(F_b^2,P)=e^{2\pi i/3}$. 
Since the table is symmetric about the horizontal axis, all second-order terms in the Taylor expansion of the billiard map $F_b$ around the point $P$ vanish, see also \cite[\S 2.1]{JZ22}.
It follows that  the constant $\gamma$ is zero in Eq.~\eqref{eq.Mey}, and the periodic point $P$ is  not a $3$-bifurcation point for the family $F_b$ at $b=1.5$ in the sense of Meyer.

\section{Periodic 6 orbits near the center of the phase space}\label{sec.topo}

In this section we obtain some  results about the possible configurations
of  periodic orbits with period $6$ that are contained in a small neighborhood of the periodic orbit $\cO_2(b)$ on the lemon table $\cQ(b)$ for $b$ close to $1.5$. 
Recall that the lemon table $\cQ(b)$ is the intersection of two unit disks centered at $O_\ell=(0,0)$ and $O_r=(b,0)$.
We will call the two circular arcs of the boundary $\pa \cQ(b)$ as $\Gamma_\ell$ (the one on the left) and $\Gamma_r$ (the one on the right), respectively. Note that there is a mismatch of the labeling between the circular arcs and their centers.

A periodic orbit of period $6$ in a small neighborhood of $\cO_2(b)$ has three reflections on each arc of the table and alternates between the two arcs. 
Denote by $\cO_6$ an oriented $6$-gon that has three (possibly repeated) points $P^\ell_1, P^\ell_2, P^\ell_3$ (indexed from top to bottom) on the left arc $\Gamma_\ell$ and three (possibly repeated) points $P^r_1, P^r_2, P^r_3$ on the right arc $\Gamma_r$ (indexed from top to bottom). 
Let  $\varphi_i^r$ be the position angle of the point $P_{j}^r$ on the arc $\Gamma_r$ with respect to the center $O_{\ell}$, $1\le j \le 3$.
Similarly, let $\varphi_j^\ell$ be the  position angle of the point $P_{j}^{\ell}$ on the arc $\Gamma_\ell$,  $1\le j \le 3$. For convenience, we introduce $\hat{\varphi}_j^\ell:=\pi-\varphi_j^\ell$. Then $\varphi_j^r>0$ (or $\hat{\varphi}_j^\ell>0$) means the point $P_j^{r}$ (or $P_{j}^{\ell}$) has positive $y$-coordinate.

The segment $O_{\ell}O_r$ divides the table $\cQ(b)$ into two parts: the upper half and the lower half.
A broken segment $P_{j_1}^\ell P_{k_1}^r P_{j_2}^\ell$ on the lemon table is said to be of the type
$(+, -, +)$ if $P_{k_1}^r$ is on the lower half while both  $P_{j_1}^\ell$ and $P_{j_2}^\ell$ are on the upper half of the table. 
Similarly one can define segments of type $(-, +, -)$.
It is easy to see that trajectories on the lemon table cannot be of these two types unless all three points are on the segment $O_{\ell}O_r$.
The following gives a characterization of the three points on $\Gamma_{\ell}$ (or $\Gamma_{r}$).
\begin{lemma}\label{lem.varphi.pm}
The periodic orbit $\cO_6$ satisfies the following dichotomy: 
\begin{enumerate}[label = (\roman*)]
\item either  $\varphi_1^{r}>0>\varphi_3^r$ and $\hat\varphi_1^{\ell}>0>\hat\varphi_3^{\ell}$;

\item or $\varphi_{j}^{r}=\hat\varphi_{j}^{\ell}=0$ for all $1\le j\le 3$: it is the periodic orbit $\cO_2(b)$ repeated three times.
\end{enumerate}
\end{lemma} 
\begin{proof}
It suffices to prove that  (ii) holds if one of the two statements in (i) fails. 
Without loss of generality, we assume $\varphi_1^{r}>0>\varphi_3^r$ does not hold. 
Then we have either $\varphi_3^r \ge 0$ or $\varphi_1^{r}\le 0$. We assume $\varphi_3^r \ge 0$. Our argument works the same in the case $\varphi_1^{r}\le 0$. 
Note that $\hat\varphi_3^{\ell} \ge 0$ (otherwise the two segments containing $P_3^{\ell}$ would be of type $(+,-,+)$).
There are two cases: $\hat \varphi_3^{\ell} \ge \varphi_3^{r}$ or $\hat \varphi_3^{\ell} \le \varphi_3^{r}$. We assume $\hat \varphi_3^{\ell} \ge \varphi_3^{r}$. Pick an index $1\le j\le 3$ such that $P_{j}^{\ell}P_{3}^{r}$ belongs to $\cO_6$. Since $\hat \varphi_j^{\ell} \ge \hat \varphi_3^{\ell} \ge \varphi_3^{r}$, the reflection at $P_{3}^{r}$ of the trajectory $\vv{P_{j}^{\ell}P_{3}^{r}}$ intersects $\Gamma_{\ell}$ at a point  $P_{k}^{\ell}$ with  $\hat \varphi_k^{\ell} \le \varphi_3^{r}$. Since $P_{3}^{r}$ is a lowest point among all six points, it follows that $k=3$ and $\hat \varphi_3^{\ell} = \varphi_3^{r}$. Then $\vv{P_{3}^{r}P_{3}^{\ell}}$ is part of the the orbit $\cO_6$, is horizontal and lies in the upper half of the table. Since $P_{3}^{r}$ and $P_{3}^{\ell}$ are the lowest points on each side, the reflection of $\vv{P_{3}^{r}P_{3}^{\ell}}$ cannot be any lower than these two points. This can only happen when $\hat \varphi_3^{\ell} = \varphi_3^{r}=0$, and $\vv{P_{3}^{\ell}P_{3}^{r}}$ is along the segment $O_{\ell}O_{r}$. 
That is, $P_{3}^{\ell}P_{3}^{r}$ is part of the orbit $\cO_2(b)$ and $\cO_6$ is the periodic orbit $\cO_2(b)$ (repeating three times). 
\end{proof}

\subsection{The combinatorial type of the periodic orbit}\label{sec.comb.type}
The trajectories of the orbit $\cO_6$ from the three points on $\Gamma_r$ to the three points on $\Gamma_\ell$ (resp. $\Gamma_\ell$ to $\Gamma_r$) gives a permutation $\sigma_{r\ell}\in S_3$ (resp. $\tau_{\ell r}\in S_3$). Note that the composition $\mu:=\tau_{\ell r}\sigma_{r\ell}$ must be a 3-cycle since they form one periodic orbit of period $6$. Therefore, $\mu=(123)\text{ or }(132)$. Alternatively, given $\sigma_{r\ell}\in S_3$, then $\tau_{\ell r}$ is determined by $\mu$ and $\sigma_{r\ell}$. So there are $|S_3|\times 2=12$ possibilities in total. On the other hand, there are three symmetries of the lemon billiards:
\begin{enumerate} 
\item the orientation reversal of the trajectory, which amounts to  
\begin{align*}
(\sigma_{r\ell}, \tau_{\ell r})&\mapsto (\sigma_{r\ell}'=\tau_{\ell r}^{-1},  \tau_{\ell r}'=\sigma_{r\ell}^{-1}),\\
\mu &\mapsto \mu'=\mu^{-1};
\end{align*}

\item the reflection about the line $x=\frac{b}{2}$, which amounts to exchanging $\tau_{\ell r}$ and $\sigma_{r\ell}$;

\item the reflection about the $x$-axis, which amounts to reversing the order of $\{1,2,3\}$, or equally, taking conjugate by $(13)$ on both $\sigma_{r\ell}$ and $\tau_{\ell r}$. 
\end{enumerate}

Modulo the above symmetries, we have $4$ different configurations classified by the length of $\sigma_{r\ell}$ and $\tau_{\ell r}$ with respect to the transpositions $(12)$ and $(23)$ (i.e. the minimal length of the expression of $\sigma_{r\ell}$ and $\tau_{\ell r}$ using only $(12)$ and $(23)$):
\begin{align}\label{eq: 4 config}
&(\sigma_{r\ell}, \tau_{\ell r})=\begin{cases}&\text{(O, II)}: (id, (123))\sim ((123), id)\sim (id, (132))\sim ((132), id), \\
&\text{(I, I)}:((12), (23))\sim ((23), (12)), \\
&\text{(I, III)}:((12), (13))\sim ((13), (12)) \sim ((23), (13))\sim ((13), (23)), \\
&\text{(II, II)}:((123), (123))\sim ((132), (132)).
\end{cases}
\end{align}

Note that certain degeneracy of $\cO_6$  can happen in the above configurations: two neighboring points on $\Gamma_\ell$ and/or $\Gamma_r$ collide. The degeneracy cases are easier to dealt with. It follows directly from Lemma \ref{lem.varphi.pm} that
\begin{lemma}
If the orbit $\cO_6$ has a triple collision on either side,  then $\cO_6$ is the periodic orbit $\cO_2(b)$ (repeating three times). 
\end{lemma} 

Next we consider the double collisions that have no triple collisions.
\begin{lemma}
Suppose  $P_1^{\ell}=P_2^{\ell} \neq P_3^{\ell}$. Then $P_1^{r}\neq P_2^{r}=P_3^{r}$, and reflections at the points
$P_{3}^{\ell}$ and $P_{1}^{r}$ are perpendicular.
\end{lemma}
\begin{proof}
Denote $P^{\ell}:=P_1^{\ell}=P_2^{\ell}$.
Then the two reflections at the double point $P^{\ell}$ reach all three points $P_j^{r}$, $1\le j \le 3$.
The bisecting lines of these two reflections intersect at both $P^{\ell}$ and $O_{r}$ and hence coincide. 
It follows that either $P_1^{r}=P_2^{r}$ or $P_2^{r}=P_3^{r}$.
We just need to show that the first case is impossible.
Suppose on the contrary that $P_1^{r}=P_2^{r}=:P^r$. Then there are two cases:
\begin{enumerate}
\item[(a1)] the segment $P_{3}^{\ell}P_{3}^{r}$ is not part of the orbit $\cO_6$:
then the reflections at both $P_{3}^{\ell}$ and $P_{3}^{r}$ bounce right back and hence are perpendicular.
So the points $P^{r}$ and $P^{\ell}$ are on the radii $O_r P_3^{\ell}$ and $O_{\ell}P_3^r$,
respectively. It follows that all six points are contained in the lower half of the lemon table, 
which can only be a triple collision along the horizontal axis by Lemma \ref{lem.varphi.pm},  contradicting the assumption that $P^{\ell} \neq P_3^{\ell}$.

\item[(a2)]  the segment $P_{3}^{\ell}P_{3}^{r}$ is part of the orbit $\cO_6$: among the six trajectory segments of $\cO_6$, exactly three of them are connected to $P_{3}^{\ell}$ or $P_{3}^{r}$ (maybe both). It follows that $P^{\ell}P^{r}$ is a multiple edge with multiplicity $3$.
This can happen only when the reflections at both points $P^{\ell}$ and $P^{r}$ are perpendicular, which leads to the periodic orbit $\cO_2(b)$ and hence a triple collision, contradicting the assumption that $P^{\ell} \neq P_3^{\ell}$.
\end{enumerate} 
In the following we assume $P^r:=P_2^{r}=P_3^{r}$. 
There are two cases:
\begin{enumerate}
\item[(b1)]  the segment $P_{3}^{\ell}P_{1}^{r}$ is not part of the orbit $\cO_6$:
then the trajectory $\vv{P^{r}P_3^{\ell}}$ has a reflection at $P_3^{\ell}$ and bounces right back to $P^{r}$. It follows that the reflections at $P_3^{\ell}$ is perpendicular. So is the reflection at $P_1^{r}$.

\item[(b2)]  the segment $P_{3}^{\ell}P_{1}^{r}$ is part of the orbit $\cO_6$: 
this case is similarly to Case (a2): there are exactly three edges connected to $P_{3}^{\ell}$ or $P_{1}^{r}$ (maybe both), and hence $P^{\ell}P^{r}$ is a multiple edge with multiplicity $3$. 
Again it leads to a triple collision, contradicting the assumption that $P^{\ell} \neq P_3^{\ell}$.
\end{enumerate} 
So the only possible degeneracy with double collisions is of the type with a double point and a perpendicular point  on the one arc and a perpendicular point  and a double point  on the other arc.
This completes the proof.
\end{proof}
An example of the above double degeneracy configuration is the hyperbolic periodic orbit $\cO_6^{h}(b)$ in Section \ref{sec.hyp.po}. It can be viewed as the limit of configuration type (I, I).
In the following we will deal with the non-degenerate cases.

An orbit segment $P_{j_1}^\ell P_{k_1}^r P_{j_2}^\ell P_{k_2}^r$ on a lemon table is said to be ``going-up'' if  $P_{j_2}^\ell$ is above $P_{j_1}^\ell$, $P_{k_2}^r$ is above $P_{k_1}^r$, and all four points are contained in the upper (or lower) half of lemon table. 
See Fig.~\ref{fig.going.up} for an illustration. The illustration looks pathological 
and we are going to  show it is impossible.

\begin{figure}[htbp]
\begin{tikzpicture}
\draw (1.6,0) circle (2);
\draw (-1.6,0) circle (2);
\coordinate (Ol) at  (-1.6,0);
\coordinate (Or) at  (1.6,0);
\coordinate (Bl) at (-0.4,0);
\coordinate (Pr1) at  ({-1.6+2*cos(15)},{2*sin(15)});
\coordinate (Pl1) at  ({1.6-2*cos(6)},{2*sin(6)});
\coordinate (Pr2) at  ({-1.6+2*cos(23)},{2*sin(23)});
\coordinate (Pl2) at  ({1.6-2*cos(25)},{2*sin(25)});
\draw[dashed] (Ol)--(Pr1) node[right] {$P_{k_1}^r$};
\draw[dashed] (Ol) node[left] {$O_\ell$}--(Or) node[right] {$O_r$};
\draw[dashed] (Or)--(Pr1) coordinate[pos=1.46](Ql);
\draw[blue] (Bl) node[below left] {$B_\ell$} -- (Pr1)--(Ql) node [left] {$Q^\ell$};
\draw[red] (Pl1) -- (Pr1) -- (Pl2) -- (Pr2);
\fill (Ol) circle[radius=0.05];
\fill (Or) circle[radius=0.05];
\fill (Bl) circle[radius=0.05];
\fill (Pl1) circle[radius=0.05];
\fill (Pr1) circle[radius=0.05];
\fill (Pl2) circle[radius=0.05];
\fill (Pr2) circle[radius=0.05];
\fill (Ql) circle[radius=0.05];
\end{tikzpicture}
\caption{An illustration of a ``going-up'' orbit $P_{j_1}^\ell P_{k_1}^r P_{j_2}^\ell P_{k_2}^r$ (red).}\label{fig.going.up}
\end{figure}
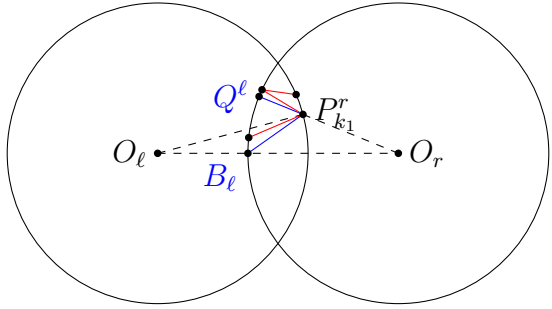

\begin{lemma}\label{lem.no.go.up}
There is no ``going-up" orbit on the lemon billiards for any $1.5 \le b \le 1+ 2^{-1/2}$.
\end{lemma}
The number $1+ 2^{-1/2}$ will appear several times later in this paper. See the discussion right after \eqref{bsquare} and the discussion right before Lemma \ref{lem.E0.eph}.
\begin{proof}
Suppose on the contrary that $P_{j_1}^\ell P_{k_1}^r P_{j_2}^\ell P_{k_2}^r$ 
is a ``going-up" orbit that is contained in the upper half of the lemon table $\cQ(b)$.
Let  $B_{\ell}=(b-1, 0)$ and $Q^{\ell}$ be the extension of the segment $O_{r}P_{k_1}^{r}$ to the  arc $\Gamma_{\ell}$.
To admit such an orbit, we must have 
\begin{enumerate}[label = (\roman*)]
\item $P_{j_2}^{\ell}$ is above $Q^{\ell}$: otherwise $P_{k_2}^r$ would be below or at $P_{k_1}^r$ after a reflection at $P_{j_2}^{\ell}$, a contradiction;

\item $P_{j_1}^{\ell}$ is below $Q^{\ell}$, since it is the reflection of the segment $P_{j_2}^{\ell}P_{k_1}^{r}$ with respect to the line $O_{\ell}P_{k_1}^r$.
\end{enumerate}
It follows that
\begin{align}
\angle O_\ell P_{k_1}^r  Q^\ell < \angle O_\ell P_{k_1}^r  P_{j_2}^{\ell}
=\angle O_\ell P_{k_1}^r  P_{j_1}^{\ell}  < \angle B_\ell P_{k_1}^r  O_\ell.
\label{eq.angle}
\end{align}
For convenience, we introduce the following three vectors:
\begin{align*}
&v_1=\vv{O_rP_{j_1}^r}=(\cos\varphi_{j_1}^r-b, \sin\varphi_{j_1}^r); \\
&v_2=\vv{P_{j_1}^r O_\ell}=(-\cos\varphi_{j_1}^r, -\sin\varphi_{j_1}^r); \\
&v_3=\vv{P_{j_1}^r B_\ell}=(-\cos\varphi_{j_1}^r+b-1, -\sin\varphi_{j_1}^r).
\end{align*}
Then it follows from \eqref{eq.angle} that 
\begin{align*}
\cos\angle O_\ell P_{k_1}^r  Q^\ell= \frac{v_1\cdot v_2}{|v_1|\cdot |v_2|} > \cos\angle  B_\ell P_{k_1}^r  O_\ell  =\frac{v_2\cdot v_3}{|v_2|\cdot |v_3|}.
\end{align*}
Simplifying the above expression, we obtain
\begin{align*}
&\frac{-1+b\cos\varphi_{j_1}^r}{(1+b^2-2b\cos\varphi_{j_1}^r)^{1/2}} 
> \frac{1-(b-1)\cos\varphi_{j_1}^r}{(1+(b-1)^2-2(b-1)\cos\varphi_{j_1}^r)^{1/2}},
\end{align*}
which is equivalent to
\begin{align}
\label{eq: ineq Pr3}(1-\cos^2\varphi_{j_1}^r)(1+2b^2\cos\varphi_{j_1}^r-2b(1+\cos\varphi_{j_1}^r))> 0.
\end{align}
Since $\cos^2\varphi_{j_1}^r \le 1$ and $b>1.5$, a necessary condition for  (\ref{eq: ineq Pr3}) is
\begin{align*}
\cos\varphi_{j_1}^r > \frac{2b-1}{2b(b-1)}.
\end{align*}
However, $\frac{2b-1}{2b(b-1)} \ge 1$ for any $1.5\le b \le 1+2^{-1/2}$, a contradiction.
This completes the proof.
\end{proof}

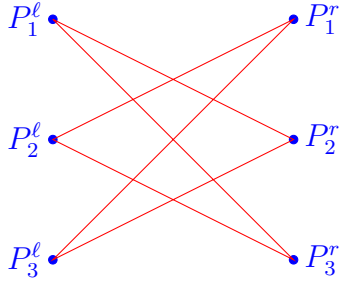
\begin{figure}[htbp]
\begin{tikzpicture}[scale=0.8]
\coordinate (l1) at  (-2,2);
\coordinate (l2) at  (-2,0);
\coordinate (l3) at  (-2,-2);
\coordinate (r1) at  (2,2);
\coordinate (r2) at  (2,0);
\coordinate (r3) at  (2,-2);
\filldraw[blue] (l1) circle (2pt) node [left] {$P_1^\ell$}; 
\filldraw[blue] (l2)  circle (2pt) node [left] {$P_2^\ell$}; 
\filldraw[blue] (l3)  circle (2pt) node [left] {$P_3^\ell$}; 
\filldraw[blue] (r1)  circle (2pt) node [right] {$P_1^r$}; 
\filldraw[blue] (r2)  circle (2pt) node [right] {$P_2^r$}; 
\filldraw[blue] (r3)  circle (2pt) node [right] {$P_3^r$}; 
\draw[red] (l1)--(r3)--(l2)--(r1)--(l3)--(r2)--(l1);
\end{tikzpicture}
\caption{A type (II, II) configuration.}\label{figure: 123,123}
\end{figure}

\begin{lemma}\label{lem.II.II}
It is impossible to have a periodic orbit of type (II, II)  on the lemon table.
\end{lemma}
\begin{proof}
Suppose on the contrary we have a periodic orbit of type (II, II), see Fig.~\ref{figure: 123,123}
for an illustration. 
It follows from Lemma \ref{lem.varphi.pm} that $\varphi_1^r>0 > \varphi_3^r$
and  $\hat{\varphi}_1^\ell > 0 > \hat{\varphi}_3^\ell$.
Then $\hat{\varphi}_2^\ell>0$ since $P_2^\ell P_1^r P_3^\ell$ cannot be of type $(-, +, -)$. 
It implies that $\varphi_3^r>0$ since  $P_1^\ell P_3^r P_2^\ell$ cannot be  $(+, -, +)$, a contradiction. This completes the proof.
\end{proof}

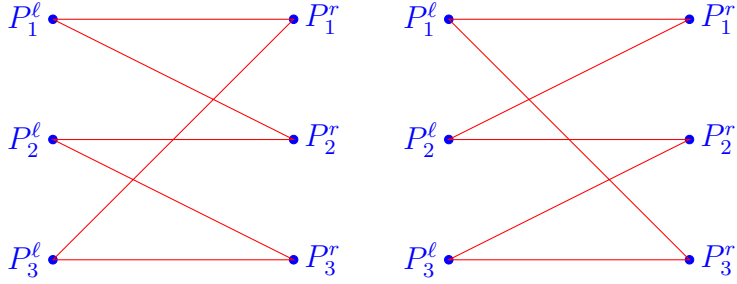
\begin{figure}[htbp]
\begin{tikzpicture}[scale=0.8]
\coordinate (l1) at  (-2,2);
\coordinate (l2) at  (-2,0);
\coordinate (l3) at  (-2,-2);
\coordinate (r1) at  (2,2);
\coordinate (r2) at  (2,0);
\coordinate (r3) at  (2,-2);
\filldraw[blue] (l1) circle (2pt) node [left] {$P_1^\ell$}; 
\filldraw[blue] (l2)  circle (2pt) node [left] {$P_2^\ell$}; 
\filldraw[blue] (l3)  circle (2pt) node [left] {$P_3^\ell$}; 
\filldraw[blue] (r1)  circle (2pt) node [right] {$P_1^r$}; 
\filldraw[blue] (r2)  circle (2pt) node [right] {$P_2^r$}; 
\filldraw[blue] (r3)  circle (2pt) node [right] {$P_3^r$}; 
\draw[red] (l1)--(r1)--(l3)--(r3)--(l2)--(r2)--(l1);
\end{tikzpicture}
\quad
\begin{tikzpicture}[scale=0.8]
\coordinate (l1) at  (-2,2);
\coordinate (l2) at  (-2,0);
\coordinate (l3) at  (-2,-2);
\coordinate (r1) at  (2,2);
\coordinate (r2) at  (2,0);
\coordinate (r3) at  (2,-2);
\filldraw[blue] (l1) circle (2pt) node [left] {$P_1^\ell$}; 
\filldraw[blue] (l2)  circle (2pt) node [left] {$P_2^\ell$}; 
\filldraw[blue] (l3)  circle (2pt) node [left] {$P_3^\ell$}; 
\filldraw[blue] (r1)  circle (2pt) node [right] {$P_1^r$}; 
\filldraw[blue] (r2)  circle (2pt) node [right] {$P_2^r$}; 
\filldraw[blue] (r3)  circle (2pt) node [right] {$P_3^r$}; 
\draw[red] (l1)--(r1)--(l2)--(r2)--(l3)--(r3)--(l1);
\end{tikzpicture}
\caption{Type (O, II) configurations: $(id, (123))$ (left) and $(id, (132))$ (right)}\label{fig.O.II}
\end{figure}

Using a similar argument, we can exclude the Type (O, II) configurations. See Fig.~\ref{fig.O.II} for an illustration. 

\begin{lemma}\label{lem.O.II}
It is impossible to have a periodic orbit of type (O, II)  on the lemon table.
\end{lemma}
\begin{proof}
We only need to consider the configuration of the type $(id, (123))$, see the left figure in Fig.~\ref{fig.O.II} for an illustration.
It follows from Lemma \ref{lem.varphi.pm} that $\varphi_1^r>0 > \varphi_3^r$
and  $\hat{\varphi}_1^\ell > 0 > \hat{\varphi}_3^\ell$.
If $\hat{\varphi}_2^\ell\geq 0$, then $\varphi_2^r\geq 0$ (otherwise $P_{1}^{\ell}P_{2}^{r}P_{1}^{\ell}$ would be of type $(+,-,+)$). Then $P_2^\ell P_2^r P_1^\ell P_1^r$ is a going-up orbit, contradicting Lemma \ref{lem.no.go.up}. It follows that $\hat{\varphi}_2^\ell<0$. Similarly, we have $\phi_2^{r} >0$ (otherwise $P_3^\ell P_3^r P_2^\ell P_2^r$ would be a going-up orbit, contradicting Lemma \ref{lem.no.go.up}). An illustration of such an orbit on the billiard table is given in Fig.~\ref{fig.id.123}.
Note that 
\begin{enumerate}[label = {(\alph*)}]
\item the distance $d(O_{\ell}, P_1^{\ell}P_2^r) > d(O_{\ell}, P_1^{\ell}P_1^r)$ since $P_2^r$ is below $P_1^r$;

\item the distance $d(O_{\ell}, P_3^{\ell}P_1^r) > d(O_{\ell}, P_2^{\ell}P_2^r)$ since $P_2^r$ is below $P_1^r$and  $P_3^\ell$ is below $P_2^\ell$;

\item the distances $d(O_{\ell}, P_3^{\ell}P_1^r) =d(O_{\ell}, P_1^{\ell}P_1^r)$ and 
$d(O_{\ell}, P_2^{\ell}P_2^r)= d(O_{\ell}, P_1^{\ell}P_2^r)$, since they are reflections at $P_1^r$ and $P_2^r$, respectively.
\end{enumerate}
Collecting terms, we have a contradiction between (a), (b) and (c). This completes the proof.
\end{proof}

\begin{figure}[htbp]
\begin{tikzpicture}
\draw (1.6,0) circle (2);
\draw (-1.6,0) circle (2);
\coordinate (Ol) at  (-1.6,0);
\coordinate (Or) at  (1.6,0);
\coordinate (Pr1) at  ({-1.6+2*cos(15)},{2*sin(15)});
\coordinate (Pr2) at  ({-1.6+2*cos(5)},{2*sin(5)});
\coordinate (Pr3) at  ({-1.6+2*cos(18)},{-2*sin(18)});
\coordinate (Pl1) at  ({1.6-2*cos(18)},{2*sin(18)});
\coordinate (Pl2) at  ({1.6-2*cos(5)},{-2*sin(5)});
\coordinate (Pl3) at  ({1.6-2*cos(15)},{-2*sin(15)});
\draw[red] (Pr1) -- (Pl1) -- (Pr2) -- (Pl2) -- (Pr3) -- (Pl3) -- cycle; 
\draw[dashed]  (Ol) -- (Or);
\draw[dashed, gray] (Pr1) -- (Ol) -- (Pr2) (Ol) -- (Pr3) (Pl1) -- (Or) -- (Pl2) (Or) -- (Pl3);
\fill (Ol) circle[radius=0.05] node [below] {$O_\ell$};
\fill (Or) circle[radius=0.05] node [below] {$O_r$};
\fill (Pl1) circle[radius=0.05] node [left] {$P_1^\ell$};
\fill (Pl2) circle[radius=0.05] node [left] {$P_2^\ell$};
\fill (Pl3) circle[radius=0.05] node [left] {$P_3^\ell$};
\fill (Pr1) circle[radius=0.05] node [right] {$P_1^r$};
\fill (Pr2) circle[radius=0.05] node [right] {$P_2^r$};
\fill (Pr3) circle[radius=0.05] node [right] {$P_3^r$};
\end{tikzpicture}
\caption{A possible configuration of type $(id, (123))$ on the lemon table}\label{fig.id.123}
\end{figure}
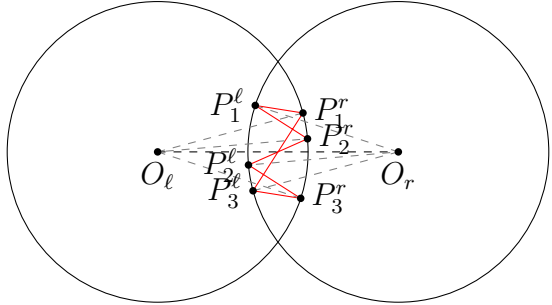

Now consider the type (I, III) configurations, see Fig.~\ref{fig.I.III} for an illustration. 

\begin{lemma}\label{lem.I.III}
It is impossible to have a periodic orbit of type (I, III) on the lemon table.
\end{lemma}
\begin{proof}
By symmetry we only need to consider the configuration ((12), (13)), see  the left-side picture in Fig.~\ref{fig.I.III}. 
It follows from Lemma \ref{lem.varphi.pm} that $\varphi_1^r>0 > \varphi_3^r$
and  $\hat{\varphi}_1^\ell > 0 > \hat{\varphi}_3^\ell$.
Then $\hat{\varphi}_2^\ell > 0$ (otherwise $P_2^\ell P_1^r P_3^\ell$  would be of the type $(-,+,-)$). Using the same argument we can show $\varphi_2^r> 0$.
The trajectory $P_2^r P_2^\ell P_1^r$ being on the upper half plane implies that $\varphi_2^r<\hat{\varphi}_2^\ell$. 
Similarly,  the trajectory $P_2^{\ell} P_2^r P_1^{\ell}$ being on the upper half plane implies that  $\hat{\varphi}_2^\ell<\varphi_2^r$, a contradiction. This completes the proof.
\end{proof}

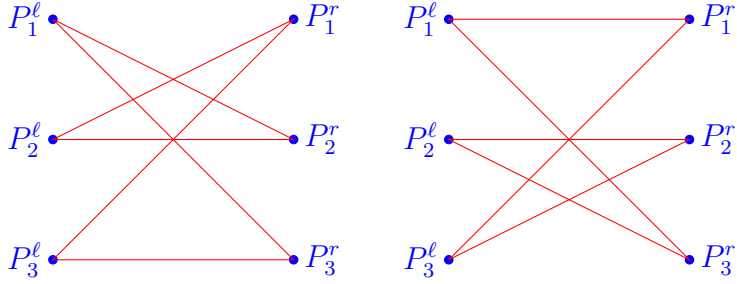
\begin{figure}[htbp]
\begin{tikzpicture}[scale=0.8]
\coordinate (l1) at  (-2,2);
\coordinate (l2) at  (-2,0);
\coordinate (l3) at  (-2,-2);
\coordinate (r1) at  (2,2);
\coordinate (r2) at  (2,0);
\coordinate (r3) at  (2,-2);
\filldraw[blue] (l1) circle (2pt) node [left] {$P_1^\ell$}; 
\filldraw[blue] (l2)  circle (2pt) node [left] {$P_2^\ell$}; 
\filldraw[blue] (l3)  circle (2pt) node [left] {$P_3^\ell$}; 
\filldraw[blue] (r1)  circle (2pt) node [right] {$P_1^r$}; 
\filldraw[blue] (r2)  circle (2pt) node [right] {$P_2^r$}; 
\filldraw[blue] (r3)  circle (2pt) node [right] {$P_3^r$}; 
\draw[red] (l1)--(r2)--(l2)--(r1)--(l3)--(r3)--(l1);
\end{tikzpicture}
\quad
\begin{tikzpicture}[scale=0.8]
\coordinate (l1) at  (-2,2);
\coordinate (l2) at  (-2,0);
\coordinate (l3) at  (-2,-2);
\coordinate (r1) at  (2,2);
\coordinate (r2) at  (2,0);
\coordinate (r3) at  (2,-2);
\filldraw[blue] (l1) circle (2pt) node [left] {$P_1^\ell$}; 
\filldraw[blue] (l2)  circle (2pt) node [left] {$P_2^\ell$}; 
\filldraw[blue] (l3)  circle (2pt) node [left] {$P_3^\ell$}; 
\filldraw[blue] (r1)  circle (2pt) node [right] {$P_1^r$}; 
\filldraw[blue] (r2)  circle (2pt) node [right] {$P_2^r$}; 
\filldraw[blue] (r3)  circle (2pt) node [right] {$P_3^r$}; 
\draw[red] (l3)--(r2)--(l2)--(r3)--(l1)--(r1)--(l3);
\end{tikzpicture}
\caption{Type (I,III):  $((12), (13))$ (left) and $((23), (13))$ (right)}\label{fig.I.III}
\end{figure}

Collecting results from Lemma \ref{lem.II.II}, \ref{lem.O.II}, \ref{lem.I.III}, we have excluded the possibilities of three of the four types defined in \eqref{eq: 4 config},
with  Type (I, I) being the only remaining type, see Fig.~\ref{fig.I.I}. 
Both periodic orbits described in Section \ref{sec.el.hy} are of this type, 
see Fig.~\ref{fig.ell.po} and \ref{fig.hyp.po} (with double collisions $P_{1}^{\ell}=P_{2}^{\ell}=P_1$ and $P_{2}^{r}=P_{3}^{r}=P_2$).

\begin{figure}[htbp]
\begin{tikzpicture}[scale=0.8]
\coordinate (l1) at  (-2,2);
\coordinate (l2) at  (-2,0);
\coordinate (l3) at  (-2,-2);
\coordinate (r1) at  (2,2);
\coordinate (r2) at  (2,0);
\coordinate (r3) at  (2,-2);
\filldraw[blue] (l1) circle (2pt) node [left] {$P_1^\ell$}; 
\filldraw[blue] (l2)  circle (2pt) node [left] {$P_2^\ell$}; 
\filldraw[blue] (l3)  circle (2pt) node [left] {$P_3^\ell$}; 
\filldraw[blue] (r1)  circle (2pt) node [right] {$P_1^r$}; 
\filldraw[blue] (r2)  circle (2pt) node [right] {$P_2^r$}; 
\filldraw[blue] (r3)  circle (2pt) node [right] {$P_3^r$}; 
\draw[red] (l1)--(r1)--(l2)--(r3)--(l3)--(r2)--(l1);
\end{tikzpicture}
\caption{An illustration of periodic orbit of type (I, I).}\label{fig.I.I}
\end{figure}
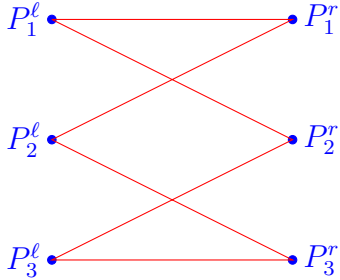

Now consider a periodic orbit of type (I, I). It follows from Lemma \ref{lem.varphi.pm} that 
$\varphi_1^r > 0 > \varphi_3^r $ and $\varphi_1^{\ell} > 0 > \varphi_3^{\ell}$.
Since the segment $P_1^{\ell}P_1^{r}$ is part of the orbit $\cO_6$ and the reflections at 
both points cannot go higher, both centers $O_{\ell}$ and $O_r$ are below (or on)
the line $P_1^{\ell}P_1^{r}$. Similarly, both centers $O_{\ell}$ and $O_r$ are above (or on)
the line $P_3^{\ell}P_3^{r}$. Let $d_{\ell, 1}$ and $d_{r,1}$ be the (signed) distance from $O_{\ell}$ and $O_{r}$ to the line $P_1^{\ell}P_1^{r}$, respectively. 
Similarly, we define the distances $d_{\ell, 3}$ and $d_{r, 3}$ from $O_{\ell}$ and $O_r$ to the line $P_3^{\ell}P_3^{r}$, respectively.  It follows that 
\begin{align}\label{eq.dlrj0}
d_{\ell, j}\ge 0, \quad d_{r,j} \ge 0, \quad j=1, 3.
\end{align}

\section{Trajectories with two parallel segments}\label{sec.prl}

In this section we will consider a special class of billiard trajectories as showing in Fig.~\ref{figure:parallel0}. This construction plays an important role
in Sections \ref{sec.gen.b.map} and \ref{sec.homoclinic}.
For convenience, we denote
\begin{align}
D=\Big\{(\alpha,\beta): 0\le \alpha, \beta \le \frac{\pi}{2}, \alpha+\beta < \frac{\pi}{2}\Big\}\backslash \{(0,0)\}.
\label{def.domD}
\end{align}
For each point $(\alpha,\beta)\in D$, we  construct a trajectory on a lemon table $\cQ(|O_{\ell}O_{r}|)$ with the left center $O_{\ell}$ at $(0,0)$ and the to-be-determined right center $O_r$.
Denote by $P_1^r$ the point on $\Gamma_r$ with coordinate angle $\alpha$ and by $P_3^r$  the point on $\Gamma_r$ with coordinate angle $-\beta$. Let $Q_1$ and $Q_3$ be the projections of $O_{\ell}$ to the horizontal lines through $P_1^r$ and $P_3^r$, respectively.
See Fig.~\ref{figure:parallel0}.

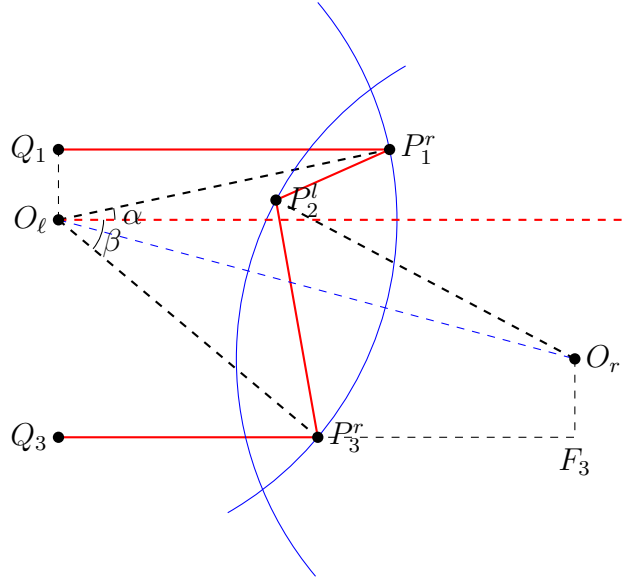
\begin{figure}[htbp]
\begin{tikzpicture}[scale=1.5]
\tikzmath{
\r=3;
\a=12;
\b=40;
\t=0.36752; 
}
\coordinate (Ol) at (-2,0);
\coordinate (P1r) at ({\r*cos(\a)-2},{\r*sin(\a)});
\coordinate (P3r) at ({\r*cos(\b)-2},{-\r*sin(\b)});
\coordinate (Q1) at (-2,{\r*sin(\a)});
\coordinate (Q3) at (-2,{-\r*sin(\b)});
\coordinate (P2l) at ({\r*cos(\a)-2-\r*\t*cos(2*\a)}, {\r*sin(\a)-\r*\t*sin(2*\a)}); 
\coordinate (Or) at ({\r*cos(\a)-2-\r*\t*cos(2*\a)+\r*cos(\a-\b)}, {\r*sin(\a)-\r*\t*sin(2*\a)+\r*sin(\a-\b)});
\coordinate (Fr) at ({\r*cos(\a)-2-\r*\t*cos(2*\a)+\r*cos(\a-\b)}, {-\r*sin(\b)});
\draw[red, thick] (Q1)--(P1r);
\draw[red, thick] (Q3)--(P3r);
\draw[thick, dashed] (Ol)--(P1r);
\draw[thick, dashed] (Ol)--(P3r);
\draw[thick, dashed] (Or)--(P2l);
\draw[red, thick] (P1r)--(P2l)--(P3r);
\draw[blue] (Ol)+({\r*cos(60)},{-\r*sin(60)})  arc (-60:40:\r);
\draw[blue] (Or)+({-\r*cos(60)},{\r*sin(60)})  arc (120:220:\r);
\draw[blue, dashed] (Ol)--(Or);
\draw[red, thick, dashed] (Ol)--(3,0);
\draw[dashed] (Ol)--(Q1); 
\draw[dashed] (Or)--(Fr) node[below] {$F_3$}--(P3r);
\draw (Ol)+(0.5,0) arc (0: \a:0.5) node[xshift=0.1 in, yshift=-0.05in]{$\alpha$};
\draw (Ol)+(0.4,0) arc (0: -\b:0.4) node[xshift=0.1 in, yshift=0.03in]{$\beta$};
\fill (Ol) circle[radius=0.05] node [left]{$O_\ell$};
\fill (P1r) circle[radius=0.05] node [right]{$P_1^{r}$};
\fill (P2l) circle[radius=0.05] node [right]{$P_2^{l}$};
\fill (P3r) circle[radius=0.05] node [right]{$P_3^{r}$};
\fill (Q1) circle[radius=0.05] node [left]{$Q_1$};
\fill (Q3) circle[radius=0.05] node [left]{$Q_3$};
\fill (Or) circle[radius=0.05] node [right]{$O_r$};
\end{tikzpicture}
\caption{The construction of an orbit segments with parallel incoming and outgoing trajectories.}\label{figure:parallel0}
\end{figure}

Suppose that $P_3^r Q_3$ is the post-collision trajectory of the incoming trajectory $Q_1 P_1$
for some lemon table $\cQ(|O_{\ell} O_r|)$ with the to-be-determined right center $O_r$. 
Then the intermediate reflection point $P_2^\ell$ is the intersection of 
the post-collision ray $\gamma_1(t)$ of $Q_1P_1^r$ and the pre-collision ray $\gamma_2(s)$ of $P_3^rQ_3$, where
\begin{align*}
\gamma_1(t)=(\cos\alpha-t\cos 2\alpha, \sin\alpha-t\sin 2\alpha)\text{ and }\gamma_2(s)=(\cos\beta-s\cos 2\beta, -\sin\beta+s\sin 2\beta),
\end{align*}
respectively. The intersection of the two rays $\gamma_1(t)$ and $\gamma_2(s)$ is 
characterized by
\begin{align*}
\begin{bmatrix}
-\cos 2\alpha&\cos 2\beta\\
-\sin 2\alpha&-\sin 2\beta
\end{bmatrix}\begin{bmatrix}t\\s
\end{bmatrix}=\begin{bmatrix}\cos\beta-\cos\alpha\\ 
-\sin\beta-\sin\alpha
\end{bmatrix},
\end{align*}
which leads to the following
\begin{align}
\begin{bmatrix}
t\\ s
\end{bmatrix}
=\frac{1}{\sin(2\alpha+2\beta)}\begin{bmatrix}\sin(\alpha+2\beta) - \sin \beta \\\sin(2\alpha+\beta) - \sin \alpha
\end{bmatrix}
=\frac{2\sin\frac{\alpha+\beta}{2}}{\sin(2\alpha+2\beta)}\begin{bmatrix}\cos\frac{\alpha+3\beta}{2}\\
\cos\frac{3\alpha+\beta}{2}
\end{bmatrix}. \label{eq.ts}
\end{align}
Note that $t+s=\frac{\cos\frac{\alpha-\beta}{2}}{\cos\frac{\alpha+\beta}{2}} \ge 1$.
It follows from \eqref{eq.ts} that the point $P_2^{\ell}$ is given by
\begin{align*}
P_2^\ell=\frac{1}{\sin(2\alpha+2\beta)}(\sin\beta\cos2\alpha+\sin\alpha\cos2\beta, \sin2\alpha\sin\beta-\sin\alpha\sin2\beta).
\end{align*}
We have the following observations:
\begin{enumerate}
\item  the $x$-component of $P_2^\ell$ is always positive: this is clear if both $\alpha$ and 
$\beta$ are less than $\pi/4$. If one of them is larger than $\pi/4$, say $\beta> \pi/4$,
then $\alpha<\pi/4$ and $\beta< \pi/2 -\alpha$, which implies
\begin{align*}
\sin\beta\cos2\alpha+\sin\alpha\cos2\beta 
\ge \sin\beta\cos2\alpha - \sin\alpha\cos2\alpha >0.
\end{align*}

\item for the $y$-component of $P_2^\ell$, note that
\begin{align*}
\sin2\alpha\sin\beta-\sin\alpha\sin2\beta=2\sin\alpha\sin\beta(\cos\alpha-\cos\beta),
\end{align*}
which is non-negative if $\alpha\leq \beta$, non-positive if $\alpha\geq \beta$,
and  zero if  either $\alpha=0$, $\beta=0$ or $\alpha=\beta$.
 \end{enumerate}

It is possible that some of the reflection points are not contained in 
the table constructed.
More precisely, on the domain $D$ given in \eqref{def.domD}, we have 
\begin{enumerate}
\item $\alpha+3\beta< \pi$ and $3\alpha+\beta< \pi$: both 
$t>0$ and $s>0$, and the trajectory is generated by regular reflections.  

\item $\alpha+3\beta= \pi$: then $t=0<s$. It means $P_2^{\ell}=P_1^r$,
and a reflection at $P_1^r$ is followed immediately by a reflection at $P_2^{\ell}$. 

\item $\alpha+3\beta> \pi$: then $3\alpha+\beta< \pi$ and hence $t<0<s$.
It means that the trajectory travels in the opposite direction by $|t|$ distance after a regular reflection at $P_1^r$,
have a regular reflection at $P_2^{\ell}$ and travels in the opposite direction by $s$ distance to reach $P_3^r$.

\item $3\alpha+\beta \ge \pi$: it is similar to the previous cases.
\end{enumerate}
See the dashed segments in Fig.~\ref{fig.3ab}.
Note that all three lines $\alpha+\beta=\frac{\pi}{2}$, $\alpha+3\beta= \pi$ and $3\alpha+\beta= \pi$
go through the point $(\frac{\pi}{4},\frac{\pi}{4})$, which is an essential singularity of 
our construction.

The center $O_r$ lies on the ray starting at $P_2^{\ell}$ and bisecting the angle $\angle P_1^r P_2^{\ell}  P_3^r$.
Therefore, the coordinate of the center $O_r$ is given by
\begin{align}\label{eq.OlOr}
(X_{\ell r},Y_{\ell r}):=\begin{bmatrix}
\frac{1}{\sin(2\alpha+2\beta)}(\sin\beta\cos2\alpha+\sin\alpha\cos2\beta)+\cos(\alpha-\beta)\\
\frac{1}{\sin(2\alpha+2\beta)}(\sin2\alpha\sin\beta-\sin\alpha\sin2\beta)+\sin(\alpha-\beta)
\end{bmatrix}.
\end{align}
Let $\fb(\alpha,\beta):=|O_\ell O_r|$. After some simplification, we get
\begin{align}
\fb(\alpha,\beta)^2=
1+\frac{\sin \alpha+\sin \beta}{\sin (\alpha+\beta)}
+\frac{\sin \alpha \sin \beta}{\sin^2(\alpha+\beta)}
+\frac{(\sin \alpha-\sin \beta)^2}{\sin^2(2 \alpha+2 \beta)}. \label{bsquare}
\end{align}
Among the four terms in the function $\fb^2$, the second term is bounded from below by $1$ and the sum of the last two terms
is bounded from below by $1/4$. It follows that $\fb(\alpha, \beta) \ge 1.5$ for $(\alpha, \beta)\in D$.
Using the linear approximation $\sin x \approx x$, we have $\fb(\alpha, \beta) \to 1.5$ as $(\alpha, \beta)\in D\to (0,0)$. 
It is easy to see that $\fb(\alpha, \beta) \to \infty$ when $\alpha+ \beta \nearrow \frac{\pi}{2}$ 
(except the point $(\frac{\pi}{4},\frac{\pi}{4})$, at where the limit is path-dependent).
In fact, along the diagonal $\alpha=\beta$, we have $\fb(\beta,\beta) =1+\frac{1}{2\cos\beta}\to b_{\max}:=1+ \frac{1}{\sqrt{2}}$ when $\beta\nearrow \frac{\pi}{4}$. Note that $b_{\max}$ also appears in Lemma \ref{lem.no.go.up} and in Section \ref{sec.gen.b.map}, see the discussion right before Lemma \ref{lem.E0.eph}.
See Fig.~\ref{fig.3ab} for two level curves
with $\fb=1.6$ (red), $\fb=b_{\max}$ (blue) and $b=2$ (brown). Note that the level curve
$\fb^{-1}(b)$ is smooth for $1.5<b< b_{\max}$
and has a singularity at the point $(\frac{\pi}{4},\frac{\pi}{4})$ for $b\ge b_{\max}$. 
For reasons that will clear later, we only need to consider the part of domain with $\fb\le b_{\max}$, on which we do have $t>0$ and $s>0$.

\begin{remark}\label{rem.origin}
Note that we have excluded the origin $(\alpha, \beta)=(0,0)$ from the definition \eqref{def.domD} of the domain $D$. In principle, one can extend our construction to the case $(\alpha, \beta)=(0,0)$ by taking limit $(\alpha, \beta)\in D\to (0,0)$. Note that this would correspond to the trajectory of the periodic orbit $\cO_2(b)$ for one particular parameter $b=1.5$. The trajectory of
$\cO_{2}(b)$ for $b\neq 1.5$ does not come from the above construction.
We have excluded $(0,0)$ from consideration to avoid possible ambiguity.
\end{remark}

\begin{figure}[htbp]
\begin{overpic}[width=2in]{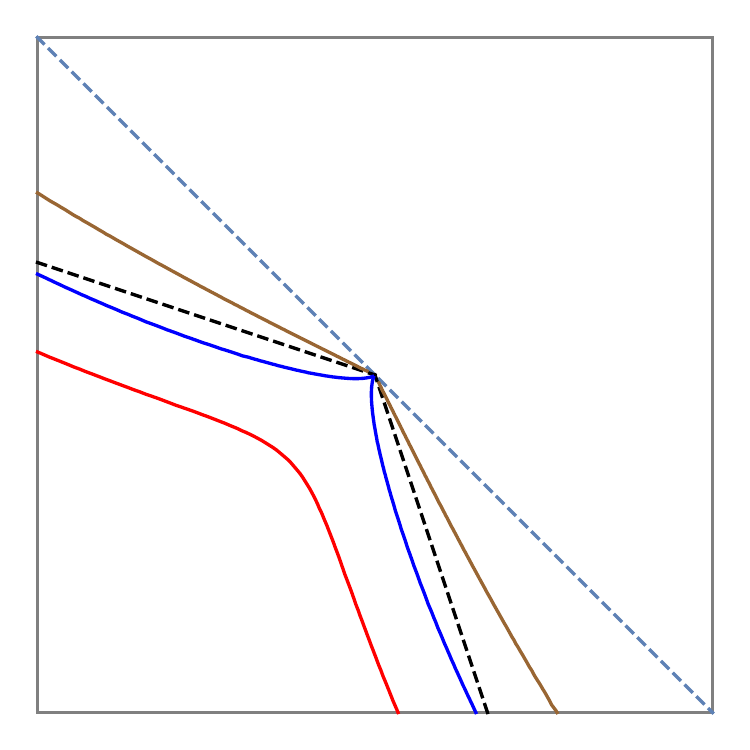}
\put (0, 0) {$0$} %
\put (67, -3) {$\frac{\pi}{4}$} %
\put (-1, 67) {$\frac{\pi}{4}$} %
\put (135,-3) {$\frac{\pi}{2}$} %
\put (-1, 135) {$\frac{\pi}{2}$} %
\end{overpic}
\caption{Plots of the level curves of $\fb(\alpha, \beta)=b$ with $b=1.6$ (red), $b=1+2^{-1/2}$ (blue) and $b=2$ (brown) on the $(\alpha, \beta)$-plane, with two segments  $\alpha+ 3\beta =\pi$ and $3\alpha+ \beta =\pi$  (both dashed).} \label{fig.3ab}
\end{figure}
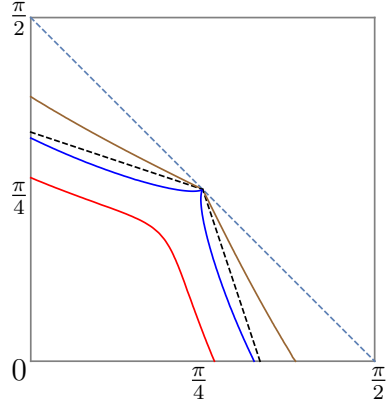

The above construction can be viewed as a deformation between the elliptic and hyperbolic periodic orbits of period $6$ in Section~\ref{sec.el.hy}:
\begin{lemma}\label{lem.ab.special}
The $\crM$-type segment is part of the elliptic periodic orbit $\cO_{6}^e(b)$ when $\alpha=\beta$, and is part of the hyperbolic periodic orbit $\cO_{6}^h(b)$ when $\alpha=0$ or $\beta=0$.
\end{lemma}
\begin{proof}
Suppose $\alpha=\beta$. Then $b=\fb(\beta,\beta)=1+\frac{1}{2\cos\beta}$, or equally,
$\cos\beta=\frac{1}{2(b-1)}$. This is exactly the parameter for the elliptic periodic orbit $\cO_{6}^e(b)$
found in Section~\ref{sec.ell.po} with $\beta=\phi$.

Next we assume $\alpha=0$. Then $b^2=\fb(\beta,\beta)^2=2+\frac{1}{4\cos^2\beta}$, or equally,
$\cos\beta=\frac{1}{2\sqrt{b^2-2}}$. This is exactly the parameter for the hyperbolic periodic orbit $\cO_{6}^h(b)$
found in Section~\ref{sec.hyp.po}  with $\beta=\angle P_0 O P_2 $, see also Eq.~\eqref{eq.hangle.O}. The case $\beta=0$ is the same. This completes the proof.
\end{proof}

Flipping the table about the horizontal line through $O_{\ell}$ if necessary, we may assume
$\alpha \le \beta$. In the following we will frequently restrict our discussion to the sub-domain 
$D'=\{(\alpha, \beta) \in D: \alpha \le \beta\}$.
All of our conclusions hold for the domain $D\backslash D'$ with by switching the variables.
One can even rotate the table and consider orbits coming from the right-hand side, see the trajectory colored in cyan in Fig.~\ref{figure:parallel}.

At the end of Section \ref{sec.topo}, we have introduced the following signed distances, see also \eqref{eq.dlrj0}. More precisely, let $d_{\ell, 1}$ and $d_{r,1}$ be the (signed) distances from $O_{\ell}$ and $O_{r}$ to the line $Q_1 P_1^{r}$,  respectively.
Similarly, we define the distances $d_{\ell, 3}$ and $d_{r, 3}$ from $O_{\ell}$ and $O_r$ to the line $P_3^{r} Q_3$, respectively. 
Note that they are functions of $(\alpha,\beta)$,  $d_{\ell, 1}=\sin\alpha$ and $d_{\ell, 3}=\sin\beta$. 
Moreover, $d_{\ell, 1} +d_{\ell, 3} =d_{r, 1} +d_{r, 3}$, which is the distance
between the two lines $Q_1 P_1^{r}$ and $P_3^r Q_3$. 
Alternatively, we have 
\begin{align}
d_{\ell, 3} -d_{r, 3} =d_{r, 1} -d_{\ell, 1}. \label{eq.equal.diff}
\end{align} 
This  characterizes the trajectory corresponding to each $(\alpha, \beta) \in D$.

\begin{lemma}\label{lemma: dl leq dr}
$d_{\ell, 1}\leq d_{r,3}$ for each $(\alpha, \beta)\in D'$,
and the equality holds if and only if one of the following holds: 
\begin{enumerate}
\item $\alpha=0$: then $d_{\ell ,1} =d_{r,3}=0$;

\item $\alpha=\beta$: then $d_{\ell ,1}=d_{r,3}=\frac{1}{2}|Q_1Q_3|$.
\end{enumerate}
\end{lemma}
\begin{proof}
It follows from \eqref{eq.OlOr} that 
\begin{align}
d_{r,3}& =\sin\beta+Y_{\ell r}=\sin\beta+\frac{2\sin\alpha\sin\beta(\cos\alpha-\cos\beta)}{\sin(2\alpha+2\beta)}+\sin(\alpha-\beta). \label{eq.dr.prl}
\end{align}
Then the inequality $d_{\ell ,1}\leq d_{r,3}$ is equivalent to
\begin{align}
\sin(2\alpha+2\beta)(\sin(\beta-\alpha)+\sin\alpha -\sin\beta)
&\leq 2\sin\alpha\sin\beta(\cos\alpha-\cos\beta), \nonumber\\
\sin(2\alpha+2\beta)\sin\frac{\beta-\alpha}{2}(\cos\frac{\beta-\alpha}{2}-\cos\frac{\beta+\alpha}{2})
&\leq 2\sin\alpha\sin\beta \sin\frac{\beta-\alpha}{2}\sin\frac{\beta+\alpha}{2}.  
\label{eq.dldrab}
\end{align}
It is easy to see that the equality in \eqref{eq.dldrab} holds if and only if $\alpha=0$ or $\alpha=\beta$.
In the following we assume $0<\alpha< \beta$. Then the inequality \eqref{eq.dldrab} is equivalent to
\begin{align*}
\sin(2\alpha+2\beta)\sin\frac{\alpha}{2}\sin\frac{\beta}{2}
&\leq\sin\alpha\sin\beta\sin\frac{\beta+\alpha}{2},\\
\cos(\alpha+\beta)\cos\frac{\alpha+\beta}{2} & \le \cos\frac{\alpha}{2}\cos\frac{\beta}{2}.
\end{align*}
The last one does hold for all $(\alpha, \beta) \in D'$.
This completes the proof. 
\end{proof}

Since $\sin(2\alpha +2\beta)>0$ on $D$, the $y$-component of the center $O_{r}$ satisfies $Y_{\ell r}=0$ if and only if 
\begin{align*}
2\sin\alpha\sin\beta(\cos\alpha-\cos\beta)-\sin(2\alpha+2\beta)\sin(\beta-\alpha)=0,
\end{align*}
or equally,
\begin{align}
\sin\frac{\alpha+\beta}{2}\sin\frac{\alpha-\beta}{2}\big((\cos\alpha+\cos\beta)\cos(\alpha+\beta)-\sin\alpha\sin\beta\big) =0. \label{meaningJ}
\end{align}
In particular, $Y_{\ell r}=0$ if either $\alpha=\beta$ or 
\begin{align}\label{eq.green.curve}
F_{\cJ}(\alpha,\beta):=(\cos\alpha+\cos\beta)\cos(\alpha+\beta)-\sin\alpha\sin\beta=0.
\end{align}
Denote by $\cJ$ the curve defined by \eqref{eq.green.curve}.
It intersects the diagonal $\alpha=\beta$ at $\alpha=\alpha_0$, where $\alpha_0$ is the solution of 
$2\cos\alpha \cos 2\alpha-\sin^2\alpha=0$. Note that
$\alpha_0\approx 0.663742$, which is larger than $\pi/5$.
We record the following simple fact for easy reference.
Recall the number $b_{\crit}\approx 1.63477$ is the solution of the equation $f(b)=-1$,
where  $f(b)$ is given in \eqref{eq.tr.e3}.
\begin{lemma} \label{lem.b.crit} 
Let $(\alpha_0, \alpha_0)$ be the intersection point of the curve $\cJ$ with the diagonal.
Then $\fb(\alpha_0, \alpha_0)=b_{\crit}$. 
\end{lemma}
\begin{proof}
Note that $\alpha_0$ satisfies the equation
$2\cos\alpha\cos2\alpha -\sin^2\alpha=0$, or equally, 
$x_0=\cos\alpha_0$ satisfies $4 x^3+x^2-2 x-1=0$.
On the other hand, plugging $b=\fb(\alpha, \alpha)=1+\frac{1}{2\cos\alpha} =1+\frac{1}{2x}$ into the function $f(b)$ in Eq.~\eqref{eq.tr.e3}, we have
\begin{align*}
f(1+\frac{1}{2x})+1=-\frac{(4 x^3+x^2-2 x-1) (4 x^4-2 x^3-3 x^2+2 x+1)}{2 x^5(1-2x^2)}.
\end{align*}
Hence $x_0$ satisfies $f(1+\frac{1}{2x})+1=0$. It follows that $\fb(\alpha_0, \alpha_0)=1+\frac{1}{2x_0}=b_{\crit}$. 
\end{proof}

\begin{lemma}\label{lem.cJ}
The curve $\cJ$ is a smooth curve contained in $\{(\alpha, \beta)\in D: \alpha+\beta\ge \frac{\pi}{3}\}$ with an endpoint in $\{\beta=0\}$ (resp. $\{\alpha=0\}$) with coordinate $(0,\frac{\pi}{2})$ (resp. $(0,\frac{\pi}{2})$). 
Moreover, the normal directions of $\cJ$ (defined by the gradient of $F_\cJ$ along $\cJ$) is contained in the positive  open cone generated by the rays $[-1,-1]$ and $[-1,1]$ (resp. $[1,-1]$) for $\alpha<\beta$ (resp. $\alpha>\beta$). 
\end{lemma}
\begin{proof}
It is easy to see that 
for each $(\alpha,\beta)\in D$,
\begin{align*}
(\partial_{\alpha}+\partial_\beta)F_{\cJ}&=\frac{1}{2} (-2 \sin (\alpha+\beta)-3 \sin (2 \alpha+\beta)-3 \sin (\alpha+2 \beta)-\sin \alpha-\sin \beta)<0.
\end{align*}
It follows that $\text{grad}F_\cJ\neq 0$ on $D$ and hence
the level curve $\cJ$ is smooth. 
Similarly, we have
\begin{align}
(\partial_{\alpha}-\partial_\beta)F_{\cJ}&=\sin(\alpha-\beta)+(\sin \beta-\sin \alpha) \cos (\alpha+\beta)\nonumber\\
&=2\sin(\frac{\alpha-\beta}{2})\big(\cos(\frac{\alpha-\beta}{2})-\cos(\frac{\alpha+\beta}{2})\cos(\alpha+\beta)\big)<0, \label{eq.v2FJ}
\end{align}
for each $(\alpha,\beta)\in D' $. 
It follows that $\text{grad}F_\cJ$ is contained in the cone 
$\{v=(a,b): a+b<0, a-b<0\}$ on $D'$.
Let $V_F$ be the unit vector field on $[0, \frac{\pi}{2}]^2\cap \{\alpha+\beta\leq \frac{\pi}{2}\}-\{(0,0)\}$ defined by that $\{\text{grad}F_\cJ, V_F\}$ is everywhere orthogonal and gives the positive orientation. Then $\cJ$ is an integrating curve of $V_F$ and is symmetric about the diagonal $\{\alpha=\beta\}$. Since $V_F$ is everywhere contained in the open positive cone generated by $[1,-1]$ and $[-1,-1]$ (resp. $[1,1]$) for $\alpha<\beta$ (resp. $\alpha>\beta$), we get $\cJ$ intersects the diagonal exactly once at $(\alpha_0, \alpha_0)$, and is contained in the upper right corner cut out by the tangent line at $(\alpha_0, \alpha_0)$.
\end{proof}

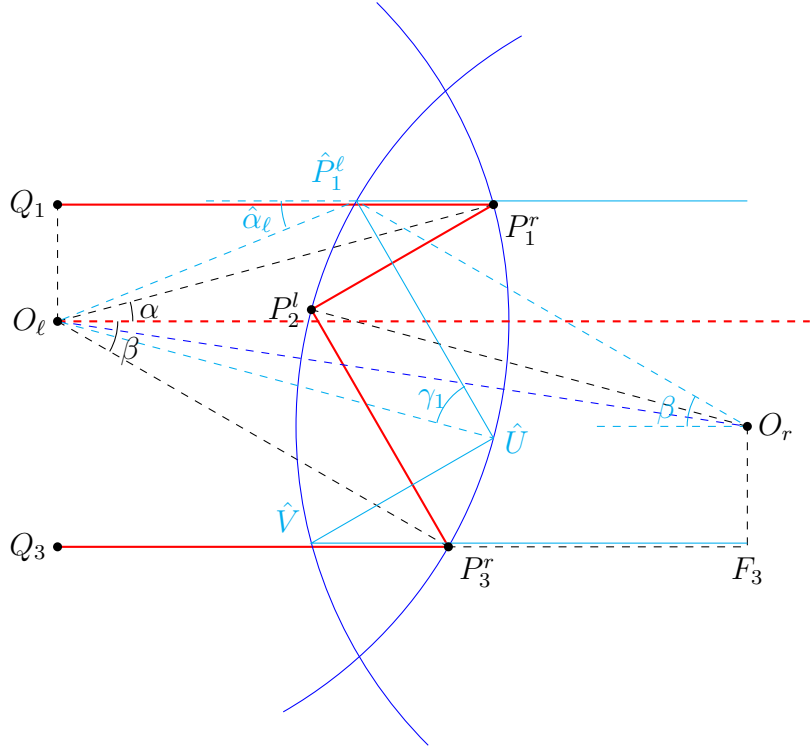
\begin{figure}[htbp]
\begin{tikzpicture}[scale=2]
\tikzmath{
\r=3;
\a=15;
\b=30;
\t=0.465926; 
}
\coordinate (Ol) at (-2,0);
\coordinate (P1r) at ({\r*cos(\a)-2},{\r*sin(\a)});
\coordinate (P3r) at ({\r*cos(\b)-2},{-\r*sin(\b)});
\coordinate (Q1) at (-2,{\r*sin(\a)});
\coordinate (Q3) at (-2,{-\r*sin(\b)});
\coordinate (P2l) at ({\r*cos(\a)-2-\r*\t*cos(2*\a)}, {\r*sin(\a)-\r*\t*sin(2*\a)}); 
\coordinate (Or) at ({\r*cos(\a)-2-\r*\t*cos(2*\a)+\r*cos(\a-\b)}, {\r*sin(\a)-\r*\t*sin(2*\a)+\r*sin(\a-\b)});
\coordinate (Fr) at ({\r*cos(\a)-2-\r*\t*cos(2*\a)+\r*cos(\a-\b)}, {-\r*sin(\b)});
\coordinate (Q3p) at ({\r*cos(\a)-2-\r*\t*cos(2*\a)+\r*cos(\a-\b)}, {\r*sin(\a)-\r*\t*sin(2*\a)+\r*sin(\a-\b)-\r*sin(\a)});
\coordinate (Q1p) at ({\r*cos(\a)-2-\r*\t*cos(2*\a)+\r*cos(\a-\b)}, {\r*sin(\a)-\r*\t*sin(2*\a)+\r*sin(\a-\b)+\r*sin(\b)});
\coordinate (P3lp) at ({\r*cos(\a)-2-\r*\t*cos(2*\a)+\r*cos(\a-\b)-\r*cos(\a)}, {\r*sin(\a)-\r*\t*sin(2*\a)+\r*sin(\a-\b)-\r*sin(\a)});
\coordinate (P1lp) at ({\r*cos(\a)-2-\r*\t*cos(2*\a)+\r*cos(\a-\b)-\r*cos(\b)}, {\r*sin(\a)-\r*\t*sin(2*\a)+\r*sin(\a-\b)+\r*sin(\b)});
\coordinate (P2rp) at ({\r*cos(\a)-2-\r*\t*cos(2*\a)+\r*cos(\a-\b)-\r*cos(\a)+\r*\t*cos(2*\a)}, {\r*sin(\a)-\r*\t*sin(2*\a)+\r*sin(\a-\b)-(\r*sin(\a)-\r*\t*sin(2*\a))}); 
\draw[red, thick] (Q1)--(P1r);
\draw[red, thick] (Q3)--(P3r);
\draw[dashed] (Ol)--(P1r);
\draw[dashed] (Ol)--(P3r);
\draw[dashed] (Or)--(P2l);
\draw[red, thick] (P1r)--(P2l)--(P3r);
\draw[blue] (Ol)+({\r*cos(60)},{-\r*sin(60)})  arc (-60:45:\r);
\draw[blue] (Or)+({-\r*cos(60)},{\r*sin(60)})  arc (120:225:\r);
\draw[blue, dashed] (Ol)--(Or);
\draw[red, thick, dashed] (Ol)--(3,0);
\draw[dashed] (Ol)--(Q1); 
\draw[dashed] (Or)--(Fr) node[below] {$F_3$}--(P3r);
\draw[cyan] (Q3p)--(P3lp) node[above left]{$\hat V$} --(P2rp) node[right]{$\hat U$} --(P1lp) node[above left]{$\hat P_1^{\ell}$} --(Q1p);
\draw [cyan, dashed]  (Or)+(-1,0) -- (Or) -- (P1lp);
\draw[cyan] (Or)+(-0.4,0) arc (180: 150:0.4) node[xshift=-0.15 in, yshift=-0.07in]{$\beta$};
\draw [cyan, dashed]  (P1lp)+(-1,0) -- (P1lp) -- (Ol) -- (P2rp) node[pos=0.87](gamma){};
\draw[cyan] (P1lp)+(-0.5,0) arc (180: 200:0.5) node[xshift=-0.15 in, yshift=0.04in]{$\hat\alpha_{\ell}$};
\draw[cyan] (gamma) arc (165: 120:0.4) node[xshift=-0.17 in, yshift=-0.07in]{$\gamma_1$};
\draw (Ol)+(0.5,0) arc (0: \a:0.5) node[xshift=0.1 in, yshift=-0.05in]{$\alpha$};
\draw (Ol)+(0.4,0) arc (0: -\b:0.4) node[right]{$\beta$};
\fill (Ol) circle[radius=0.03] node [left]{$O_\ell$};
\fill (P1r) circle[radius=0.03] node [below right]{$P_1^{r}$};
\fill (P2l) circle[radius=0.03] node [left]{$P_2^{l}$};
\fill (P3r) circle[radius=0.03] node [below right]{$P_3^{r}$};
\fill (Q1) circle[radius=0.03] node [left]{$Q_1$};
\fill (Q3) circle[radius=0.03] node [left]{$Q_3$};
\fill (Or) circle[radius=0.03] node [right]{$O_r$};
\end{tikzpicture}
\caption{The $\crM$-orbit from left (red) and the $\ccM$-orbit from the right (cyan)}\label{figure:parallel}
\end{figure}

In Fig.~\ref{figure:parallel} we have two orbits: the $\ccM$-orbit on right side (cyan) is obtained from the $\crM$-orbit on the left (red)
by rotating the table around the midpoint of $O_{\ell}O_r$ by $\pi$. So these two orbits have exactly the same characterizations, and the following result works for both orbits. 
For later application, we will work with the $\ccM$-orbit. 
For any point $p\in M_b$, let $v_p \in T_{p} M_b$ be the tangent vector with backward
focusing distance $f^{-}(v_p)=\infty$, see \S \ref{sec.Bmap} for more details.

\begin{lemma}\label{lem.foc.infty}
Let $(\alpha,\beta)\in D$, $b=\fb(\alpha, \beta)$, $p=p(\alpha, \beta)\in M_b$
be the point with first reflection at $\hat P_1^{\ell}$ coming from its right side, $v_p \in T_{p}M_b$ be the vector with $f^{-}(v_p)=\infty$, and $f_{j}^{\pm}(v_p)=f^{\pm}(DF_b^{j-1}(v_p))$, $1\le j \le 3$.
Then $f_3^{+}(v_p)=\infty$ if and only if 
\begin{align}
F_{\cT}(\alpha, \beta):=
&\cos^2(\alpha+\beta)\big((\cos\alpha+\cos\beta)\cos(\alpha+\beta)-\sin\alpha\sin\beta)\big) \nonumber \\ &+(1+\cos(\alpha+\beta))(\sin\alpha-\sin\beta)^2=0. \label{eq.lem.focal.infty}
\end{align}
\end{lemma}

\begin{proof}
Let $\tau_1=|\hat{P}_1^\ell\hat{U}|$ and $\tau_2=|\hat{U}\hat{V}|$, see Fig.~\ref{figure:parallel}.
Suppose $f_3^{+}(v_p)=\infty$. It follows from \eqref{eq.mirror} that $f_3^-(v_p)=\frac{\cos\alpha}{2}$, and  
\begin{align*}
&f_2^-(v_p)=\tau_1-f_1^+(v_p)=\tau_1 -\frac{\cos\beta}{2}, \\
&f_2^+(v_p)=\tau_2-f_3^-(v_p) =\tau_2 -\frac{\cos\alpha}{2}.
\end{align*} 
It follows that \eqref{eq.mirror} again, we get
\begin{align}\label{eq: focal infty}
\frac{1}{\tau_1-\frac{\cos\beta}{2}}+\frac{1}{\tau_2-\frac{\cos\alpha}{2}}=\frac{2}{\cos(\alpha+\beta)}.
\end{align}
On the other hand, we have (similarly to Eq.~\eqref{eq.ts}, just with $\alpha$ and $\beta$ switched) 
\begin{align*}
&\tau_1\sin2\beta+\tau_2\sin2\alpha=\sin\alpha+\sin\beta, \\
&\cos\beta-\tau_1\cos2\beta=\cos\alpha-\tau_2\cos2\alpha,
\end{align*}
from which we solve for $\tau_1,\tau_2$:
\begin{align*}
\begin{bmatrix}
\tau_1\\ \tau_2
\end{bmatrix}
=&\frac{1}{\sin(2\alpha+2\beta)}
\begin{bmatrix}
\sin(2\alpha+\beta) -\sin\alpha \\ \sin(2\beta+\alpha) -\sin\beta
\end{bmatrix}.
\end{align*}
Plugging $\tau_1$ and $\tau_2$ into \eqref{eq: focal infty} and simplifying it, we have
\begin{align*}
&4 \sin ^4\left(\frac{\alpha+\beta}{2}\right) \big(4 (\cos \alpha+\cos \beta) \cos ^3(\alpha+\beta)\\
&-(\sin (2 \alpha+\beta)+2 \sin \alpha-3 \sin\beta) (\sin (\alpha+2 \beta)-3 \sin \alpha+2 \sin\beta)\big)=0. 
\end{align*}

In the prescribed range of $\alpha$ and $\beta$, it is equivalent to 
\begin{align}\label{ineq: focal a,b}
&4 (\cos \alpha+\cos \beta) \cos ^3(\alpha+\beta)\\
&-(\sin (2 \alpha+\beta)+2 \sin \alpha-3 \sin\beta) (\sin (\alpha+2 \beta)-3 \sin \alpha+2 \sin\beta)=0. 
\end{align}
Using
\begin{align*}
&\sin (2 \alpha+\beta)+2 \sin \alpha-3 \sin\beta=2\sin\alpha\cos(\alpha+\beta)+2\sin\alpha-2\sin\beta\\
&\sin (2 \alpha+\beta)+2 \sin \beta-3 \sin\alpha=2\sin\beta\cos(\alpha+\beta)+2\sin\beta-2\sin\alpha, 
\end{align*}
we get exactly Eq.~\eqref{eq.lem.focal.infty}. 
\end{proof} 

Let $\cT$ denote for the curve defined by Eq.~\eqref{eq.lem.focal.infty}. It is of the shape of a ``teardrop" contained in $D$, see Figure \ref{figure:teardrop}, in which it has been extended to be a nodal curve inside the square $[0,\frac{\pi}{2}]^2$. 

\begin{lemma}\label{lem.CTcJ}
The curve $\cT$ is a closed curve contained in the region bounded by $\cJ$ and $\{\alpha+\beta=\frac{\pi}{2}\}$, and intersects the curve $\cJ$ tangentially at a single point $(\alpha_0, \alpha_0)$. 
\end{lemma}
\begin{proof}
It follows from the definition of the functions $F_{\cJ}$ in Eq.~\eqref{eq.green.curve} and $F_{\cT}$ in Eq.~\eqref{eq.lem.focal.infty} that the intersection of $\cJ$ and $\cT$ happens only at their intersection with the diagonal $\alpha=\beta$, which corresponds to the point $(\alpha_0, \alpha_0)$. To show that $\cT$ is closed and is contained in the region bounded by $\cJ$ and the line $\alpha+\beta=\frac{\pi}{2}$, we only need to compute the intersection of $\cT$ with the line $\alpha+\beta=\frac{\pi}{2}$. Then we get  $\alpha=\beta$, and there is exactly one solution $(\alpha,\beta)=(\frac{\pi}{4}, \frac{\pi}{4})$. 
\end{proof}

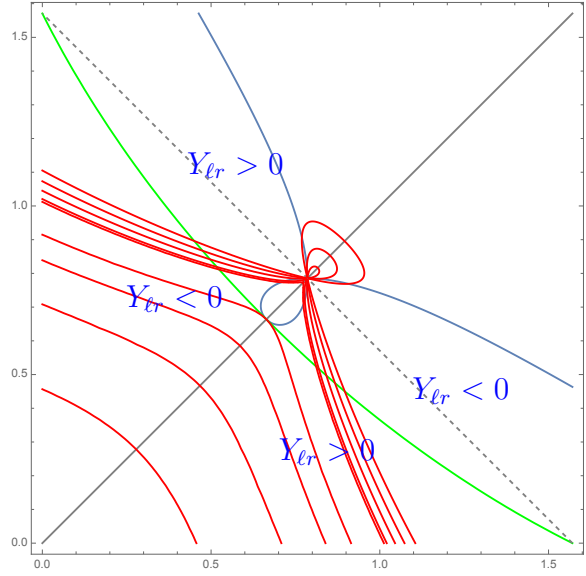
\begin{figure}[htbp]
\begin{overpic}[height=3in,unit=0.2in]{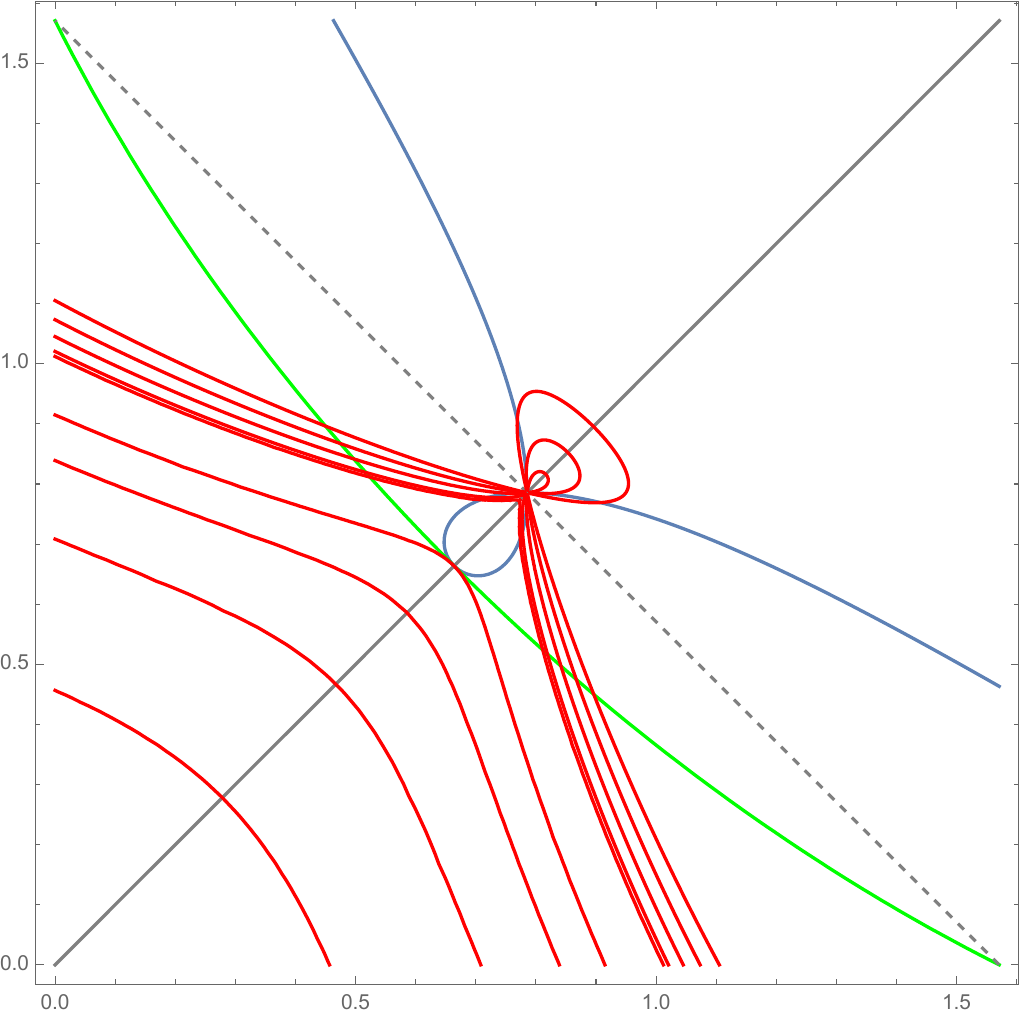}
\put (7,3) {\color{blue} $Y_{\ell r}>0$}
\put (3, 7) {\color{blue} $Y_{\ell r}<0$}
\put (4.6, 10.5) {\color{blue}$Y_{\ell r}>0$}
\put (10.5,4.6) {\color{blue}$Y_{\ell r}<0$}
\end{overpic}
\caption{The blue curve is $\cT$ which is defined by (\ref{eq.lem.focal.infty}). The green curve is $\cJ$ which is defined by (\ref{eq.green.curve}). The function $Y_{\ell r}$ vanishes on the union of $\cJ$ and the diagonal. The red curves are some level curves of $b$.
} \label{figure:teardrop}
\end{figure}

Next we will show that the level curves of the function $\fb=\fb(\alpha, \beta)$ are smooth.

\begin{lemma}\label{lem.partial.a-b}
$(\pa_{\alpha}-\pa_{\beta})\fb^2>0$ for $\alpha>\beta$ 
and $(\pa_{\alpha}-\pa_{\beta})\fb^2<0$ for $\alpha<\beta$.
\end{lemma}
\begin{proof}
By direct computation, we have
\begin{align*}
(\pa_{\alpha}-\pa_{\beta})\fb^2
&=-\frac{\sin (\alpha-\beta)}{\sin^2(\alpha+\beta)}+\frac{2 (\sin \alpha-\sin \beta) (\cos \alpha+\cos \beta)}{\sin^2 2 (\alpha+\beta)}+\frac{\cos \alpha-\cos \beta}{\sin (\alpha+\beta)} \\
&= \frac{8 \sin \frac{\alpha+\beta}{2}\sin \frac{\alpha-\beta}{2}}{\sin^2 (2\alpha+2\beta)}
\left(\sin \frac{3\alpha+3\beta}{2} \cos\frac{\alpha-\beta}{2}-\sin (\alpha+\beta) \cos ^2(\alpha+\beta)\right).
\end{align*}
Note that the term in the parentheses of the last term is strictly positive.
Then the lemma follows.
\end{proof}

\begin{corollary}
The gradient $\nabla \fb^2$ is  nowhere vanishing on the domain $D$,
and the level curves of $\fb^2$ are smooth and  tangent to the lines $\alpha+\beta=c$ exactly when crossing the diagonal $\alpha=\beta$.
\end{corollary}
\begin{proof}
By Lemma \ref{lem.partial.a-b}, only need to check if $\nabla \fb^2\neq 0$ along the diagonal $\alpha=\beta$.
Recall that $\fb^2(\beta, \beta)=(1+\frac{1}{2\cos\beta})^2$.
In particular, $(\pa_{\alpha}+\pa_{\beta})\fb^2 >0$ along the diagonal.
Collecting terms, we have $\nabla \fb^2\neq 0$ on $D$. The remaining properties follows from the characterizations of $(\pa_{\alpha}-\pa_{\beta})\fb^2$.
This completes the proof.
\end{proof}

Recall $b_{\max}=1+ 2^{-1/2}=\lim\limits_{\beta\nearrow \frac{\pi}{4}}\fb(\beta, \beta)$
and $b_{\crit}=\fb(\alpha_0, \alpha_0)$, see Lemma \ref{lem.b.crit}.
\begin{lemma}\label{lem.fb.level}
The level curve $\fb(\alpha,\beta)=b$ does not intersect the curve $\cJ$ for $1.5<b< b_{\crit}$,
and intersects  $\cJ$ exactly once in the upper wedge
$\alpha<\beta$, crosses the diagonal and then intersects $\cJ$ exactly once in the lower wedge $\alpha>\beta$ for $b_{\crit} < b < b_{\max}$.
\end{lemma}
\begin{proof}
Note that a tangent vector of a level curve of $\fb^2$ at $(0, \beta)$ is given by
$(1, -\frac{2\cos^3\beta -2\cos^2\beta+1}{1+\cos \beta})$, for which the second component
is between $-1$ and $0$. 
It follows from Lemma \ref{lem.partial.a-b}  that the minimum of $\fb(\alpha, \beta)$ can be only obtained along the diagonal, 
along which we have $\fb(\beta, \beta)= 1+\frac{1}{2\cos\beta}$.
Moreover, in the upper wedge $\alpha<\beta$, each level curve intersects
the lines $\alpha+\beta=c$ transversely. In particular, it intersects each line $\alpha+\beta=c$ at most once.
It follows that each level curve intersects $\cJ$ at most once in the upper wedge.
Then by symmetry,  a level curve 
crosses the diagonal and then intersects the curve $\cJ$ at most once in the lower wedge $\alpha>\beta$.
\end{proof}

In order to obtain a better understanding of the level curves of $\fb$, we will study $(\pa_{\alpha} +\pa_{\beta})\fb^2$. We will consider the two parts in  
$\fb^2=X_{\ell r}^2 +Y_{\ell r}^2$ separately, see \eqref{eq.OlOr}.
\begin{align}
\label{eq: partial_aplusb,X}(\partial_\alpha+\partial_{\beta})X^2_{\ell r}=& \frac{2}{\sin(2\alpha +2\beta)} \big(\cos(2\alpha +\beta) + \cos(\alpha +2\beta)\\
\nonumber&-4 \cot (2\alpha +2\beta)(\sin \alpha \cos 2\beta+\cos 2\alpha \sin \beta)-\sin2\alpha \sin \beta-\sin \alpha \sin 2\beta\big )\\
\nonumber&\Big(\cos (\alpha-\beta)+\frac{1}{\sin(2\alpha +2\beta)}(\sin \alpha \cos 2\beta+\cos 2\alpha \sin \beta)\Big), \\
(\partial_\alpha+\partial_{\beta})Y^2_{\ell r}
=&\frac{\sin^2\frac{\alpha-\beta}{2}\sin \frac{\alpha +\beta}{2}}{4\cos^2 \frac{\alpha +\beta}{2} \cos^3(\alpha +\beta)} \big(\cos (2\alpha +2\beta)-4 \cos (\alpha-\beta)-3 \cos 2\alpha-3 \cos 2\beta+1\big) \nonumber\\
& \big(\cos (2\alpha +\beta)+\cos (\alpha +2\beta)-2\sin\alpha\sin\beta+\cos \alpha+\cos \beta\big). \label{eq: partial_aplusb,Y}
\end{align}

For convenience we introduce two subsets of the domain $D$ given in Eq.~\eqref{def.domD}:
\begin{align}
D_1 &:=\{(\alpha, \beta) \in D: \alpha+\beta \ge \frac{\pi}{3}\}; \\
D_2 &:=\{(\alpha, \beta) \in D:(\cos\alpha+\cos\beta)\cos(\alpha+\beta)-\sin\alpha\sin\beta\leq 0\}.
\end{align}
Note that $D_2 \subset D_1$ since $(\cos\alpha+\cos\beta)\cos(\alpha+\beta)-\sin\alpha\sin\beta \ge  0$
for $(\alpha, \beta) \in D$ with $\alpha+\beta=\frac{\pi}{3}$.
\begin{lemma}\label{lem.partial.a+b}
(1) $(\partial_\alpha+\partial_\beta)X^2_{\ell r}> 0$ for every $(\alpha,\beta)\in D_1$.

(2) $(\partial_\alpha+\partial_\beta)Y^2_{\ell r}\geq 0$ for $(\alpha,\beta)\in D_2$.

(3) $(\partial_\alpha+\partial_\beta)\fb^2 \geq 0$ for $(\alpha,\beta)\in D_2$.
\end{lemma}

\begin{proof}
(1) The expression of $(\pa_{\alpha}+\pa_{\beta})X_{\ell r}^2$ given in  \eqref{eq: partial_aplusb,X}
can be rewritten as
\begin{align*}
(\pa_{\alpha}+\pa_{\beta})X_{\ell r}^2=&
\frac{2X_{\ell r}}{\sin(2\alpha+2\beta)}(\cP\cX_1 -4 \cot (2\alpha+2\beta)\cP\cX_2+\cP\cX_3),
\end{align*}
where
\begin{align*}
&\cP\cX_1=\cos(2\alpha+\beta)+\cos(\alpha+2\beta);\\
&\cP\cX_2= \sin \alpha \cos 2\beta+\cos 2\alpha \sin\beta;\\
&\cP\cX_3=-\sin2\alpha \sin \beta-\sin\alpha \sin 2\beta.
\end{align*} 
Note that $X_{\ell r}>0$ and $\sin(2\alpha+2\beta)>0$ for any $(\alpha, \beta) \in D$.
To prove $(\partial_\alpha+\partial_{\beta})X^2_{\ell r}\geq 0$, it suffices to show that 
$\cP\cX_1-4 \cot (2\alpha+2\beta)\cP\cX_2+\cP\cX_3\geq 0$. 

Note that for any $(\alpha, \beta) \in D_1$,
\begin{align*}
(\partial_{\alpha}-\partial_\beta)\cP\cX_1=&\sin (\alpha+2\beta)-\sin (2\alpha+\beta)=2\cos\frac{3(\alpha+\beta)}{2}\sin \frac{\beta-\alpha}{2};\\
(\partial_{\alpha}-\partial_\beta)\cP\cX_2=&-\cos 2\alpha \cos \beta+\cos \alpha \cos 2\beta
+4 \sin \alpha \sin \beta (\cos \beta-\cos \alpha)\\
=&(\cos \beta-\cos \alpha)(2\cos \beta \cos \alpha+1+4\sin \alpha\sin \beta)\\
=&(\cos \beta-\cos \alpha)(1+3\cos(\alpha-\beta)- \cos(\alpha+\beta));\\
(\partial_{\alpha}-\partial_\beta)\cP\cX_3
=&2\big( \sin \alpha \cos 2\beta
-\cos 2\alpha \sin \beta\big)+\big(\sin 2\alpha \cos \beta-\cos \alpha \sin 2\beta\big)\\
=&2(\sin \alpha-\sin \beta)(3+4\sin \alpha\sin \beta).
\end{align*}
It follows that, for each $j=1,2,3$, $(\partial_{\alpha}-\partial_\beta)\cP\cX_j<0$ when $\alpha<\beta$
and $>0$ when $\alpha>\beta$.
Therefore, the minimum of $\cP\cX_j$ along the line $\alpha+\beta=\text{cst}$ is achieved at $\alpha=\beta$. 
Since $\cot (2\alpha+2\beta)$ is a negative constant along the line $\alpha+\beta=\text{cst}$,
 the minimum of  $\cP\cX_1-4 \cot (2\alpha+2\beta)\cP\cX_2+\cP\cX_3$ is achieved at $\alpha=\beta$, at where we have 
\begin{align*}
(\cP\cX_1- 4\cot(2\alpha+2\beta)\cP\cX_2+\cP\cX_3)|_{\alpha=\beta}
=2 \cos 2\beta \sin \beta \tan \beta>0.
\end{align*}
This completes the proof of (1).

(2). Let $(\alpha, \beta) \in D_2$, and 
\begin{align*}
\cP\cY_1=&-4 \cos (\alpha-\beta)+\cos (2\alpha+2\beta)-3 \cos2\alpha-3 \cos2\beta+1,\\
\cP\cY_2=&\cos (2\alpha+\beta)+\cos (\alpha+2\beta)-2\sin\alpha\sin\beta+\cos \alpha+\cos\beta \\
=&2(\cos\alpha+\cos\beta)\cos(\alpha+\beta)- 2\sin\alpha\sin\beta. 
\end{align*}
Note that $\cP\cY_2\le0$ on $D_2$ by the our choice of the domain $D_2$.
From the expression (\ref{eq: partial_aplusb,Y}), 
we have $(\partial_\alpha+\partial_{\beta})Y^2_{\ell r}\geq 0$ if and only if $\cP\cY_1\cP\cY_2\geq 0$.
Moreover, $(\partial_\alpha+\partial_{\beta})Y^2_{\ell r}= 0$  if and only if either $\alpha=\beta$ or $\cP\cY_1\cP\cY_2= 0$. Note that 
\begin{align*}
(\partial_{\alpha}-\partial_{\beta})\cP\cY_1=&8 \sin (\alpha-\beta)+6 \sin 2\alpha -6 \sin2\beta \\
=&\sin(\alpha-\beta)(8+12\cos(\alpha+\beta)).
\end{align*}
For $\alpha<\beta$ (resp. $\alpha>\beta$) in the the region $D_2$, 
$(\partial_{\alpha}-\partial_{\beta})\cP\cY_1<0$ (resp.  $(\partial_{\alpha}-\partial_{\beta})\cP\cY_1>0$), hence for fixed $\alpha+\beta$, the maximum is obtained when $\alpha=0$ (or $\beta=0$):
\begin{align*}
\cP\cY_1(0, \beta)=-4\cos \beta -2\cos 2\beta -2<0.
\end{align*}
Therefore $\cP\cY_1<0$ on $D_2$.

(3) It follows directly from (1) and (2).
\end{proof}

It follows from the above proof that for $(\alpha,\beta)\in D_2$, $(\partial_\alpha+\partial_\beta)Y^2_{\ell r}=0$
if and only if $Y_{\ell r}=0$, which is equivalent to either \eqref{eq.green.curve} or $\alpha=\beta$.

\section{Generalized billiard maps}\label{sec.gen.b.map}

Let $\cQ(b)$ be the lemon table with $1.5<b<2$, $x=(p(x),v(x)) \in M_b$ be a point in the phase space of the lemon billiards, where
$p(x) \in \pa \cQ(b)$ is the base point of $x$ and $v(x)$ is the direction of the trajectory of $x$.
Let $L(x)=\{p(x)+t v(x): t\in \bR\}$ be the oriented line corresponding to the point $x$.
Let $d_{\ell}(x)$ and $d_{r}(x)$ be the oriented distances from the centers $O_{\ell}$
and $O_r$ to the line $L(x)$, respectively.
After one reflection on the left arc, the quantity $d_{r}':=d_{r}(F_b x)$ does not change,
while $d_{\ell}' :=d_{\ell}(F_b x)$ changes. The case with reflections on the right arc is similar.
Now we find an explicit formula for $(d_{\ell}', d_{r}')$ in terms of $(d_{\ell}, d_{r})$.

\begin{figure}[htbp]
\includegraphics[scale=0.7]{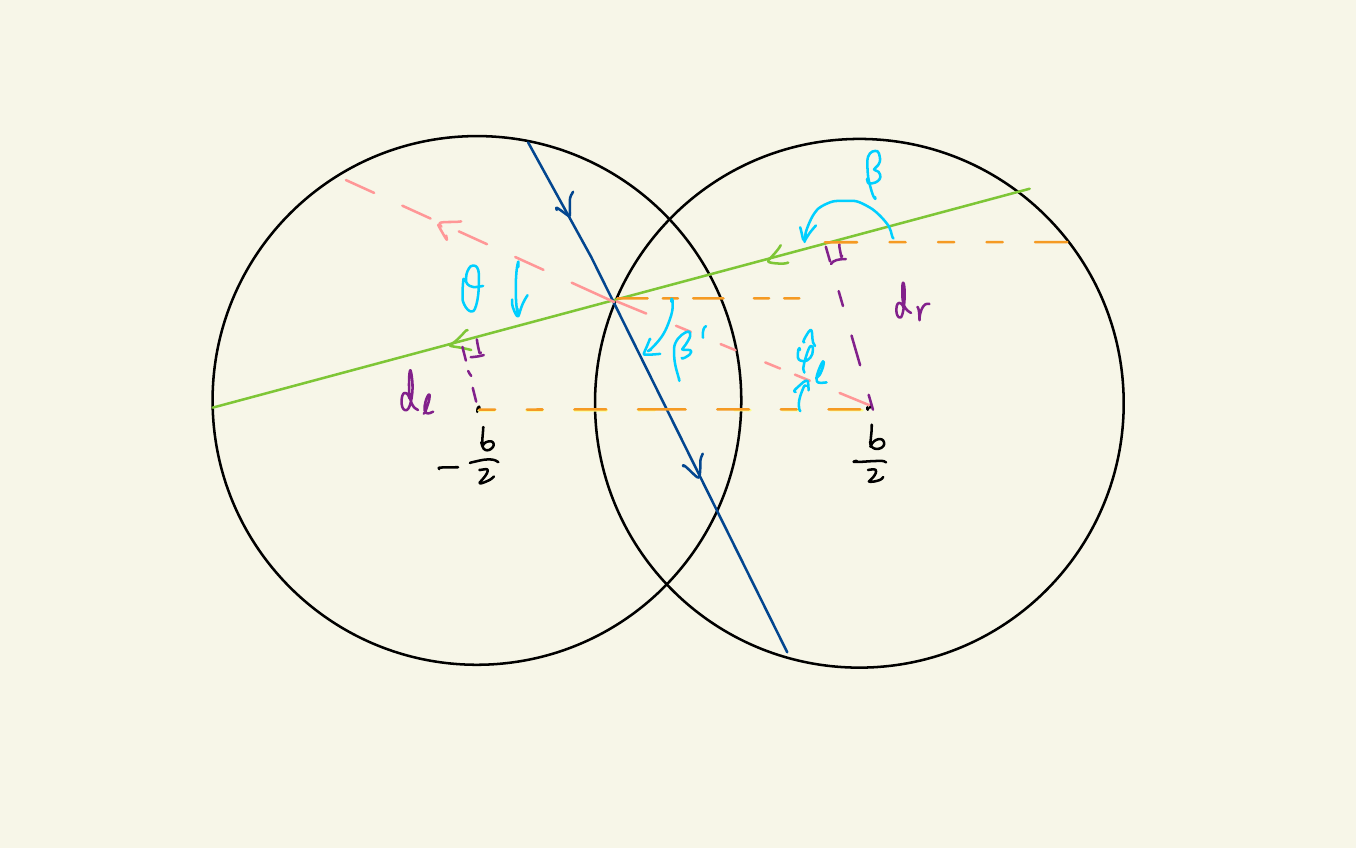}
\caption{A trajectory on the lemon table $\cQ(b)$.}\label{def.dldr}
\end{figure}

For certainty, consider a trajectory starting on the right arc and reflecting on the left arc, see Fig.~\ref{def.dldr} for an illustration.
Let $\varphi_{\ell}$ be the reflection position on the left arc, $\hat\varphi_{\ell}= \pi -\varphi_{\ell}$, $\theta$ be the angle from the outer normal direction at $p(\varphi_{\ell})$ to the direction of the trajectory. Then 
\begin{enumerate}
\item the angle $\beta$ from the positive horizontal direction to the direction of the 
trajectory is given by $\beta=\varphi_{\ell}+\theta$;

\item the normal direction of the trajectory is 
$(\cos(\beta+\frac{\pi}{2}), \sin(\beta+\frac{\pi}{2}))$;
  
\item the angle $\beta'$ from the positive horizontal direction to the direction of the 
reflected trajectory is given by $\beta'=(\varphi_{\ell}-\pi)-\theta$;

\item the normal direction of the reflected trajectory is 
$(\cos(\beta'+\frac{\pi}{2}), \sin(\beta'+\frac{\pi}{2}))$. 
\end{enumerate}

It follows that
\begin{align}
d_r=&(-\cos\hat{\varphi}_\ell, \sin\hat{\varphi}_\ell)\cdot (\cos(\frac{\pi}{2}+\beta), \sin(\frac{\pi}{2}+\beta)); \label{def.dl}\\
d_\ell=&(b-\cos\hat{\varphi}_\ell, \sin\hat{\varphi}_\ell)\cdot (\cos(\frac{\pi}{2}+\beta), \sin(\frac{\pi}{2}+\beta)). \label{def.dr}
\end{align}
Note that \eqref{def.dl} is consistent with $d_r=-\sin\theta$ and is handy for expressing  the billiard
map in terms of $(d_{\ell},d_r)$.
Also observe that $\sin\beta=\frac{d_r-d_\ell }{b}$.
After reflecting on the left arc, $d_{r}$ remains the same, 
$\beta'=\beta-2\theta-\pi$ and 
\begin{align}
d_\ell'&=(b-\cos\hat{\varphi}_\ell, \sin\hat{\varphi}_\ell)\cdot (\cos(\frac{\pi}{2}+\beta'), \sin(\frac{\pi}{2}+\beta')) \nonumber\\
&=-b\sin\beta'+\sin(\hat{\varphi}_\ell+\beta')=b\sin(\beta-2\theta)-\sin\theta \nonumber\\
&=b(\sin\beta(2\cos^2\theta-1)-2\cos\beta\sin\theta\cos\theta)-\sin\theta \nonumber\\
&=b(\frac{d_r-d_\ell}{b}(2(1-d_r^2)-1)-2\sqrt{1-(\frac{d_r-d_\ell}{b})^2}d_r\sqrt{1-d_r^2})+d_r \nonumber\\
&=(d_r-d_\ell)(1-2d_r^2)-2d_r\sqrt{(b^2-(d_r-d_\ell)^2)(1-d_r^2)}+d_r. \label{eq:dell'}
\end{align}
In other words, the reflection on the arc $\Gamma_{\ell}$ can be written as
$\cL_{b}(d_{\ell},d_{r})=(d_{\ell}', d_r)$, where $d_{\ell}'$ is given by \eqref{eq:dell'}.
The formula for the reflection on the arc $\Gamma_{r}$ can be derived in a similar way. 
To emphasize the difference of  reflections on the left arc and the right arc, we denote it by $\cR_{b}$. 
Alternatively, we can use the symmetry of the lemon table:  the two maps $\cL_{b}$ and $\cR_{b}$ commute with the reflection $\cI: (d_{\ell}, d_r) \mapsto (d_r, d_{\ell})$ with respect to the diagonal:
\begin{align}
\cI\circ \cR_{b} = \cL_{b}\circ \cI,  \label{eq.lr.sym}
\end{align}
It follows that $\cR_{b}(d_\ell, d_r)=(d_\ell, d_r')$, where
\begin{align}
d_r'=(d_\ell-d_r)(1-2d_\ell^2)-2d_\ell\sqrt{(b^2-(d_\ell-d_r)^2)(1-d_\ell^2)}+d_\ell. 
\label{eq:dr'}
\end{align}

Now we specify the domains of the two reflection maps $\cL_{b}$ and $\cR_{b}$
by identifying the points at where they are not defined and/or not differentiable.
We will start with the map $\cL_{b}$. 
Note that $|d_r| <1$ since $\cL_{b}$ represents a reflection on the left arc.
Moreover, the branched locus of the map $\cL_{b}$ consists of two segments $L_{b}$ and $L_{-b}$, where
\begin{align}
L_{\pm b}=\{(d_{\ell}, d_r) \in \bR^2: d_{r} - d_{\ell} = \pm b, |d_r|<1\}. \label{eq.Lb}
\end{align}
Each point in $L_{\pm b}$ corresponds to a vertical oriented line whose distance to the center $O_r$ is less than 1.
From now on, we will restrict our discussion to the parallelogram
\begin{align}
P(\cL_b) = \{(d_{\ell}, d_r) \in \bR^2: |d_{r} - d_{\ell}| < b, |d_r|<1\}.
\end{align} 
It is worth pointing out that both factors within the square root in Eq.~\eqref{eq:dell'} are positive on the parallelogram $P(\cL_{b})$.
To find the domain on which the map $\cL_{b}$ is smooth, we note that 
\begin{align}
\label{eq: partial d_l', d_l}&\frac{\partial d_\ell'}{\partial d_\ell}=-(1-2d_r^2)+\frac{2d_r(1-d_r^2)(d_\ell-d_r)}{\sqrt{(b^2-(d_r-d_\ell)^2)(1-d_r^2)}}.
\end{align}
It follows that $\frac{\partial d_{\ell}'}{\partial d_{\ell}}=0$ when
\begin{align*}
&(1-2d_r^2)^2(b^2-(d_\ell-d_r)^2)=4d_r^2(1-d_r^2)(d_r-d_\ell)^2,
\end{align*}
which is equivalent to $(d_r-d_\ell)^2=b^2(1-2d_r^2)^2$, or equally,
$d_{\ell}=\pm b(1-2d_r^2)+d_r$.
Moreover,  plugging in $d_{\ell}=b(1-2d_{r}^2)+d_r$ into the right-hand-side of \eqref{eq: partial d_l', d_l}, we have $d_r > 0$. Similar,  plugging in $d_{\ell}=-b(1-2d_{r}^2)+d_r$ into the right-hand-side of \eqref{eq: partial d_l', d_l}, we have $d_r<0$.
So the map $\cL_{b}$ is singular at the following curves:
\begin{align}
\cS_{1,\pm}(b)& :=\{(d_{\ell}, d_r)\in P(\cL_{b}): \pm d_r >0, d_\ell-d_r=\pm b(1-2d_r^2)\}.
\label{eq:C_1.pm}
\end{align}
Note that $L_{b} \cap \cS_{1,+}(b) =\{(1-b, 1)\}$, $L_{b} \cap \cS_{1,-}(b) =\{(-b,0)\}$,
$L_{-b} \cap \cS_{1,+}(b) =\{(b, 0)\}$ and $L_{-b} \cap \cS_{1,-}(b) =\{(b-1,-1)\}$, 
see Fig.~\ref{fig.sing}. The geometric meaning of these singularity curves 
$\cS_{1, \pm}(b)$ will be clear after Proposition~\ref{pro.involution}.

Similarly, 
we consider the parallelogram
$P(\cR_b) = \{(d_{\ell}, d_r) \in \bR^2: |d_{r} - d_{\ell}| < b, |d_{\ell}|<1\}$ and find
\begin{align}
\label{eq: partial d_r', d_r}&\frac{\partial d_r'}{\partial d_r}=-(1-2d_\ell^2)+\frac{2d_\ell(1-d_\ell^2)(d_r-d_\ell)}{\sqrt{(b^2-(d_\ell-d_r)^2)(1-d_\ell^2)}}.
\end{align}
It follows that the map $\cR_{b}$ is singular at the following curves:
\begin{align}
\cS_{2,\pm}(b)&=\{(d_{\ell}, d_r)\in P(\cR_b): \pm d_\ell>0, d_r-d_\ell=\pm b(1-2d_\ell^2)\}.\label{eq:C_2.pm}
\end{align}

\begin{figure}[htbp]
\begin{overpic}[height=6in, unit=0.1in]{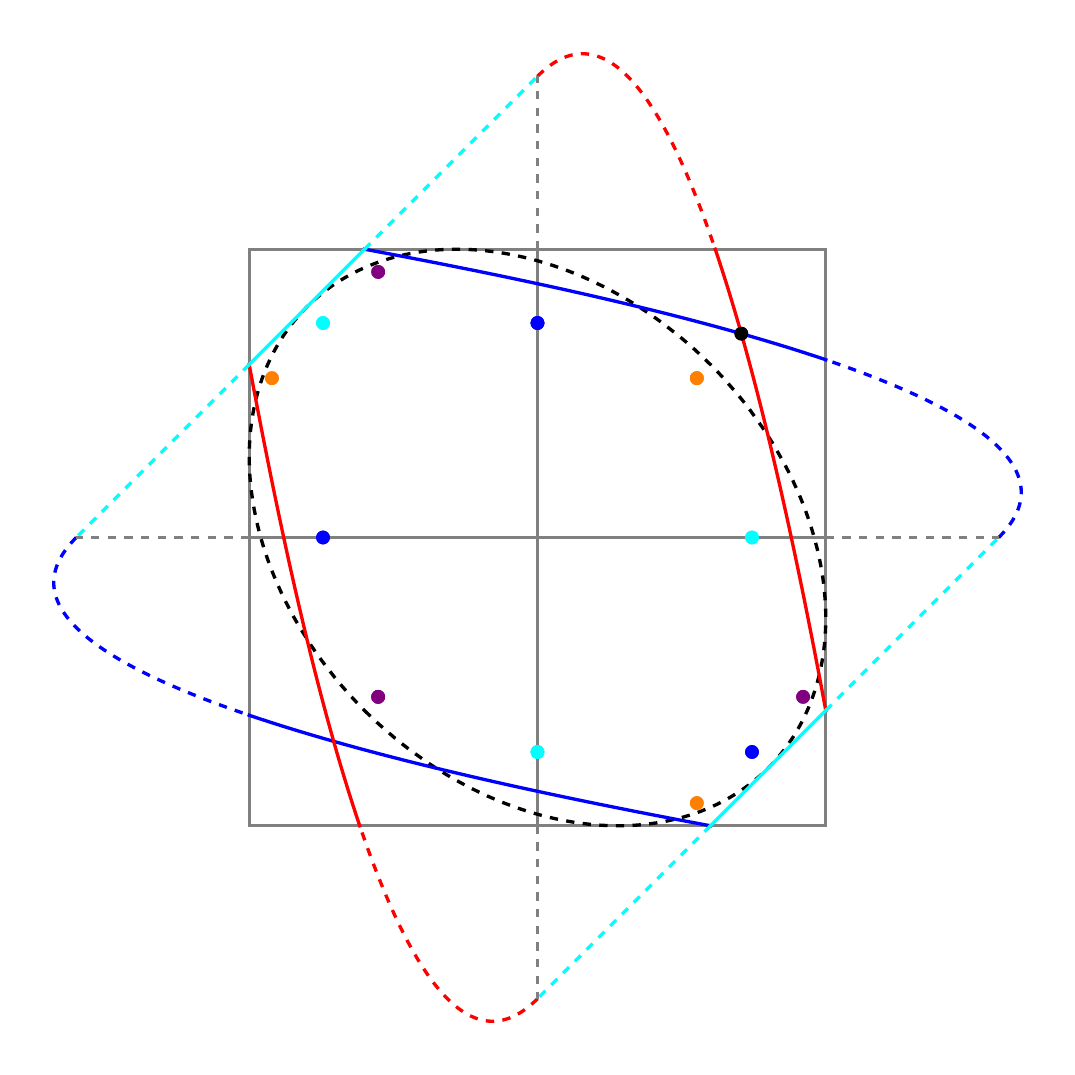}
\put (42, 42) {$S_0$} %
\put (37, 37) {$E_0$} %
\put (27.5, 41.5) {$P_0$} %
\put (41.5, 28) {$Q_0$} %
\put (18, 40) {$Q_1$} %
\put (17, 31) {$P_1$} %
\put (27.5, 18) {$Q_2$} %
\put (39.5, 18.5) {$P_2$} %
\put (15, 43) {$L_b$} %
\put (43, 15) {$L_{-b}$} %
\put (36, 43) {$\cS_{1,+}$} %
\put (22, 16) {$\cS_{1,-}$} %
\put (43, 36) {$\cS_{2,+}$} %
\put (15, 22) {$\cS_{2,-}$} %
\put (10.5, 30.5) {$\text{Cor}_b$} %
\end{overpic}
\caption{Plots of the singularity curves $L_{\pm b}$ (cyan), $\cS_{1,\pm}(b)$ (blue), $\cS_{2,\pm}(b)$ (red) for the generalized billiard maps $(\cL_{b}, \cR_{b})$.  The curve $\text{Cor}_{b}$ (dashed black) is added just  for comparison. Here $b=1.6$.}\label{fig.sing}
\end{figure}
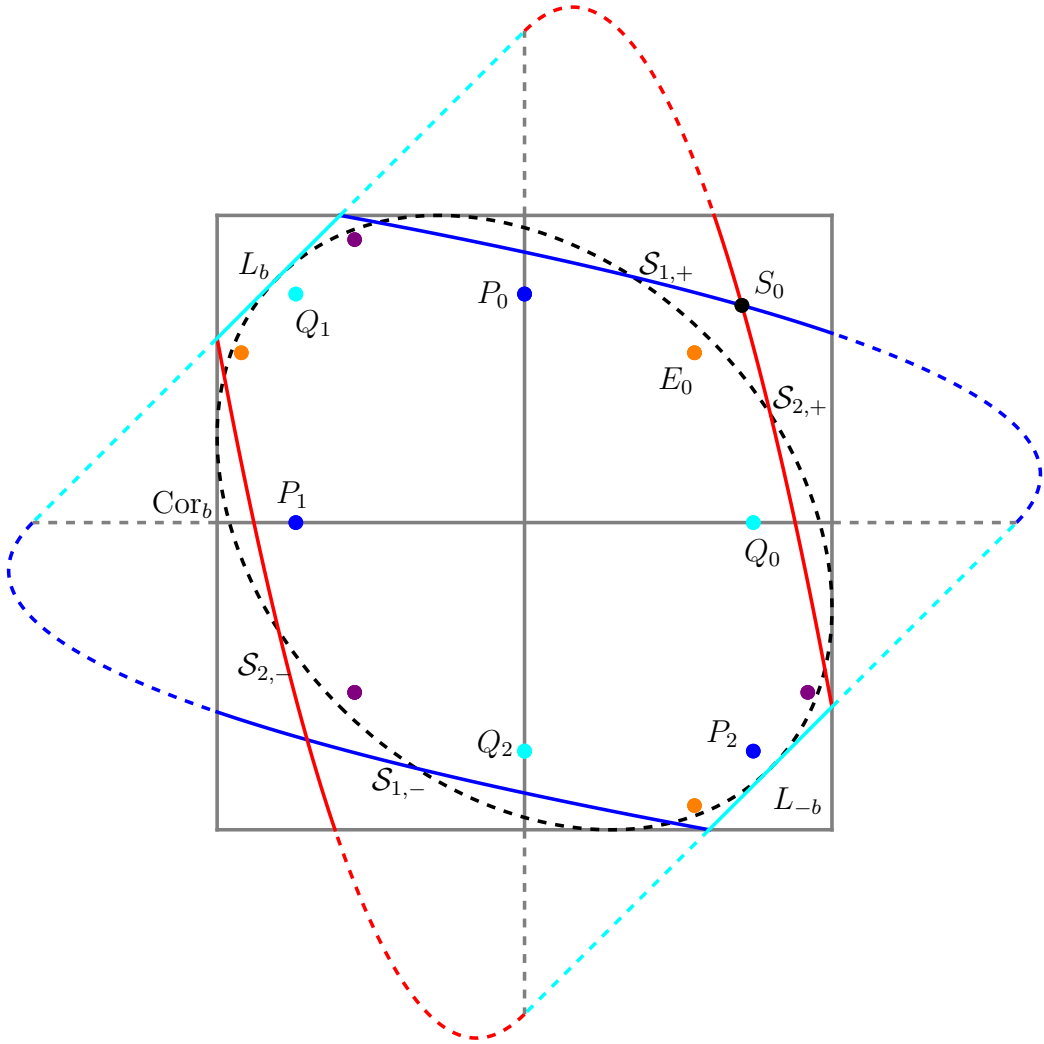

Note that the two curves $\cS_{j, +}(b)$ and $\cS_{j,-}(b)$ are center-symmetric, $j=1,2$,
and $\cI(\cS_{1,\pm}(b)) = \cS_{2,\pm}(b)$. Moreover,
$\cS_{1,+}(b)\cap \cS_{2,+}(b) =\{S_0\}$, where $S_0=(2^{-1/2}, 2^{-1/2})$  is independent of the parameter $b$.

\begin{definition}\label{def.domain}
Let $\cD(\cL_b) \subset P(\cL_{b})$ be the domain bounded between the two curves $\cS_{1,+}(b)$ and $\cS_{1,-}(b)$, and $\cD(\cR_b)\subset P(\cR_{b})$ be the domain bounded between the two curves $\cS_{2,+}(b)$ and $\cS_{2,-}(b)$. 
Let $\cD_b = \cD(\cL_b)\cap \cD(\cR_b)$.
\end{definition}
It follows that $\cL_{b}$ is well defined and a local diffeomorphism on $D(\cL_{b})$,
so is  $\cR_{b}$ on $D(\cL_{b})$. 
For convenience, we introduce the ambient phase space $\cM=[-1, 1] \times [-1,1]$,
which consists   oriented lines that intersect both circles.
It is clear that $\cD_{b} \subset \cM$, see Fig.~\ref{fig.sing}.

\begin{remark}\label{rem.abs.bill}
The two maps $\cL_{b}$ and $\cR_{b}$ are induced from the billiard map $F_b$ on a lemon table $\cQ(b)$. However, the connection between the maps is subtle.
For example, the restricted domain $\cD_{b}$ is different from the admissible phase space of the billiard map. Moreover, the composition $\cR_{b}\circ \cL_{b}$, when exist, is \emph{not} always the same as the composition of the billiard map. This is due to the fact that $\cR_{b}\circ \cL_{b}$ is the map that reflects on the left arc and right arc of the table in an alternating manner (even if the reflection point is outside the lemon table) while the billiard map can have multiple consecutive reflections on one side of the table. It is important to note that they do agree if the corresponding billiard trajectory reflects alternatively on the two arcs $\Gamma_{\ell}$ and $\Gamma_r$ of $\pa \cQ(b)$. 
More precisely, let $\text{Cor}_{b} \subset \cM$ be the set of oriented lines intersecting one of the two corners $A$ or $B$ on the billiard table $\cQ(b)$,  see the dashed black curve in Fig.~\ref{fig.sing} and Fig.~\ref{fig.phase156}. Then the curve $\text{Cor}_{b} \subset \cM$ separates the oriented lines intersecting both arcs of $\pa \cQ(b)$ from those intersecting only one arc of  $\pa \cQ(b)$.  
\end{remark}

\subsection{Basic properties of the generalized maps}
Below we list some basic properties of the maps $\cL_{b}$ and $\cR_{b}$.
See Definition~\ref{def.domain} for the domains of these two maps.
\begin{proposition}\label{pro.involution}
$\cL_{b}\circ \cL_{b}=Id$ on its domain $D(\cL_{b})$ and $\cR_{b}\circ \cR_{b}=Id$ on $D(\cR_{b})$. Moreover,
$\cS_{1,\pm }(b) \subset \cL_{b}(L_{\pm b})$, $\cL_{b}(\cS_{1,\pm }(b)) \subset  L_{\pm b}$,
$\cS_{2,\pm }(b) \subset \cR_{b}(L_{\pm b})$, and  $\cR_{b}(\cS_{2,\pm }(b)) \subset L_{\mp b}$.
\end{proposition}
\begin{proof}
Given a point $(d_{\ell}, d_{r}) \in D(\cL_{b})$, $\cL_{b}\circ \cL_{b}(d_{\ell}, d_{r})$
just means that the second reflection of the orient line is at the same arc and undoes the first reflection.
So it recovers the initial point $(d_{\ell}, d_{r})$. The reasoning for $\cR_{b}\circ \cR_{b}=Id$ on $D(\cR_{b})$
is the same.

Note that $\cL_{b}$ preserves the $d_r$ coordinate.
Then  $\cS_{1, +}(b) \subset \cL_{b}(L_{b})$ follows directly from Eq.~\eqref{eq:dell'}, while $\cL_{b}(\cS_{1, +}(b)) \subset  L_{b}$ follows by plugging
the expression $d_{\ell} = b(1-2d_r^2) + d_r$ in Eq.~\eqref{eq:dell'}.
The verification of the remaining equalities are similar and hence omitted.
\end{proof}

Collecting terms, we conclude that $\cL_{b}: D(\cL_b) \to D(\cL_b)$ is a diffeomorphism,
so is $\cR_{b}: D(\cR_{b}) \to D(\cR_{b})$.

The symmetries of the lemon table lead to additional symmetries of the generalized maps $\cL_{b}$ and $\cR_{b}$ besides Eq.~\eqref{eq.lr.sym}: 
\begin{enumerate}
\item $F_b$ commutes with the reflection of the table $\cQ(b)$ about the $x$-axis,
which implies that $\cL_{b}$ and $\cR_{b}$ commute with the center symmetry $g:(d_{\ell}, d_r) \mapsto (-d_{\ell}, -d_r)$. That is, $\cL_{b}\circ g = g\circ \cL_{b}$ and $\cR_{b}\circ g = g\circ \cR_{b}$;

\item $F_b$ commutes with the reflection of the table $\cQ(b)$ about the axis through the two corners of $\cQ(b)$,
which implies that $\cL_{b}$ and $\cR_{b}$ commute with $h: (d_{\ell}, d_r) \mapsto (-d_r, -d_{\ell})$ in the sense that $\cL_{b}\circ h = h\circ \cR_{b}$.
\end{enumerate}

Note that a point $(d_{\ell}, d_r) \in \cM$ is fixed by $\cL_{b}$ if and only only if 
\begin{align*}
d_{\ell}=(d_r-d_\ell)(1-2d_r^2)-2d_r\sqrt{(b^2-(d_r-d_\ell)^2)(1-d_r^2)}+d_r,
\end{align*}
or equally, $d_{\ell}=(1-b)d_r$. Such a point corresponds to a billiard trajectory
intersecting the arc $\Gamma_{\ell}$ at the point $B_{\ell}=(b-1,0)$ on the table $\cQ(b)$.
Similarly,  a point $(d_{\ell}, d_r) \in \cM$ is fixed by $\cR_{b}$ if and only if 
$d_{r}=(1-b)d_{\ell}$, which corresponds to a billiard trajectory
intersecting the arc $\Gamma_{r}$ at the point $B_r=(1,0)$ on the table $\cQ(b)$.

\begin{remark}\label{rem.period6}
It is easy to see that in the new phase space $\cM$, there are 
\begin{enumerate}
\item two elliptic periodic orbits of period $6$: one is given by 
\begin{enumerate}
\item $E_0=\big(\big(1-\frac{1}{4(b-1)^2}\big)^{1/2}, \big(1-\frac{1}{4(b-1)^2}\big)^{1/2}\big)$, 

\item $E_1=\cL_{b}(E_0)=(\frac{1}{1-b}\big(1-\frac{1}{4(b-1)^2}\big)^{1/2}, \big(1-\frac{1}{4(b-1)^2}\big)^{1/2}) \in \Fix(\cR_{b})$;

\item $E_2=\cR_{b}(E_1)=E_1$ since $E_1\in \Fix(\cR_{b})$;

\item $E_3=\cL_{b}(E_2)=E_0$ back to $E_0$;

\item $E_4=\cR_{b}(E_0)=(\big(1-\frac{1}{4(b-1)^2}\big)^{1/2}, \frac{1}{1-b}\big(1-\frac{1}{4(b-1)^2}\big)^{1/2}) \in \Fix(\cL_{b})$;

\item $E_5=\cL_{b}(E_4) =E_4$ since $E_4\in \Fix(\cL_{b})$; 
\end{enumerate}
the other one is the center-symmetric image of the first one.
They correspond to the periodic points shown in Fig.~\ref{fig.ell.po}.
See Fig.~\ref{fig.sing} and Fig.~\ref{fig.phase156};

\item one hyperbolic periodic points of period $6$: 
\begin{enumerate} 
\item $P_0=\big(0,\big(1-\frac{1}{4 b^2-8}\big)^{1/2}\big)$, 

\item $Q_1=\cL_{b}(P_0)=(-\big(1-\frac{1}{4 b^2-8}\big)^{1/2}, \big(1-\frac{1}{4 b^2-8}\big)^{1/2})$, 

\item $P_1=\cR_{b}(Q_1)=(-\big(1-\frac{1}{4 b^2-8}\big)^{1/2},0)$,

\item $Q_0=\cL_{b}(P_1)=\big(\big(1-\frac{1}{4 b^2-8}\big)^{1/2}, 0\big)$,

\item $P_2=\cR_{b}(Q_0)=(\big(1-\frac{1}{4 b^2-8}\big)^{1/2}, -\big(1-\frac{1}{4 b^2-8}\big)^{1/2})$,

\item $Q_2=\cL_{b}(P_2)=(0, -\big(1-\frac{1}{4 b^2-8}\big)^{1/2})$.
\end{enumerate}
It correspond to the periodic points shown in Fig.~\ref{fig.hyp.po}.  The indices of the hyperbolic periodic orbit may looks
strange. It is because that we will mainly use the two points $P_0$ and $Q_0$ that are contained in the first quadrant,
and the above indexing is compatible in the sense $P_j=(\cR_{b}\cL_{b})^j P_0$ and $Q_j=(\cR_{b}\cL_{b})^j Q_0$, $j \in \bZ/3\bZ$. See  Fig.~\ref{fig.sing} and Fig.~\ref{fig.phase156}.
\end{enumerate}
\end{remark}

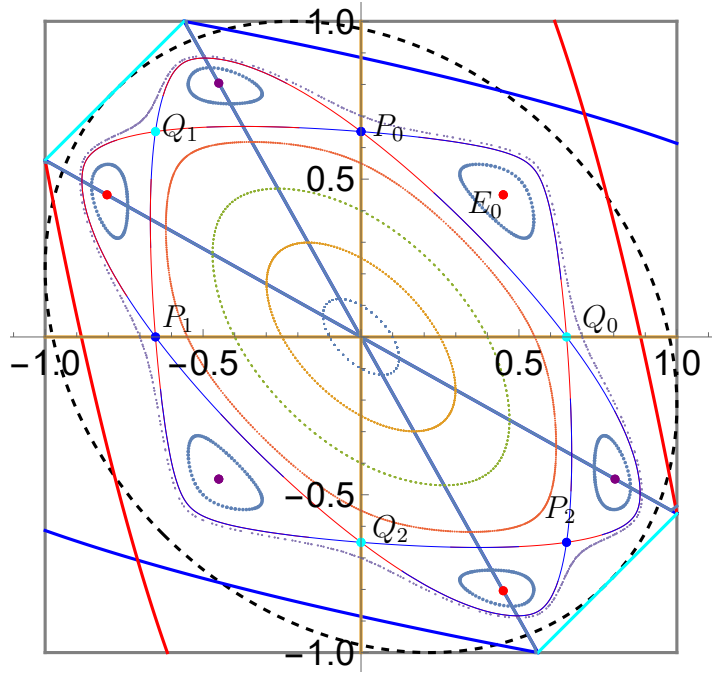
\begin{figure}[htbp]
\begin{overpic}[height=3.5in, unit=0.1in]{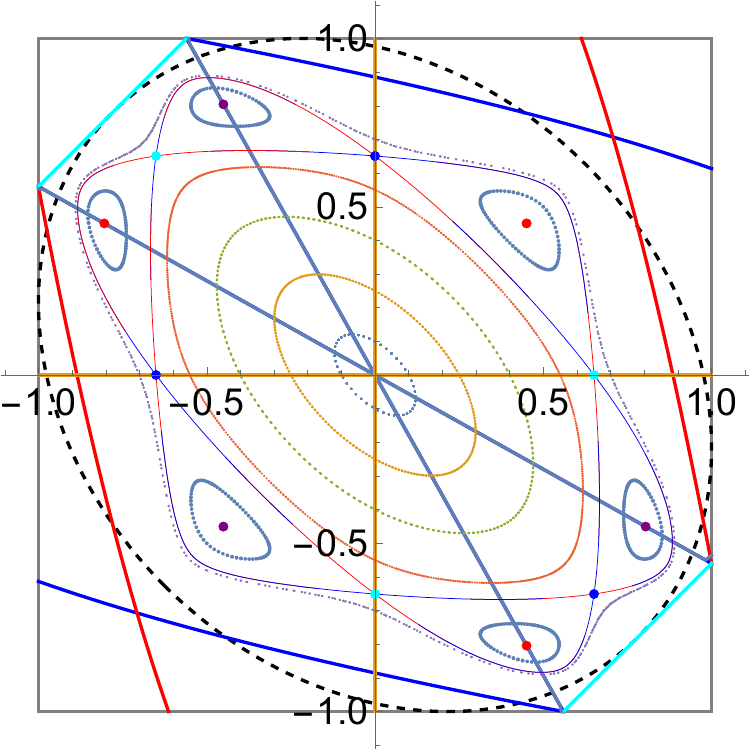}
\put (19, 28) {$P_0$} %
\put (24, 24) {$E_0$} %
\put (30, 18) {$Q_0$} %
\put (8, 18) {$P_1$} %
\put (8, 28) {$Q_1$} %
\put (28, 8) {$P_2$} %
\put (19, 7) {$Q_2$} %
\end{overpic}
\caption{Phase portrait of the generalized billiard maps $(\cL_{b}, \cR_{b})$ for $b=1.56$.
The points in blue and cyan form a hyperbolic periodic orbit of period $6$, and the points in red and purple are two elliptic periodic orbits of period $6$.
The two two lines going through the origin are $\text{Fix}(\cL_{b})$ (more steep) and $\text{Fix}(\cR_{b})$ (less steep).} \label{fig.phase156}
\end{figure}

From now on we will restrict the domains of two maps  $\cL_{b}$ and $\cR_{b}$ to their common domain $\cD_{b}$ (see Definition \ref{def.domain}) and it is safe to treat them as undefined outside $\cD_{b}$ (note that $\cD_{b}$ is not invariant for any of the two maps). 
It follows from Proposition \ref{pro.involution} that the restrictions 
$\cL_{b}|_{\cD_{b}}$ and $\cR_{b}|_{\cD_{b}}$  are diffeomorphisms from the common domain $\cD_{b}$ onto their corresponding images. 
For convenience, we introduce short notations for some frequently used combinations:
\begin{align}
\Phi_{b}:=\cL_{b} \cR_{b}\cL_{b}, \quad \Psi_{b}:=\cR_{b}\cL_{b} \cR_{b},
\quad \Theta_b:=(\cR_{b}\cL_{b})^3=\Psi_b\Phi_b. \label{def.phipsif}
\end{align}

\begin{remark}\label{rem.res.rl}
The domains of the compositions and iterations ($\Phi_b$, $\Psi_b$ and $\Theta_b$ for example), admittedly complicated, are open subsets of $\cD_{b}$.
We want to emphasize that we are interested in the subsets contained in $\cM$ that are dynamically invariant under both $\cL_{b}$ and $\cR_{b}$, on which both are diffeomorphisms.  
\end{remark}

Note that a point $(d_{\ell}, d_r) \in \cM$ is fixed by $\Phi_{b}$ if and only only if 
$\cR_{b}\cL_{b}(d_{\ell}, d_r) =\cL_{b} (d_{\ell}, d_r)$. That is, $\cL_{b}(d_{\ell}, d_r) \in \Fix(\cR_{b})$.
It follows that $\Fix(\Phi_{b})$ is the graph of the following function
\begin{align}
d_{\ell}=\frac{d_{r}}{b-1}\big(-1+2 b (1-d_{r}^2)-2b\sqrt{((1-b)^2 -d_r^2)(1-d_{r}^2) } \big), 
1-b< d_{r} <b-1.
\end{align}
Note that $\Fix(\Phi_{b})$ intersects the diagonal $d_{\ell}=d_{r}$ at exactly three points: the origin $(0,0)$, and elliptic points $\pm E_0(b)$. 
Similarly, we have $\Fix(\Psi_{b})=\cR_{b}(\Fix(\cL_{b}))$, which is the reflection of $\Fix(\Phi_{b})$
about the diagonal and intersects the diagonal at the same points as $\Fix(\Phi_{b})$.

\begin{remark}\label{rem.existenceIM}
We give a brief description on the construction and properties of the stable and unstable manifolds
of the hyperbolic periodic orbit $\{(P_j, Q_j): j =0 ,1, 2\}$.
The local stable manifolds at these points  exist, are analytic and depend analytically on the parameter $b$ since these are local properties, see \cite[\S 6.2]{KH95}. Generally speaking,  global stable manifolds are immersed curves and are obtained by taking the corresponding preimages under the map. In our case, the domain $\cD_{b}$ on which
the maps $\cL_{b}$ and $\cR_{b}$ are defined is not invariant.
Therefore, the construction of global stable manifolds can unfold in two distinct scenarios:
\begin{enumerate}
\item It can be continued infinitely, as observed when the parameter $b$ remains small, yielding analytic immersed curves.

\item Alternatively, it may be punctuated by singularities after a finite number of preimage steps, producing piecewise analytic curves in the case of large $b$.
\end{enumerate}
In the latter scenario, our attention is directed toward the connected components encompassing  hyperbolic periodic points. These components can be perceived as maximally extended local stable manifolds.
We claim that at least one of these maximally extended local stable manifolds extends to intersect the boundary $\pa \cD_b$. 
\end{remark}
\begin{proof}[Proof of the claim]
To prove this claim, we will argue by contradiction: suppose that none of the maximally extended stable manifolds intersect the boundary $\pa \cD_{b}$.
Given two points $A, B$ on a stable manifold $W^s(P)$, we will denote by $(A,B)$ the part of the stable manifold between the two points $A$ and $B$.
Let $A_0 \in W^s_{\text{max}}(P_0) \subset \cD_b$ be one endpoint and $B_2\in W^s_{\text{max}}(Q_2) \subset \cD_b$ be the corresponding endpoint in the sense that the stable branches $(P_0, A_0) \cap \cR_b((Q_2, B_2)) \neq\emptyset$. There are three possibilities:
\begin{enumerate}
\item $(P_0, A_0) \subsetneq  \cR_b((Q_2, B_2))$: then $W^s_{\text{max}}(P_0)$ can be extended from beyond $A_0$, contradicting the definition of $W^s_{\text{max}}(P_0)$;

\item $(P_0, A_0) \supsetneq  \cR_b((Q_2, B_2))$: then $W^s_{\text{max}}(Q_2)$ can be extended beyond $B_2$, contradicting the definition of $W^s_{\text{max}}(Q_2)$;

\item $(P_0, A_0) =  \cR_b((Q_2, B_2))$, which implies $A_0 = \cR_b(B_2)$.
\end{enumerate} 
We apply this argument to all six corresponding branches of the  maximally extended stable manifolds of the six hyperbolic periodic points (since none of them reach the boundary $\pa \cD_b$ under the working hypothesis), and conclude that their endpoints form a new periodic orbit of period 6, which is absurd since there is no new periodic orbits on any stable manifolds besides the hyperbolic periodic orbit itself. This conclude the proof.
\end{proof}

Consider the  periodic point $E_0(b)$, which corresponds to the elliptic periodic  point of $F_b$ of period $6$ given in Section \ref{sec.el.hy},
for $1.5< b< \frac{1+\sqrt{5}}{2}\approx 1.618$.
At $b=\frac{1+\sqrt{5}}{2}$, the corresponding billiard trajectory hits the corners of the table $\cQ(b)$ and 
ceases to exist for $b\ge \frac{1+\sqrt{5}}{2}$.
However, this point $E_0(b)$ continues to exist and remains as a $\Theta_b$-fixed point for 
$b \ge \frac{1+\sqrt{5}}{2}$. 
At $b_{\max}:=1+ 2^{-1/2}$, $E_0(b_{\max})=(2^{-1/2},2^{-1/2})$ is
the intersection of the singular curves $\cS_{1,+}(b_{\max})$ and $\cS_{2,+}(b_{\max})$. 
From now on we will assume  $1.5 < b \le b_{\max}$. Recall that the parameter $b_{\crit}\approx 1.63477$ is the solution of $f(b)=-1$ (see Eq.~\eqref{eq.tr.e3} for the definition of $f$)
and is also the critical level curve of the function $\fb(\alpha,\beta)=b$ intersecting the curve $\cJ$, see Lemma \ref{lem.b.crit} and Lemma \ref{lem.fb.level}.
See Remark \ref{rem.iterate} about our convention about classifications of periodic points with symmetries.
\begin{lemma}\label{lem.E0.eph}
The $\Theta_b$-fixed point  $E_0(b)$ is  elliptic for $1.5 < b < b_{\crit}$, parabolic for $b = b_{\crit}$, 
and hyperbolic for $b_{\crit}< b \le b_{\max}$.
\end{lemma}
\begin{proof}
Using the reflection symmetry \eqref{eq.lr.sym} we have 
\begin{align}
\Phi_{b}&=\cL_{b} \cR_{b} \cL_{b} = \cI \cR_{b} \cL_{b} \cR_{b}  \cI= \cI \Psi_{b} \cI, 
\label{eq.phiIphiI} \\
\Theta_b &= (\cR_{b}  \cL_{b})^3=\Psi_{b} \circ \Phi_{b} =(\Psi_{b} \cI)^2.\label{eq.phipsi}
\end{align}

Let $A=(a_{ij})$ be the Jacobian matrix of $\Psi_b=\cR_{b} \cL_{b} \cR_{b}$ at 
$E_0$. 
Since $\cI(E_0)=E_0$, we have $D_{E_0}\Theta_b =(AJ)^2$.
So we just need to find the trace of the  matrix $\hat{A}=AJ$. Note that 
\begin{align}
a_{11} &= \frac{b (-8 b^3+24 b^2-20 b+3)}{b-1}; \label{eq.Psi.S0.11} \\
a_{12} &= \frac{-2 b^2+4 b-1}{b-1}; \label{eq.Psi.S0.12} \\
a_{21} &= \frac{(8 b^4-24 b^3+20 b^2-4 b+1)(8 b^4-24 b^3+20 b^2-2 b-1)}{(b-1)(2 b^2-4 b+1)}; \label{eq.Psi.S0.21}  \\
a_{22} &= \frac{b (8 b^3-24 b^2 +20 b-3)}{b-1}. \label{eq.Psi.S0.22}
\end{align}
It follows from \eqref{eq.Psi.S0.12} and \eqref{eq.Psi.S0.21} that
\begin{align}
\frac{1}{2}\mathrm{tr}\hat{A} & = \frac{a_{12}+a_{21}}{2}
= \frac{32 b^7 -160 b^6+288 b^5 - 216 b^4+ 54 b^3+ 2 b^2-4b+1}{2 b^2-4 b+1}. \label{traceEll}
\end{align}
Note that $\frac{1}{2}\mathrm{tr}\hat{A}=f(b)$,  the function given in Eq.~\eqref{eq.tr.e3}. 
Let $b_{\crit} \approx 1.63477$ be the unique solution of $f(b)=-1$ on the interval
$1.5 < b \le  b_{\max}$. Then $-1<f(b)<1$ for $1.5<b< b_{\crit}$ and $f(b)<-1$ for $b_{\crit}< b < b_{\max}$. This completes the proof.
\end{proof}

It follows from Lemma \ref{lem.E0.eph} that there are new periodic points 
bifurcating from the corresponding point $E_0(b)$ for $b_{\crit}< b \le b_{\max}$.
See Lemma \ref{lem.per.bif} for more details.
Numerical results suggest that there is no other periodic orbit of period $6$ beyond the ones we have found for all $1.5< b < b_{\crit}$. We are able to verify this rigorously for $1.5<b<1.5+\delta$ for some $\delta>0$. See Proposition \ref{pro.no.new.po6}.

\subsection{The set of points with parallel trajectories}
Let $\cM^{+}=\{(d_{\ell}, d_{r})\in \cM: d_{\ell}\ge 0, d_{r}\ge 0\} \backslash \{(0,0)\}$ be the first quadrant of the phase space $\cM$,
$\cC_{\Phi_{b}}$ be the set of points $(d_{\ell}, d_{r})\in \cM^{+}$ such that $(d_{\ell}', d_{r}'):=\Phi_{b}(d_{\ell}, d_{r})$
satisfies $d_{\ell}'-d_{r}' = d_{r}-d_{\ell}$, and $\cC_{\Psi_{b}}$ be the set of points $(d_{\ell}, d_{r})\in \cM^{+}$  such that $(d_{\ell}', d_{r}'):=\Psi_{b}(d_{\ell}, d_{r})$
satisfies $d_{\ell}'-d_{r}' = d_{r}-d_{\ell}$. Analytically, we have
\begin{align}
\cC_{\Phi_{b}}&=\{(d_{\ell}, d_{r})\in \cM^{+}: (1,-1)\cdot \Phi_{b}(d_{\ell}, d_{r}) =d_{r}-d_{\ell}\}; \label{eq.cphi}\\
\cC_{\Psi_{b}}&=\{(d_{\ell}, d_{r})\in \cM^{+}: (1,-1)\cdot\Psi_{b}(d_{\ell}, d_{r}) =d_{r}-d_{\ell}\}. \label{eq.cpsi}
\end{align}
Geometrically, these two conditions mean the final trajectory $(d_{\ell}', d_{r}')$ under $\Phi_{b}=\cL_{b} \cR_{b}\cL_{b}$ and $\Psi_{b}=\cR_{b}\cL_{b} \cR_{b}$ is parallel to the initial trajectory $(d_{\ell}, d_{r})$ on the billiard table $\cQ(b)$, respectively.
This is exactly the set of points studied in Section~\ref{sec.prl}. See also Eq.~\eqref{eq.equal.diff}.
Recall that for a given point $(\alpha, \beta)$ in the domain $D$ defined in \eqref{def.domD}, consider the corresponding parallel orbit coming from right (with angle $\beta$ on top) on the table $\cQ(b)$, where $b=\fb(\alpha, \beta)$ is given in \eqref{bsquare}. 
See Fig.~\ref{figure:parallel}.
Consider the map on the domain $D$ induced by the $\ccM$-shaped orbits:
\begin{align}
H_{\sccM}: D \to \cM, (\alpha, \beta)\mapsto \Big(\sin\beta-\sin (\beta-\alpha)+\frac{\sin2\alpha\sin\beta-\sin\alpha \sin2\beta}{\sin (2\alpha+2 \beta)}, \sin\beta \Big). \label{abdldr}
\end{align} 
We also consider  the map on the domain $D$ induced by the $\crM$-shaped orbits:
\begin{align}
H_{\scrM}: D \to \cM, (\alpha, \beta)\mapsto \Big(\sin\alpha, \sin\alpha+\sin (\beta-\alpha)-\frac{\sin2\alpha\sin\beta-\sin\alpha \sin2\beta}{\sin (2\alpha+2 \beta)}\Big). 
\end{align}
Let $(d_\ell', d_r')=\Phi_{b}(d_\ell, d_r)$. Then the point $(d_\ell, d_r)=H_{\sccM}(\alpha, \beta)$ is characterized by
\begin{align}
d_{\ell}' - d_r'=d_r-d_\ell. \label{char.prl}
\end{align}
Therefore, we have $\cC_{\Phi_b}=H_{\sccM}(\fb^{-1}(b))$.
Similarly, we obtain  $\cC_{\Psi_b}=H_{\scrM}(\fb^{-1}(b))$.
It follows directly from the construction of the \ccM- and \crM-orbits that $\Phi_{b}(\cC_{\Phi_b})=\cC_{\Phi_b}$ and $\Psi_{b}(\cC_{\Psi_b})=\cC_{\Psi_b}$, respectively.
Rotating the table by $\pi$, we have $\cI(\cC_{\Psi_b})=\cC_{\Phi_b}$ and vice versa.
Let $\cC_{\Phi_b}'=H_{\sccM}(\fb^{-1}(b)\cap D')$, where $D'=\{(\alpha, \beta) \in D: \alpha \le \beta\}$. Then it follows from Lemma~\ref{lemma: dl leq dr} that $\Phi_b(\cC_{\Phi_b}')$ (up to its two endpoints) is contained in the interior of the domain bounded by $\cC_{\Phi_b}'$, $\cI(\cC_{\Phi_b}')$ and the two axes.  See Fig.~\ref{fig.prl} for illustrations of the two curves $\cC_{\Phi_b}$ and $\cC_{\Psi_b}$ for various choices of the parameter $b$. Note that the dashed parts of both curves will be omitted from consideration in \eqref{eq.C.prl} when constructing the set $\cC_{\prl, b}$ for $b> b_{\crit}$.

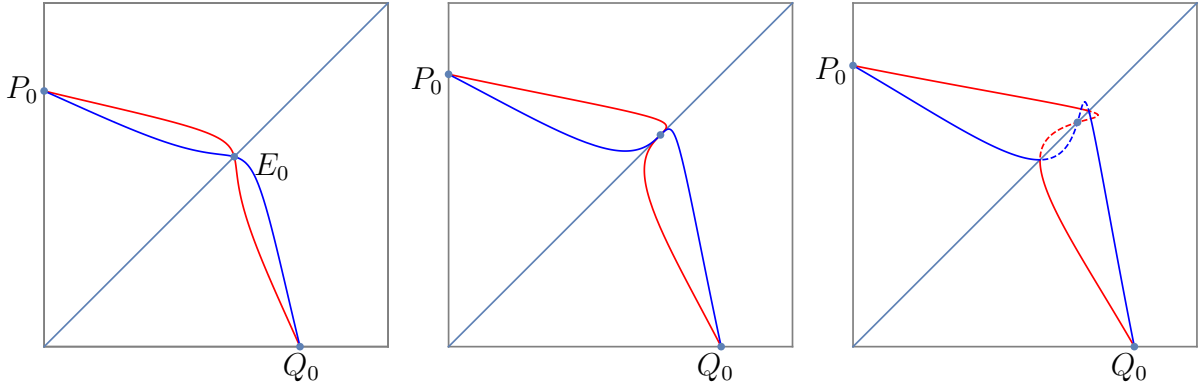
\begin{figure}[htbp]
\begin{overpic}[width=2in, unit=0.1in]{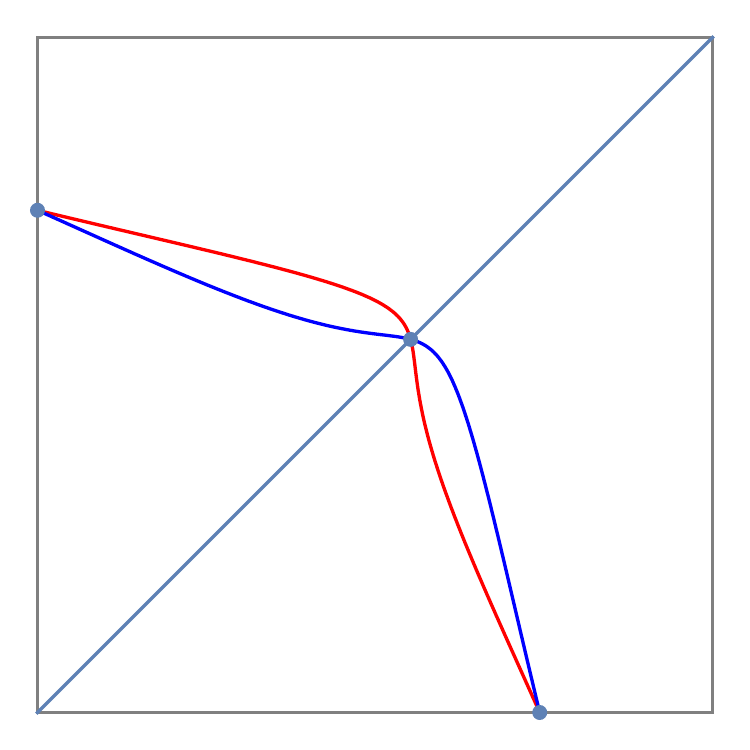}
\put (-1,14) {$P_0$} %
\put (13.5,-0.5) {$Q_0$} %
\put (12,10) {$E_0$} %
\end{overpic}
\quad
\begin{overpic}[width=2in, unit=0.1in]{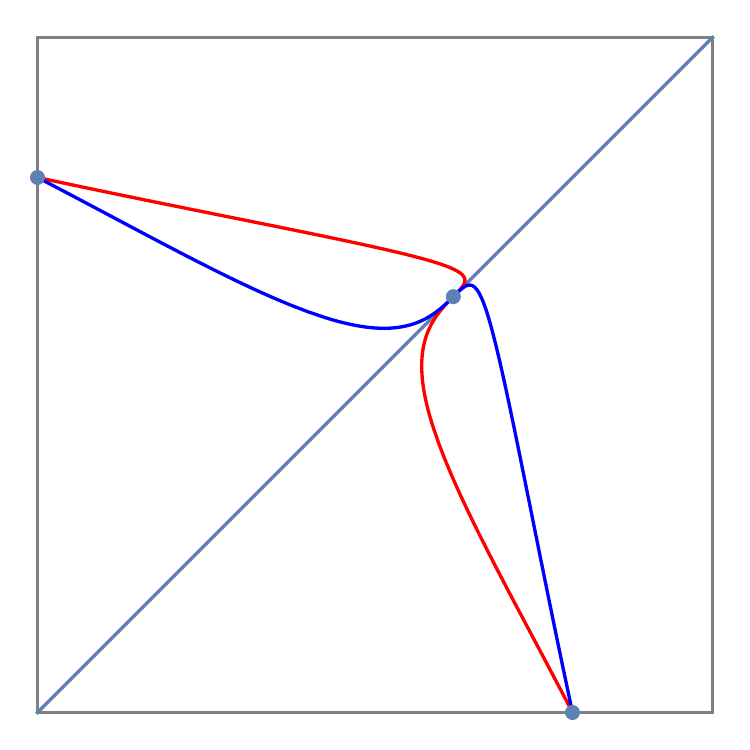}
\put (-1,14.5) {$P_0$} %
\put (14,-0.5) {$Q_0$} %
\end{overpic}
\quad
\begin{overpic}[width=2in, unit=0.1in]{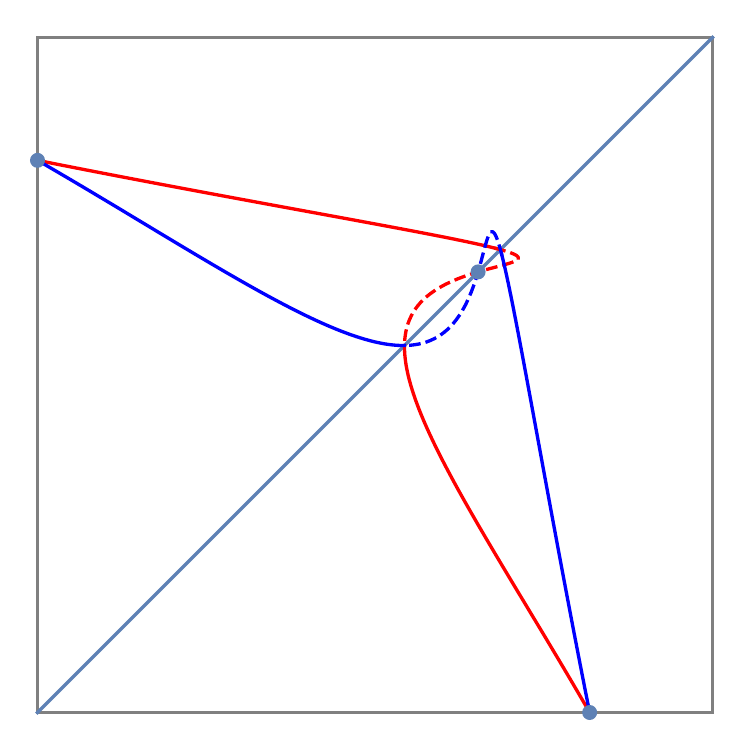}
\put (-1,15) {$P_0$} %
\put (14.5,-0.5) {$Q_0$} %
\end{overpic}
\caption{The curves $\cC_{\Phi_b}$ (red)
and $\cC_{\Psi_b}$ (blue)  for $b=1.6$ (left),
$b=1.63477$ (middle) and $b=1.66$ (right).} \label{fig.prl}
\end{figure}

Note that both curves $\cC_{\Phi_b}$ and $\cC_{\Psi_b}$ end at the corresponding hyperbolic periodic points $P_0(b)$ and $Q_0(b)$.
For example, for $(0,\beta) \in \fb^{-1}(b)$, we have $b^2=\fb(0,\beta)^2=2+ \frac{1}{4\cos^2\beta}$, and hence
$H_{\sccM}(0, \beta)=H_{\scrM}(0, \beta)= (0,\sin\beta)=(0, (1- \frac{1}{4(b^2-2)})^{1/2})=P_0(b)$. In the same way we can check that 
\begin{enumerate}
\item for  $(\alpha, 0) \in \fb^{-1}(b)$: $b^2=2+ \frac{1}{4\cos^2\alpha}$ and  $H_{\sccM}(\alpha, 0)=H_{\scrM}(\alpha, 0)=(\sin\alpha, 0)=Q_0(b)$, 

\item for $(\alpha, \alpha) \in  \fb^{-1}(b)$: $b=1+\frac{1}{2\cos\alpha}$ and $H_{\sccM}(\alpha, \alpha) =H_{\scrM}(\alpha, \alpha)=(\sin\alpha, \sin\alpha) =E_0(b)$.
\end{enumerate}
It follows that for each $1.5<b < b_{\crit}$, both  $\cC_{\Phi_b}$ and $\cC_{\Psi_b}$ are  simple smooth curves
in the first quadrant  $\cM^{+} \subset \cM$ connecting $P_0$ to $Q_0$  with a simple transverse intersection with the diagonal at the elliptic periodic point $E_0$.
Note that for $b\in [b_{\crit}, b_{\max}]$, and for each point $(\alpha, \beta) \in \fb^{-1}(b) \cap \cJ$, the $y$-component of the center $O_{r}$ satisfies $Y_{\ell r}=0$, 
see Fig.~\ref{figure:teardrop}. Then the line $O_{\ell} O_r$ is parallel to the trajectories $Q_1P_1^r$ and $P_3^r Q_3$, see Fig.~\ref{figure:parallel0}, which implies $H_{\sccM}(\alpha, \beta)\in \{d_{\ell}=d_{r}\}$. It follows that
at $b=b_{\crit}$, the intersection of   $\cC_{\Phi_b}$ and $\cC_{\Psi_b}$ at $E_0$ becomes tangential
while for  $b_{\crit} < b \le b_{\max}$, 
there are three intersections between  $\cC_{\Phi_b}$ ($\cC_{\Psi_b}$, resp.) and the diagonal $\{d_{\ell}=d_{r}\}$:
the first and the last intersections, say $E_1(b)$ and $E_2(b)$, correspond to the intersection $\fb^{-1}(b)\cap \cJ$, while the intermediate intersection is the elliptic periodic point $E_0(b)$ corresponding to the intersection $\fb^{-1}(b)\cap \{\beta=\alpha\}$. 

\begin{lemma}\label{lem.per.bif}
Let  $b_{\crit} < b \le b_{\max}$, and $E_1(b)$ and $E_2(b)$ be the two intersection points
of $\cC_{\Phi_b}$ (or equally, $\cC_{\Phi_b}$) with the diagonal. Then 
$\Phi_{b}(E_1(b))=\Psi_{b}(E_1(b))= E_2(b)$ and vice versa. In particular, both points  $E_1(b)$ and $E_2(b)$ are fixed points of  $\Theta_b$.
\end{lemma}
\begin{proof}
Let  $b_{\crit} < b \le b_{\max}$  be given. There exists a unique pair $(\alpha, \beta)\in D' \cap \cJ$
such that $b=\fb(\alpha, \beta)$.
Note that the curve $\cJ$ is characterized by $Y_{\ell r}=0$.
In particular,  the line $O_{\ell} O_r$ is parallel with the trajectories $Q_1P_1^r$ (corresponding to $E_1$) and $P_3^r Q_3$ (corresponding to $E_2$). Then we have 
$E_1=(\sin\beta, \sin\beta)$, $E_2=(\sin\alpha, \sin\alpha)$
and $\Psi_b(\sin\alpha, \sin\alpha)=(\sin\beta, \sin\beta)$ and vice versa.
Note that the trajectory obtained by applying  the reflection of the table $\cQ(b)$  about the vertical line $x=\frac{b}{2}$ is the same (just with the reversed direction). It follows that 
$\Phi_b(\sin\alpha, \sin\alpha)=(\sin\beta, \sin\beta)$ and vice versa. This completes the proof by noting that $\Theta_b=\Psi_b \Phi_b$.
\end{proof}

We need to know a little bit more about the stable and unstable directions of the hyperbolic periodic points $P_0(b)$ and $Q_0(b)$. 
Moreover, we will find the tangent directions of the two curves $\cC_{\Phi_b}$ and $\cC_{\Psi_b}$ at the two endpoints $P_0$ and $Q_0$ and show that they are contained in the cone of the stable and unstable directions of the hyperbolic periodic points $P_0$ and $Q_0$ in the first quadrant.
We will focus  on the point $P_0$, since what happens at $Q_0$ is just a reflection about the diagonal. 
\begin{lemma}\label{lem.tan.cprl}
Let $m_{s,u}(b)$ the slopes of the stale and unstable tangent space at $P_0(b)$,
$m_{\sccM}(b)$ and $m_{\scrM}(b)$ the the slopes of the two curves $\cC_{\Phi_b}$ and $\cC_{\Psi_b}$  at $P_0(b)$, respectively. Then we have
\begin{align}
m_u(b) < m_{\scrM}(b) < m_{\sccM}(b) < m_s(b) <0
\end{align}
for any $1.5 < b \le b_{\max}$.
\end{lemma}
\begin{proof}
The tangent vector of the level curve $\fb^{-1}(b)$ at the point  $(0, \beta) \in \fb^{-1}(b)$
with normalized first component is given by
\begin{align}
T_b(0,\beta) = (1, -\frac{2\cos^3\beta -2\cos^2\beta+1}{\sin\beta}\tan\frac{\beta}{2})
=(1, -\frac{2\eta^3 -2\eta^2+1}{\eta+1}),
\end{align}
where $\eta=\cos\beta=\frac{1}{2(b^2 -2)^{1/2}}$ varies from $1$ to $0.522933$
as $b$ varies from $1.5$ to $b_{\max}$. 

(1) The tangent map of $H_{\sccM}$ at $(0, \beta)$ is given by 
\begin{align}
DH_{\sccM}(0,\beta) =\begin{bmatrix}\cos\beta+\frac{1}{\cos\beta} -1 & 0 \\ 0 & \cos \beta \end{bmatrix}
=\begin{bmatrix} \eta+\frac{1}{\eta} - 1 & 0 \\  0 & \eta \end{bmatrix}.
\end{align}
It follows that a tangent vector of  $\cC_{\Phi_b}$ at the point $P_0=H_{\sccM}(0,\beta)$ is given by
\begin{align}
T_{\sccM, b}:=DH_{\sccM}(0,\beta) T_b(0,\beta) &=\begin{bmatrix} \eta+\frac{1}{\eta} - 1 & 0 \\  0 & \eta \end{bmatrix}
\begin{bmatrix}1 \\ -\frac{2\eta^3 -2\eta^2+1}{\eta+1} \end{bmatrix}
=\begin{bmatrix} \eta+\frac{1}{\eta} - 1 \\  -\frac{\eta(2\eta^3 -2\eta^2+1)}{\eta+1}
\end{bmatrix}.
\end{align}
For an easy comparison, we will normalize the first component  of $T_{\sccM, b}$
and rewrite it as $T_{\sccM, b}=(1, m_{\sccM}(\eta))$, where  
\begin{align}
m_{\sccM}(\eta)=  -\frac{\eta^2(2\eta^3 -2\eta^2+1)}{\eta^3+1}. \label{eq.mPhi}
\end{align}
Note that $m_{\sccM} \to -\frac{1}{2}$ when $\eta\to 1$. See Fig.~\ref{fig.comp.slop} for  a plot
of the function $m_{\sccM}(\eta)$.

(2) The tangent map of $H_{\scrM}$ at $(0, \beta)$ is given by 
\begin{align}
DH_{\scrM}(0,\beta) =\begin{bmatrix}1 & 0 \\ \cos\beta+\frac{1}{\cos\beta} & -\cos \beta \end{bmatrix}
=\begin{bmatrix} 1 & 0 \\ 2- \eta-\frac{1}{\eta} & \eta \end{bmatrix}.
\end{align}
It follows that a tangent vector of  $\cC_{\Psi_b}$ at the point $P_0=H_{\scrM}(0,\beta)$ is given by
\begin{align}
T_{\scrM, b}:=DH_{\scrM}(0,\beta) T_b(0,\beta) &=\begin{bmatrix}1 & 0 \\ 2- \eta-\frac{1}{\eta} & \eta \end{bmatrix}
\begin{bmatrix}1 \\ -\frac{2\eta^3 -2\eta^2+1}{\eta+1} \end{bmatrix}
=\begin{bmatrix} 1 \\ \frac{-2 \eta^5+2 \eta^4-\eta^3+\eta-1}{\eta(\eta+1)}
\end{bmatrix}.
\end{align}
Denote $T_{\scrM, b}=(1, m_{\scrM}(\eta))$, where  
\begin{align}
m_{\scrM}(\eta)=  \frac{-2 \eta^5+2 \eta^4-\eta^3+\eta-1}{\eta(\eta+1)}. \label{eq.mPsi}
\end{align}
Note that $m_y \to -\frac{1}{2}$ when $\eta\to 1$. See Fig.~\ref{fig.comp.slop} for  a plot
of the function $m_{\scrM}(\eta)$.

(3) 
Let $A=(a_{ij})_{1\leq i,j\leq 2}:=D_{Q_0}\Psi_{b}$ be the Jacobian matrix of $\Psi_{b}=\cR_{b} \cL_{b} \cR_{b}$ at $\cI(P_0)=Q_0$.
Since $\Psi_{b} \cI(P_0)=P_0$, it follows from \eqref{eq.phiIphiI} that $D_{P_0}\Phi_{b}=J A J$, 
where  $J=\begin{bmatrix}0&1\\ 1&0 \end{bmatrix}$. 
Similarly, it follows from \eqref{eq.phipsi} that 
\begin{align*}
D_{P_0}\Theta_b=D_{Q_0}\Psi_{b} \; D_{P_0}\Phi_{b}
=A  J A  J = \hat{A}^2,
\end{align*}
where 
$\hat{A}=AJ=\begin{bmatrix}a_{12}&a_{11}\\ a_{22}&a_{21} \end{bmatrix}$.
Since the point $P_0$ is a hyperbolic periodic point of period $6$, the two matrices $\hat{A}$ and $D_{P_0}\Theta_b$ share the same eigenspaces. 
After some simplification, we get 
\begin{align}
a_{11} &=-\frac{2 \left(b^2-2\right) \left(16 b^4-64 b^2+63\right)}{2 b^2-3}; \label{eq.dQ11} \\
a_{12} &= - \frac{8 b^4-36 b^2+39}{2 b^2 -3};  \label{eq.dQ12} \\
a_{21} &=2(16 b^4 -64 b^2 +63)\sqrt{b^2-2} -\frac{(8 b^2-15) (8b^4 - 32b^2 + 31)}{2b^2-3};  \label{eq.dQ21} \\
a_{22} &=\frac{8 b^4-36 b^2+39}{\sqrt{b^2-2}} -\frac{2(8 b^4 - 35 b^2 + 36)}{2 b^2-3}.  \label{eq.dQ22}
\end{align}

Using the relation $b^2=2+\frac{1}{4\eta^2}$, we can rewrite the matrix
$A=(a_{ij})_{1\leq i,j\leq 2}$ as 
\begin{align*}
\frac{1}{\eta^5(1+2\eta^2)}
\begin{bmatrix}
\eta (\eta^4-1) & \eta^3 (2 \eta^4+2 \eta^2-1) \\
2 \eta^7-2 \eta^6+4 \eta^5-\eta^4-\eta^3+2 \eta^2-2 \eta+1 & -4 \eta^8+8 \eta^7-6 \eta^6+3 \eta^5-2 \eta^3+\eta^2
\end{bmatrix}.
\end{align*}
Then the stable and unstable eigenvalues $\lambda_{s,u}(\eta)$ and the corresponding slopes of their eigenvectors $m_{s,u}(\eta)$ of the matrix $\hat A= A J$ are given by
\begin{align}
\lambda_s(\eta)&= \frac{4 \eta^7-2 \eta^6+6 \eta^5-\eta^4-2 \eta^3+2 \eta^2-2 \eta+1- \sqrt{\Delta}}{2 (2 \eta^7+\eta^5)}, \\
\lambda_u(\eta)&= \frac{4 \eta^7-2 \eta^6+6 \eta^5-\eta^4-2 \eta^3+2 \eta^2-2 \eta+1+\sqrt{\Delta}}{2 (2 \eta^7+\eta^5)}, \\ 
m_s(\eta)&= -\frac{2 \eta^2 (1 - \eta - \eta^2 + 2 \eta^3 - 4 \eta^4 + 4 \eta^5)}{(1 + \eta) (1 + \eta^2) (1 - 2 \eta + 2 \eta^2)+(2 \eta^2+1)\sqrt{\Gamma}}, \\
m_u(\eta)&= -\frac{2 \eta^2 (1 - \eta - \eta^2 + 2 \eta^3 - 4 \eta^4 + 4 \eta^5)}{(1 + \eta) (1 + \eta^2) (1 - 2 \eta + 2 \eta^2)-(2 \eta^2+1)\sqrt{\Gamma}},
\end{align}
where 
\begin{align*}
\Delta &=(4 \eta^7 -2 \eta^6 +6 \eta^5 -\eta^4 -2 \eta^3 +2 \eta^2 -2 \eta +1)^2-4\eta^{10} (4 \eta^{4}+4 \eta^{2}+1), \\
\Gamma &=(1-\eta) (4 \eta^6+3 \eta^5+\eta^4+2 \eta^3-2 \eta^2-\eta+1).
\end{align*}
See Fig.~\ref{fig.comp.slop} for the plots of the functions $m_{s}(\eta)$ and $m_{u}(\eta)$.
To complete the proof, it suffices to note that

(i) $m_{\sccM}(\eta)- m_{\scrM}(\eta) 
=\frac{(1-\eta)^2}{\eta^4+\eta} (2 \eta^5-2 \eta^4+\eta^3+\eta^2+1) >0$ for $0.5<\eta<1$.

(ii) To estimate $m_{s}(\eta)- m_{\scrM}(\eta)$, note that both denominators are positive. We multiply it by the common denominator and eliminate the common factor $\eta^2$:
\begin{align}
&-2(1 - \eta - \eta^2 + 2 \eta^3 - 4 \eta^4 + 4 \eta^5)(\eta^3+1) \nonumber \\
&+((1 + \eta) (1 + \eta^2) (1 - 2 \eta + 2 \eta^2)+(2 \eta^2+1)\sqrt{\Gamma})(2\eta^3 -2\eta^2+1) \nonumber \\
=&(2 \eta^2+1)\Big( (1-\eta)(2\eta^5+3\eta^2-1) +(2\eta^3 -2\eta^2+1)\sqrt{\Gamma}\Big). \label{slope.sta.phi}
\end{align}
Set $p_1(\eta):=(1-\eta)(2\eta^5+3\eta^2-1)$ and $q_1(\eta):=(2\eta^3 -2\eta^2+1)\sqrt{\Gamma}$. 
It is easy to see that 
\begin{align}
q_1(\eta)^2 - p_1(\eta)^2=4\eta^3 (1-\eta^2)(\eta^3-\eta^2+1) (4 \eta^5-4 \eta^4+2 \eta^3-\eta^2-\eta+1). \label{eq.q1p1}
\end{align}
Note that all factors on the right side of \eqref{eq.q1p1} are positive on the interval $(0.5, 1)$.
It follows that $|q_1(\eta)| > |p_1(\eta)|$ on the interval  $(0.5, 1)$.
Combining with the fact that $p_1(0.5)=-0.09375$ and $q_1(0.5)\approx 0.363092$,
we have $q_1(\eta)=|q_1(\eta)| > |p_1(\eta)| \ge -p_1(\eta)$  on the interval  $(0.5, 1)$.
It follows that  $p_1(\eta)+q_1(\eta)>0$,  or  equally, $m_{s}(\eta) > m_{\scrM}(\eta)$ for $0.5<\eta<1$.

(iii) To estimate $m_{\sccM}(\eta) - m_{u}(\eta)$, we multiply it by the common denominator:
\begin{align*}
&(-2 \eta^5+2 \eta^4-\eta^3+\eta-1)((1 + \eta) (1 + \eta^2) (1 - 2 \eta + 2 \eta^2)-(2 \eta^2+1)\sqrt{\Gamma})  \\
&+2 \eta^2 (1 - \eta - \eta^2 + 2 \eta^3 - 4 \eta^4 + 4 \eta^5)\eta(\eta+1)  \\
=&(1 + 2 \eta^2) \Big((1 - \eta^2)  (-1 + 2 \eta - \eta^2 - \eta^3 + 3 \eta^4 -  6 \eta^5 + 2 \eta^6)+(2 \eta^5-2 \eta^4+\eta^3-\eta+1)\sqrt{\Gamma}\Big).
\end{align*}
Let $p_2(\eta)=(1 - \eta^2)  (-1 + 2 \eta - \eta^2 - \eta^3 + 3 \eta^4 -  6 \eta^5 + 2 \eta^6)$
and $q_2(\eta)=(2 \eta^5-2 \eta^4+\eta^3-\eta+1)\sqrt{\Gamma}$.
It is easy to see that 
\begin{align*}
q_2(\eta)^2 - p_2(\eta)^2
&=4\eta^7 (1-\eta^2)(\eta^3-\eta^2+1) (4 \eta^5-4 \eta^4+2 \eta^3-\eta^2-\eta+1) \\
&=\eta^4(q_1(\eta)^2 - p_1(\eta)^2)>0. 
\end{align*}
It follows that $|q_2(\eta)| > |p_2(\eta)|$ on the interval  $(0.5, 1)$.
Combining with the fact that $p_2(0.5)=-0.2578125$ and $q_2(0.5) \approx 0.272319$,
we have $q_2(\eta)=|q_2(\eta)| > |p_2(\eta)| \ge -p_2(\eta)$  on the interval  $(0.5, 1)$.
It follows that  $p_2(\eta)+q_2(\eta)>0$,  or  equally, $m_{\sccM}(\eta) > m_{u}(\eta)$ for $0.5<\eta<1$.

Collecting terms, we have that $m_u(b) < m_{\scrM}(b) < m_{\sccM}(b) < m_s(b)<0$ for $1.5<b \le b_{\max}$.
\end{proof}

\begin{figure}[htbp]
\includegraphics[height=2in]{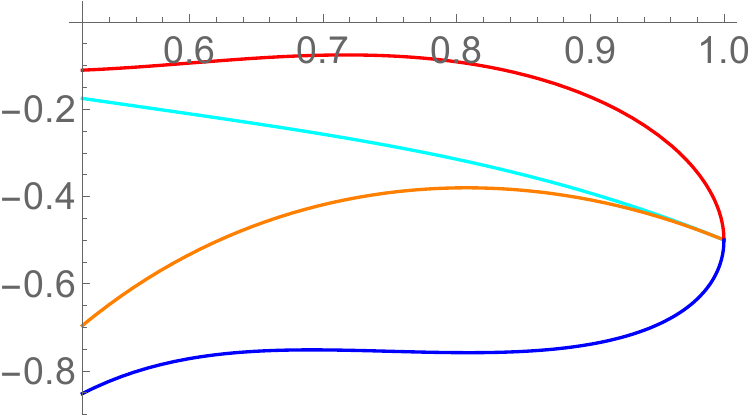}
\caption{Plot of slopes from top to bottom: $m_s(\eta)$ (red), $m_{\sccM}(\eta)$ (cyan), $m_{\scrM}(\eta)$ (orange) and $m_u(\eta)$ (blue) as $0.52<\eta <1$.} \label{fig.comp.slop}
\end{figure}

Let  $\cC(E^u_{P_0},E^s_{P_0})$ be the cone in the plane $T_{P_0}M$ from
$E^u_{P_0}$ to $E^s_{P_0}$ (counterclockwise). It follows from Lemma \ref{lem.tan.cprl}
that the tangent vectors of the two curves $\cC_{\Phi_b}$ and $\cC_{\Psi_b}$ at $P_0$ are contained in the cone $\cC(E^u_{P_0},E^s_{P_0})$. The same conclusion holds at the point $Q_0$.

Let  $1.5<b \le b_{\max}$,  $D(\Phi_{b})$ be the set of points $(d_{\ell},d_r) \in \cM^{+}$ such that $(d_{\ell}', d_{r}')=\Phi_{b}(d_{\ell}, d_r)$ is well-defined. It is possible that $D(\Phi_{b})$ has multiple connected components. 
Note that $\cC_{\Phi_b} \subset D(\Phi_{b})$ and $\cC_{\Phi_b}$ is invariant under $\Phi_{b}$. Let $D_0(\Phi_{b}) \subset D(\Phi_{b})$ be the connected component of $D(\Phi_{b})$ containing $\cC_{\Phi_b}$.
Then the curve $\cC_{\Phi_b}$ divides $D_0(\Phi_{b})$ into two parts: 
denote the exterior part by $D_{\text{ext}}(\Phi_{b})$, and the interior part by $D_{\text{int}}(\Phi_{b})$. 
\begin{lemma}\label{lem.mono}
Let $1.5<b \le b_{\max}$. Then
\begin{align}
d_{\ell}' - d_{r}' > d_r - d_{\ell} \label{eq.mono}
\end{align}
 for any $(d_{\ell}, d_r) \in D_{\text{ext}}(\Phi_{b})$, where $(d_{\ell}', d_{r}')=\Phi_{b}(d_{\ell}, d_r)$.
The reversed inequality hold for points in $D_{\text{int}}(\Phi_{b})$.
\end{lemma}
\begin{proof}
Let $1.5<b \le b_{\max}$ be fixed. Note that the equality $d_{\ell}' - d_{r}'= d_r - d_{\ell}$ holds if and only if $(d_{\ell}, d_r)\in \cC_{\Phi_b}$.
Then the inequality $d_{\ell}' - d_{r,\delta}' \neq d_r - d_{\ell}$ holds 
for any
$(d_{\ell}, d_r) \in D_{\text{ext}}(\Phi_{b})$, where $(d_{\ell}', d_{r}')=\Phi_{b}(d_{\ell}, d_r)$.
Since $D_{\text{ext}}(\Phi_{b})$ is connected on which $\Phi_{b}$ is smooth,
 the direction of the inequality  \eqref{eq.mono} is fixed with respect to the coordinates $(d_{\ell}, d_r) \in D_{\text{ext}}(\Phi_{b})$ and with respect to 
the parameter $b$.
So we only need to check the direction of the inequality at one point $(d_{\ell}, d_r) \in D_{\text{ext}}(\Phi_{b})$ for one $b\in (1.5, b_{\max})$, say $b=1.6$.
See the first picture in Fig.~\ref{fig.prl} for $b=1.6$.
In the proof of Lemma \ref{lem.E0.eph}, we have obtained the tangent matrix 
$D_{E_0}\Psi_b =\begin{bmatrix} a_{11} & a_{12} \\ a_{21} & a_{22} \end{bmatrix}$.
It follows that $D_{E_0}\Phi_{1.6}=J D_{E_0}\Psi_{1.6} J=\begin{bmatrix} a_{22}(1.6) & a_{21}(1.6) \\ a_{12}(1.6) & a_{11}(1.6)\end{bmatrix} =\begin{bmatrix} 0.874667 & 0.503482 \\ 0.466667 & -0.874667\end{bmatrix}$.
Then the image of the tangent vector $v=(1,1) \in T_{E_0}\cM$ satisfies 
\begin{align}
D_{E_0}\Phi_b(v) &=\begin{bmatrix} a_{22}(1.6)+ a_{21}(1.6) \\ a_{12}(1.6)+ a_{11}(1.6)\end{bmatrix} =\begin{bmatrix}1.37815 \\ -0.408\end{bmatrix}.
\end{align}
Therefore, $d_{\ell}' - d_{r}' > d_r - d_{\ell}$ for $(d_{\ell}, d_r) \in D_{\text{ext}}(\Phi_{b})$. The case about points in $D_{\text{int}}(\Phi_{b})$ can be proved in the same way. This completes the proof.
\end{proof}


There is a similar property as Lemma \ref{lem.mono} with the curve $\cC_{\Psi_b}$ and $\Psi_{b}$
for $1.5<b \le b_{\max}$. More precisely, let $D(\Psi_{b}) \subset \cM^{+}$ be the part of the domain of $\Psi_{b}$ in the first quadrant $\cM^{+}$. It is an open subset and contains the curve $\cC_{\Psi_b}$. Let $D_{0}(\Psi_{b})\subset D(\Psi_{b})$ be the connected component of $D(\Psi_{b})$ containing $\cC_{\Psi_b}$. 
Then $\cC_{\Psi_b}$ divides $D_0(\Psi_{b})$ into two parts: 
denote the exterior part by $D_{\text{ext}}(\Psi_{b})$, and the interior part by $D_{\text{int}}(\Psi_{b})$. 
\begin{lemma}\label{lem.mono3}
Let $1.5<b \le b_{\max}$. Then
\begin{align}
d_{\ell}' - d_{r}' < d_r - d_{\ell} \label{eq.mono3}
\end{align}
 for any $(d_{\ell}, d_r) \in D_{\text{ext}}(\Psi_{b})$, where $(d_{\ell}', d_{r}')=\Psi_{b}(d_{\ell}, d_r)$.
The reversed inequality holds for points in $D_{\text{int}}(\Psi_{b})$.
\end{lemma}
The proof is omitted, since it follows from Lemma \ref{lem.mono} by applying the involution $\cI$.

\begin{lemma}\label{lemma: F, GF}
For $b\in (1.5, b_{\max})$, there does \emph{not} exist $(d_\ell, d_r)\in \cS_{1,+}(b)$ such that 
\begin{align}\label{eq: lemma F, GF}
\cL_{b}\circ \cI (d_\ell, d_r) = \cR_{b}\circ \cL_{b}(d_\ell, d_r).
\end{align}
\end{lemma}
\begin{proof}
Suppose on the contrary that there is a solution $(d_\ell, d_r)\in \cS_{1,+}(b)$ of \eqref{eq: lemma F, GF}. It follows that
\begin{align*}
&d_\ell-d_r=b(1-2d_r^2); \\
&d_\ell= \cR_{b}(d_r-b,d_r)_2=d_r-b-b(1-2(d_r-b)^2)=2b(d_r-b)^2+d_r-2b; \\
&(d_\ell-d_r)(1-2d_\ell^2)-2d_\ell\sqrt{(b^2-(d_\ell-d_r)^2)(1-d_\ell^2)}+d_\ell=d_r-b.
\end{align*} 
Solving $d_r$ from the first equation, we have $d_{r,\pm}=\frac{b\pm \sqrt{3-b^2}}{2}$.
We plug it into the second equation and obtain
\begin{align*}
d_{\ell,\pm}:=d_{r,\pm}+b(1-2d_{r,\pm}^2)=\pm\frac{1}{2}(1-2b^2)\sqrt{3-b^2}.
\end{align*}
Plugging $(d_{\ell,\pm},d_{r, \pm})$ into the third equation, we get, respectively
\begin{align}
\label{eq: l,r,p,b}&\scriptstyle{\frac{1}{4} \big(-\left(1-2 b^2\right)^3 \left(3-b^2\right)^{3/2}+4 \left(1-2 b^2\right) \sqrt{3-b^2}-4 \sqrt{3-b^2}-\left(1-2 b^2\right)^2 \left(b^2-3\right) \left(\sqrt{3-b^2}+b\right)}\\
\nonumber&\scriptstyle{+\left(2 b^2-1\right) \sqrt{3-b^2} \sqrt{b^2 \left(4 b^6-16 b^4+13 b^2+1\right) \left(4 b^4-12 b^2-4 \sqrt{3-b^2} b+3\right)}\big)}=0,\\
\label{eq: l,r,m,b}&\scriptstyle{\frac{1}{4} \big(\left(1-2 b^2\right)^3 \left(3-b^2\right)^{3/2}+4 \left(2 b^2-1\right) \sqrt{3-b^2}+4 \sqrt{3-b^2}+\left(b^2-3\right) \left(1-2 b^2\right)^2 \left(\sqrt{3-b^2}-b\right)}\\
\nonumber&\scriptstyle{+\left(1-2 b^2\right) \sqrt{3-b^2} \sqrt{b^2 \left(4 b^6-16 b^4+13 b^2+1\right) \left(4 b^4-12 b^2+4 \sqrt{3-b^2} b+3\right)}\big)}=0.
\end{align}
The solutions to the Equations \eqref{eq: l,r,p,b} and \eqref{eq: l,r,m,b} are
\begin{align*}
&b=0,-\sqrt{\frac{3}{2}}, \pm\sqrt{3},\frac{1}{2}(2-\sqrt{2}), \frac{1}{2}(-2-\sqrt{2});\\
&b=0,\sqrt{\frac{3}{2}}, \pm \sqrt{3}, \frac{1}{2}(-2\pm\sqrt{2}), \frac{1}{2}(2+\sqrt{2}),
\end{align*}
respectively. Since none of the solutions is in $(1.5, 1+ 2^{-1/2})$, the lemma follows. 
\end{proof}

\section{Heteroclinic intersections and topological entropy}\label{sec.homoclinic}

Let $1.5<b \le b_{\max}$, $F_b$ be the billiard map on the lemon table  $\cQ(b)$.
As we have seen in Section \ref{pre.lemon}, the periodic orbit $\cO_2(b)$ is elliptic and $D_P F_b^2$ is conjugate to the rotation matrix $R_{\theta}$, where $\theta=\arccos(2(b-1)^2 -1)$.
It follows that  the rotation number $\rho(F_{b}^2, P)$ of $F_b^2$ at $P$ satisfies $\rho(F_{1.5}^2, P)=\frac{1}{3}$ and is decreasing with respect to $b$ for $b>1.5$.
Let $\rho>\frac{1}{3}$ be a Diophantine number that is sufficiently close to $\frac{1}{3}$, $\delta_{\rho}>0$ be a positive number given by Proposition \ref{pro.delta.rho}.
For any $1.5 < b < 1.5+\delta_{\rho}$, let $C_{\rho}(b)$ be the union of two $F_b^2$-invariant circles of rotation number $\rho$ surrounding the elliptic periodic points $P, F_b(P)\in \cO_2(b)$, and $D_{\rho}(b)\subset M$ be the union of two $F_b^2$-invariant disks enclosed by $C_{\rho}(b)$. Since $\rho>\frac{1}{3}$, $D_{\rho}(b)$ contains both periodic orbits $\cO_6^e(b)$ and $\cO_6^h(b)$ constructed in Section \ref{sec.el.hy}.
In the phase space $\cM$, the elliptic periodic orbit $\cO_2(b)$ corresponds to the origin $(0,0)$,
$C_{\rho}(b)$ corresponds to a circle $\cC_{\rho}(b)$ surrounding the origin $(0,0)$ that is invariant under both $\cL_{b}$ and $\cR_{b}$,
and $D_{\rho}(b)$ corresponds to a disk  $\cD_{\rho}(b)$ containing the origin $(0,0)$ that is invariant under both $\cL_{b}$ and $\cR_{b}$.
In particular, $\cD_{\rho}(b)$ contains four periodic orbits of period $6$ containing the points $P_0$, $Q_0$, $E_0$ and $-E_0$, respectively.

\subsection{Some characterizations of the system for $b$ close to $1.5$}
Let $\cC_{\Phi_{b}}^{\sym} \subset \cM^+$ be the set of points $(d_{\ell}, d_{r}) \in \cM^+$ such that
\begin{align}
d_\ell' -d_r' =d_\ell-d_r,\label{eq.cphi.sym}
\end{align}
where $(d_{\ell}', d_{r}')=\Phi_{b} (d_{\ell}, d_{r})$.
Geometrically, $\cC_{\Phi_{b}}^{\sym}$ consists of points in $\cM^+$ whose reflection aims at the point  $B_r=(1,0) \in \Gamma_r$.
It follows that $\cC_{\Phi_{b}}^{\sym}=\cL_{b}(\{(-x, (b-1)x): 0\le x\le 1\})$.
Note that $\cC_{\Phi_{b}}^{\sym}$ intersects the diagonal exactly once at $E_0$,
since any such intersection point
leads to a parallel orbit that is symmetric with respect to both axes and must be the periodic point $E_0$. Similarly we can define $\cC_{\Psi_{b}}^{\sym}$. It is easy to see that $\cI(\cC_{\Phi_{b}}^{\sym})=\cC_{\Psi_{b}}^{\sym}$. See Fig.~\ref{fig.csym}.

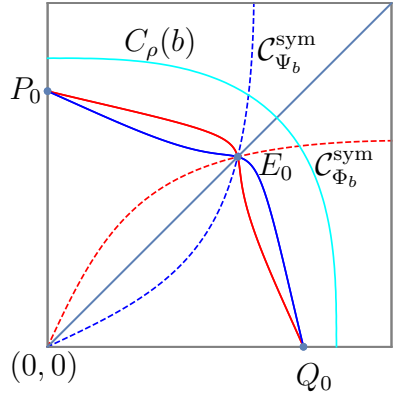
\begin{figure}[htbp]
\begin{overpic}[width=2in, unit=0.1in]{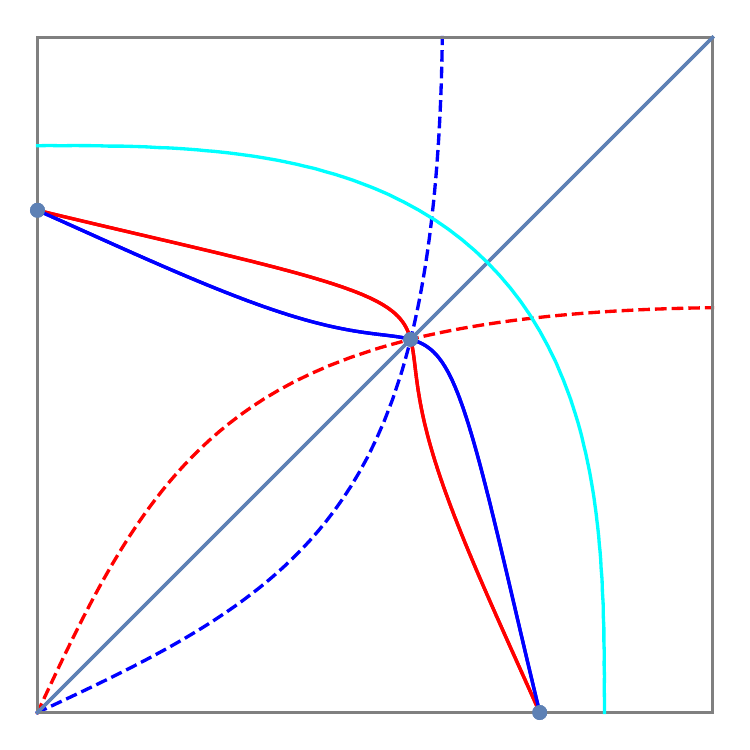}
\put (-1,-0.5) {$(0,0)$} %
\put (12,10) {$E_0$} %
\put (-1,14) {$P_0$} %
\put (14,-1) {$Q_0$} %
\put (5,16.5) {$C_{\rho}(b)$} %
\put (15,10) {$\cC_{\Phi_b}^{\sym}$} %
\put (12,16) {$\cC_{\Psi_b}^{\sym}$} %
\end{overpic}
\caption{The curves $\cC_{\Phi_b}$ (red)
and $\cC_{\Psi_b}$ (blue), $\cC_{\Phi_{b}}^{\sym}$ (red and dashed) and  $\cC_{\Psi_{b}}^{\sym}$ (blue and dashed), where $b=1.6$.   The cyan curve is an illustration of part of the Diophantine invariant curve $C_{\rho}(b)$ in the first quadrant.} \label{fig.csym}
\end{figure}

\begin{lemma}\label{lem.cphi.sym}
(1) Both $\cC_{\Phi_{b}}^{\sym}\cap \cC_{\Phi_{b}}=\{E_0(b)\}$
and $\cC_{\Psi_{b}}^{\sym}\cap \cC_{\Psi_{b}}=\{E_0(b)\}$ hold  for any $1.5<b  <b_{\crit}$. 

(2) There exists a positive number $\delta_{C}>0$ such that
$\cC_{\Phi_{b}}^{\sym}\cap \cC_{\Psi_{b}}=\{E_0(b)\}$ 
and $\cC_{\Psi_{b}}^{\sym}\cap \cC_{\Phi_{b}}=\{E_0(b)\}$  for any $1.5<b \le  1.5+\delta_C$. 
\end{lemma}
\begin{proof}
(1) It is clear that $E_0(b)\in \cC_{\Phi_{b}}^{\sym}\cap \cC_{\Phi_{b}}$.
It suffices to show that $(d_{\ell}, d_{r})=E_0$ for any $(d_{\ell}, d_{r})\in \cC_{\Phi_{b}}^{\sym}\cap \cC_{\Phi_{b}}$. Note that such a point satisfies
\begin{align}
d_\ell' -d_r' =d_\ell-d_r = d_r - d_\ell,
\end{align}
where $(d_{\ell}', d_{r}')=\Phi_{b} (d_{\ell}, d_{r})$.
It follows that $d_\ell =d_r$, and hence $(d_{\ell}, d_{r})\in \cC_{\Phi_{b}}\cap \Delta =\{E_0(b)\}$ for all $1.5<b\le b_{\crit}$. The second intersection follows in the same way.

(2) Note that the tangent maps of $\cL_{b}$, $\cR_{b}$ and $\Phi_{b}$ at the origin for $b_0=1.5$:
\begin{align*}
T_0 \cL_{b_0}=\begin{bmatrix} -1 & -1 \\  0 & 1 \end{bmatrix},\quad 
T_0\cR_{b_0}=\begin{bmatrix} 1 & 0 \\ -1 & -1 \end{bmatrix}, \quad
T_0\Phi_{b}=\begin{bmatrix} 0 & 1 \\  1 & 0 \end{bmatrix}.
\end{align*}
In this case, the \emph{linearized} equation of (\ref{eq.cphi.sym}) at $(0,0)$ is 
\begin{align*}
d_\ell=d_r,
\end{align*}
which gives the diagonal. 
Since for $b$ close to $1.5$ and in a small neighborhood of the origin, Eq.~(\ref{eq.cphi.sym}) is a $C^\infty$-perturbation of $d_\ell-d_r=0$,  $\cC_{\Phi_b}^{\sym}$ is cut out smoothly as a $C^\infty$-perturbation of the diagonal.  
It follows that for any $\eps>0$, there exist a number $\delta>0$ and a neighborhood $U\subset \cM$ of the origin $(0,0)$ such that for every $1.5 < b < 1.5 +\delta$, $\cC_{\Phi_b}^{\sym} \cap U$ is a smooth (properly embedded) curve that is $\eps$-close to the diagonal under $C^1$ norm. 
Similarly, there exists a number $\delta'>0$ such that for each $1.5<b<1.5+\delta'$, 
the slope of the curve $\cC_{\Psi_b}$ is strictly negative.  It follows such two curves can intersect only once for $1.5<b< 1.5+ \min\{\delta,  \delta'\}$, which corresponds to the point $E_0$. This completes the proof. 
\end{proof}

Let $\delta_{\rho}$ be the constant given in Proposition~\ref{pro.delta.rho}
and $\delta_{C}$ be given in Lemma~\ref{lem.cphi.sym}.
For convenience we introduce the constant 
\begin{align}
\delta_0=\min\{\delta_{\rho},  \delta_{C}\}. \label{eq.delta0}
\end{align}

\begin{proposition}\label{pro.no.new.po6}
Let $b\in(1.5, 1.5+\delta_0)$. Then there is no other periodic orbit of period $6$ that is contained in the domain $\cD_{\rho}(b)$ beside the hyperbolic and elliptic periodic orbits found in Section \ref{sec.el.hy}.
\end{proposition}

\begin{proof}
Let $b\in (1.5, 1.5+\delta_0)$.
It follows from Eq.~\eqref{eq.dlrj0} that any periodic orbit of period $6$ in the domain $\cD_{\rho}(b)$
must intersect the first quadrant $\cD_{\rho}^{+}(b):=\cD_{\rho}(b) \cap \cM^{+}$. Moreover, if $(d_{\ell}, d_{r}) \in \cD_{\rho}^{+}(b)$ is a periodic point of period $6$, then $\Psi_{b} \Phi_{b}(d_{\ell}, d_{r})=(d_{\ell}, d_{r})$ and $\Phi_{b}(d_{\ell}, d_{r})=\Psi_{b}(d_{\ell}, d_{r}) \in \cD_{\rho}^{+}(b)$. 
To prove Proposition \ref{pro.no.new.po6},
we will divide the domain $\cD_{\rho}^{+}(b)$ into smaller components and consider them separately. More precisely, the union $\cC_{\Phi_{b}} \cup \cC_{\Psi_{b}}$ divides $\cD_{\rho}^{+}(b)$ into four connected components: 
\begin{enumerate}[label = (\roman*)]
\item $\cD_{\rho}^{+}(b, 1)$ -- the closed component containing the curve $\cC_{\rho}(b)\cap M^{+}$;

\item $\cD_{\rho}^{+}(b, 2)$ -- the closed (modulo the origin) component neighboring to the origin; 

\item $\cD_{\rho}^{+}(b, 3)$ --  the upper-left open component of the domain enclosed by $\cC_{\Phi_{b}} \cup \cC_{\Psi_{b}}$;

\item $\cD_{\rho}^{+}(b, 4)$ -- the lower-right open component of the domain enclosed by $\cC_{\Phi_{b}} \cup \cC_{\Psi_{b}}$.
\end{enumerate}

(1). Note that 
any point $(d_{\ell}, d_{r}) \in \cD_{\rho}^{+}(b, 1)$ satisfies $(1,-1)\cdot\Phi_{b}(d_{\ell}, d_{r})+d_{\ell}- d_{r}\ge 0$ and $(1,-1)\cdot\Psi_{b}(d_{\ell}, d_{r})+d_{\ell}- d_{r}\le 0$.
Moreover, at most one of the two inequalities is an equality except at $P_0$, $Q_0$ or $E_0$.
The reversed inequalities hold for points in $\cD_{\rho}^{+}(b, 2)$.
Let $(d_{\ell}, d_{r}) \in \cD_{\rho}^{+}(b, 1)$ be a periodic point of period $6$.
Then $\Phi_{b}(d_{\ell}, d_{r})= \Psi_{b}(d_{\ell}, d_{r})=:(d_{\ell}', d_{r}')$. It follows that $d_{\ell}'+ d_{r}'+d_{\ell}- d_{r} \ge 0$ and $d_{\ell}'+ d_{r}'+d_{\ell}- d_{r}\le 0$ hold simultaneously.
Hence $d_{\ell}'+ d_{r}'+d_{\ell}- d_{r}=0$, and 
\begin{align*}
(d_{\ell},d_{r}) \in \cC_{\Phi_{b}} \cap \cC_{\Psi_{b}} =\{P_0, Q_0, E_0\}.
\end{align*}
The case when   $(d_{\ell}, d_{r}) \in \cD_{\rho}^{+}(b, 2)$ can be proved in the same way.
Therefore, there is no other periodic point of period $6$ in the two components $\cD_{\rho}^{+}(b, 1) \cup \cD_{\rho}^{+}(b, 2)$ besides  $P_0$, $Q_0$ and $E_0$.

(2). Consider the remaining open components
$\cD_{\rho}^{+}(b, 3)$ and $\cD_{\rho}^{+}(b, 4)$. Note that  
\begin{enumerate}[label = (\alph*)]
\item a point $(d_{\ell}, d_{r}) \in \cD_{\rho}^{+}(b, 3)$ satisfies $(1,-1)\cdot\Phi_{b}(d_{\ell}, d_{r})+d_{\ell}- d_{r}< 0$
and $(1,-1)\cdot\Psi_{b}(d_{\ell}, d_{r})+d_{\ell}- d_{r} < 0$;

\item a point $(d_{\ell}, d_{r}) \in \cD_{\rho}^{+}(b, 4)$ satisfies $(1,-1)\cdot\Phi_{b}(d_{\ell}, d_{r})+d_{\ell}- d_{r} > 0$
and $(1,-1)\cdot\Psi_{b}(d_{\ell}, d_{r})+d_{\ell}- d_{r} > 0$.
\end{enumerate}
It follows that any periodic point $(d_{\ell}, d_{r})$ of period $6$ in $\cD_{\rho}^{+}(b, 3)$ satisfies 
$\Phi_{b}(d_{\ell}, d_{r}) \in \cD_{\rho}^{+}(b, 3)$, since $\Phi(d_{\ell}, d_{r})=\Psi(d_{\ell}, d_{r})$. Therefore, $(d_{\ell}, d_{r})\in \cD_{\rho}^{+}(b, 3)\cap \Phi_{b}D_{\rho}^{+}(b, 3)$. Similarly, any periodic point $(d_{\ell}, d_{r})$ of period $6$ in $\cD_{\rho}^{+}(b, 4)$ satisfies  $(d_{\ell}, d_{r})\in \cD_{\rho}^{+}(b, 4)\cap \Phi_{b}D_{\rho}^{+}(b, 4)$.
Next we will show that there is no periodic point of period $6$ in the two open components $\cD_{\rho}^{+}(b, 3) \cup \cD_{\rho}^{+}(b, 4)$.

The curve $\cC_{\Phi_b}^{sym}$ divides the domain $\cD_{\rho}^{+}(b)$ into two components: one contains $P_0$ and the other one contains $Q_0$. Denote them by $\cD_{\rho}^{+}(b,P_0)$ and $\cD_{\rho}^{+}(b,Q_0)$, respectively. 
Let $(d_{\ell}', d_{r}') =\Phi_b(d_{\ell}, d_{r})$.
Since $d_{\ell}' -d_{r}' - d_{\ell} + d_{r} \neq 0$ on $\cD_{\rho}^{+}(b,P_0)$ and it is positive at $P_0$, it follows that  $d_{\ell}' -d_{r}' - d_{\ell} + d_{r} > 0$ on $\cD_{\rho}^{+}(b,P_0)$.
Similarly, we have  $d_{\ell}' -d_{r}' - d_{\ell} + d_{r} < 0$ on $\cD_{\rho}^{+}(b,Q_0)$.

It follows from Lemma \ref{lem.cphi.sym} that $\cD_{\rho}^{+}(b,3) \subset \cD_{\rho}^{+}(b,P_0)$ and $\cD_{\rho}^{+}(b,4) \subset \cD_{\rho}^{+}(b,Q_0)$.
Any periodic point of period $6$ in  $\cD_{\rho}^{+}(b,3)$, say $(d_{\ell}, d_{r})$, satisfies that
$(d_{\ell}', d_{r}') =\Phi_b(d_{\ell}, d_{r}) \in \cD_{\rho}^{+}(b,3) \subset \cD_{\rho}^{+}(b,P_0)$.
Since $\Phi_b$ is an involution, $\Phi_b(d_{\ell}', d_{r}') =(d_{\ell}, d_{r})$.
It follows that $d_{\ell} -d_{r} - d_{\ell}' + d_{r}' > 0$ since $(d_{\ell}', d_{r}') \in \cD_{\rho}^{+}(b,P_0)$. This contradicts the hypothesis that $(d_{\ell}, d_{r}) \in \cD_{\rho}^{+}(b,P_0)$.
In the same way we have that there is no periodic point of period $6$ in $\cD_{\rho}^{+}(b,4)$. This completes the proof.
\end{proof}

Let  $1.5 < b<1.5+\delta_0$. We have that the two hyperbolic periodic orbits of $P_0$ and $Q_0$
are contained in the invariant domain $\cD_{\rho}(b)$ bounded by the invariant circle $\cC_{\rho}(b)$ of rotation number $\rho$. Note that there is no singular points in $\cD_{\rho}(b)$, and hence every point in it alternates bounces between the two arcs $\Gamma_{\ell}$ and $\Gamma_r$ of the boundary  $\pa \cQ(b)$.
It follows that the iterates of the billiard map $F_b$ on the domain $D_{\rho}(b)$ agree with the iterates of $\cR_{b}$ and $\cL_{b}$ on the domain $\cD_{\rho}(b)$. 
Note that the stable and unstable manifolds of the hyperbolic periodic orbit of $P_0$ are contained in $\cD_{\rho}(b)$. It follows that iterates of the billiard map $F_b$ on the stable and unstable manifolds of the hyperbolic periodic orbits of period $6$ agree with the iterates of $\cR_{b}$ and $\cL_{b}$.

\begin{lemma}\label{lem.union.inv}
The union of the stable and unstable manifolds of the hyperbolic periodic orbit of period $6$ of $P_0$  is invariant under $\cL_{b}$ and $\cR_{b}$. 
\end{lemma}
\begin{proof}
Let $P\in W^s(P_0)$. That is $F^{6n}(P) \to P_0$ as $n\to \infty$. 
Then $\cL_{b}(P)\in \cL_{b}W^s(P_0)$ satisfies that 
$F^{-6n}(\cL_{b}(P))=\cL_{b}(F^{6n}(P)) \to \cL_{b}(P_0)=Q_1$ as $n\to \infty$.
It follows that $\cL_{b}W^s(P_0)=W^u(Q_1)$. Similarly, we have $\cR_{b}W^u(Q_1)=W^s(P_1)$
and $\Phi_{b}(W^s(P_0))=W^u(Q_0)$. The completes the proof.
Alternatively, note that the involution $I:(\varphi, \theta)\mapsto (\varphi, \pi-\theta)$ on the phase space $M_b$ of the billiard map $F_b$ amounts to reversing the direction of the trajectories. In particular, it takes a stable manifold of a hyperbolic periodic point to the unstable manifold of the reversed hyperbolic periodic point and vice versa. It follows that the union of the stable and unstable manifolds of the the hyperbolic periodic orbit of $P_0$ is invariant under the involution $I$. Note that $\cL_{b}$ and $\cR_{b}$ are the realizations of the involution $I$ on the left and right arc, respectively. Hence the lemma follows. 
\end{proof}

\begin{definition}\label{def.Z}
Let $Z\subset (1.5, 1.5+\delta_0)$ be the set of parameters $b$ such that there is no homoclinic intersection for the hyperbolic periodic point $P_0(b)$. 
\end{definition}
It is easy to see that the point $P_0(b)$  in this definition can be replaced by any of the remaining $5$ hyperbolic periodic points given in Remark \ref{rem.period6}. 
Numerical results \cite{CMZZ} indicate that $Z=\emptyset$, and Theorem \ref{thm.main} is equivalent to a slightly weaker version that $Z$ has no limit point in $(1.5, 1.5+\delta_0)$. 

\begin{proposition}\label{pro.homo.entr}
The  topological entropy $h_{top}(F_b)$ of the billiard map $F_b$ is positive for every $b\in (1.5, 1.5+\delta_0)\backslash Z$. 
\end{proposition}
\begin{proof}
Let $b\in (1.5, 1.5+\delta_0)\backslash Z$ be given. That is, there is a homoclinic intersection for the hyperbolic $\Theta_b$-fixed point $P_0(b)$. We claim that it cannot be a homoclinic loop. 
Suppose on the contrary that it is a homoclinic loop, say $\eta(P_0)$.
Let $D(P_0)\subset \cM$ be the domain bounded by the loop $\eta(P_0)$. Then by Proposition~\ref{pro.franks}, there must be a $\Theta_b$-fixed point contained in the interior of $D(P_0)$, say $T$. Note that the only points fixed by $\Theta_b$ are those given in Remark \ref{rem.period6} and the origin $(0,0)$. If $T=(0,0)$ or the  elliptic periodic point $E_0$, then by symmetry, $\cI(\eta(P_0))$ is a homoclinic loop for $Q_0$ and also encloses the same point $T$, a contradiction. Similarly, there is a contradiction for any other choice of $T$.

Since both stable and unstable manifolds of $P_0$ are analytic, they can only intersect at discrete points. Moreover, since the system is conservative, there must be an intersection that is a homoclinic intersection of $P_0$ that is topological crossing. See \cite[Theorem 3.1]{BW95}. Then it follows from Proposition \ref{pro.BW.crossing} that $h_{top}(F_b)>0$.
\end{proof}

\subsection{Local phase portraits}
Let $\Delta=\{d_{\ell}=d_r\}$ be the diagonal of the phase space, $P_0$ and $Q_0$
be the two hyperbolic periodic points in the first quadrant, $E_0$ be the elliptic periodic point  in the first quadrant, see Fig.~\ref{fig.sing}. 
Let $\Gamma^s_{b,P_0,+}$ and $\Gamma^u_{b, P_0,+}$ be the two branches of the stable and unstable manifolds of $P_0$ that contains the local manifolds that enter the first quadrant, respectively. Similarly, we define  $\Gamma^s_{b,Q_0,+}$ and $\Gamma^u_{b, Q_0,+}$.
See Fig.~\ref{fig.phase156}. Recall that the definition of $Z\subset (1.5, 1.5+\delta_0)$ is given in 
Definition~\ref{def.Z}. 

\begin{proposition}\label{pro.connection.PQ}
Let $b\in Z$.
Then $\Gamma^s_{b,P_0,+}=\Gamma^u_{b,Q_0,+}$ is a heteroclinic connection between $P_0$ and $Q_0$.
Moreover, this connection intersects the diagonal exactly once, and the intersection is perpendicular.
The same holds for $\Gamma^u_{b, P_0,+}=\Gamma^s_{b,Q_0,+}$. 
Furthermore,  both connections are contained in the first quadrant $\cM^{+}$
and their union encloses the elliptic periodic point $E_0$.
\end{proposition}
\begin{proof}
Let $b\in Z$ be fixed. We divide the proof into several steps:

\noindent\textbf{Step 1.} We claim that $\Gamma^s_{b, P_0,+}$ can intersect the diagonal $\Delta$ at most once.
Moreover, such an intersection point, if exists,  implies that
$\Gamma^s_{b, P_0,+}=\Gamma^u_{b, Q_0,+}$ is a heteroclinic connection.
More precisely, $\Gamma^s_{b, P_0,+}=[P_0,P]^{s}_{b} \cup [Q_0,P]^u_{b}$,
and $P$ is the intersection point of  $\Gamma^s_{b, P_0,+}$ with $\Delta$.
The same holds for $\Gamma^u_{b, P_0,+}$. 

\begin{proof}[Proof of Claim]
We only need to prove the claim for $\Gamma^s_{b, P_0,+}$.
Suppose  $P\in \Gamma^s_{b, P_0,+}\cap \Delta$ is the first intersection. Then by symmetry, $\Gamma^u_{b, Q_0,+}$ will intersect $\Delta$
at the same point $P$ for the first time. It follows that $P$ is a heteroclinic point between $P_0$ and $Q_0$. 
Suppose on the contrary that $\Gamma^s_{b, P_0,+}\neq \Gamma^u_{b, Q_0,+}$.
Then there exists a point $A \in\Gamma^s_{b, P_0,+}\cap \{(d_{\ell}, d_r) \in M^+: d_{r}<d_{\ell}\} \backslash \Gamma^u_{b, Q_0,+}$. Then by symmetry, the point $A'=\cI(A) \in\Gamma^u_{b, Q_0,+}\backslash \Gamma^s_{b, P_0,+}$. It follows that $\Gamma^u_{b, Q_0,+}$ and $ \Gamma^s_{b, P_0,+}$ have a crossing intersection between $A$ and $A'$. Such a crossing leads to a crossing heteroclinic intersection between the corresponding branches of $P_1$ and $Q_1$ (applying $\cL_{b}$) and  a crossing heteroclinic intersection between the corresponding branches of $P_2$ and $Q_2$ (applying $\cR_{b}$). Then by center-symmetry, we get crossing heteroclinic intersections of the corresponding branches of $(P_1, Q_2)$, $(P_0, Q_1)$ and $(P_2, Q_0)$. Together they form a cycle of crossing heteroclinic intersections that lie on the same side. Then the point $P_0$ admits a crossing homoclinic intersection by the lambda-lemma. This contradicts the choice $b\in Z$.
\end{proof}

\noindent\textbf{Step 2.} 
Suppose that $\Gamma^s_{b, P_0,+}\cap \{d_\ell=0\}-\{P_0\} \neq \emptyset$. 
Let $T=(0,d_r^\dagger)\neq P_0$ be the first intersection of $\Gamma^s_{b, P_0,+}$ with $\{d_\ell=0\}$.
We claim that $0<d_r^\dagger< \big(1-\frac{1}{4b^2-8}\big)^{1/2}$.

\begin{proof}[Proof of the Claim] 
It follows from Step (1) that $[P_0, T]^{s}_{b} \cap \Delta=\emptyset$.
In particular, $[P_0, T]^{s}_{b}$ is contained in the wedge $\{0<d_{\ell} < d_r<1\}$.
So there are two possibilities for $d_r^\dagger$: 
either $d_r^\dagger> \big(1-\frac{1}{4b^2-8}\big)^{1/2}$ or $0<d_r^\dagger< \big(1-\frac{1}{4b^2-8}\big)^{1/2}$.
It remains to show that it is impossible to have $d_r^\dagger> \big(1-\frac{1}{4b^2-8}\big)^{1/2}$.

Suppose on the contrary that $d_r^\dagger> \big(1-\frac{1}{4b^2-8}\big)^{1/2}$. 
Consider the following two parametric curves: 
\begin{align*}
\eta_1^{+} :=&\cL_{b}(\{0\}\times [0,1])=\{(2d_r(1-d_r^2-\sqrt{(1-d_r^2)(b^2-d_r^2)}),d_r): d_r\in [0,1]\};\\
\eta_2^{-} :=&\cR_{b}([-1,0]\times \{0\})=\{(d_r, 2d_r(1-d_r^2-\sqrt{(1-d_r^2)(b^2-d_r^2)})): d_r\in [-1,0]\}.
\end{align*}
See Fig.~\ref{fig.cross}. Note that $Q_1=\cL_{b}(P_0) \in \eta_1^{+}$
and $Q_1= \cR_{b}(P_1) \in \eta_2^{-}$.
Moreover, it is easy to check that $\eta_{1}^+$ is below the line $d_{\ell}+d_r=0$ in the region 
$0 < d_r < \big(1-\frac{1}{4b^2-8}\big)^{1/2}$
and above the line  $d_{\ell}+d_r=0$ in the region 
$\big(1-\frac{1}{4b^2-8}\big)^{1/2} < d_{r} <1$.
The pushforward stable arc $\cL_{b}([P_0, T]^{s}_{b_0})$ is an unstable arc of the hyperbolic periodic point $Q_1$, while the reflection of $\cL_{b}([P_0, T]^{s}_{b_0})$ along the line  about $d_r+d_\ell=0$ is a stable arc of $Q_1$. Since $\cL_{b}$ is orientation-reversing inside the domain $\cR_{b}$, the stable arc $\cL_{b}([P_0, T]^{s}_{b_0})$ connects $Q_1$ and $\cL_{b}(T)$ through the left side of $\eta_1^{+}$ and has to intersect the line $d_r+d_\ell=0$ via a point $T' \neq Q_1$.
Then the unstable arc, being the reflection of the stable arc with respect to $d_r+d_\ell=0$,
intersect the line $d_r+d_\ell=0$ at  the same point $T'$.
It turns out that $T'$ is a crossing homoclinic intersection for $Q_1$. It follows that $\cL_{b}(T')$ is a crossing homoclinic intersection for $P_0$, contradicting $b\in Z$. This completes the proof of the claim.
\end{proof}

\begin{figure}[htbp]
\includegraphics[height=2in]{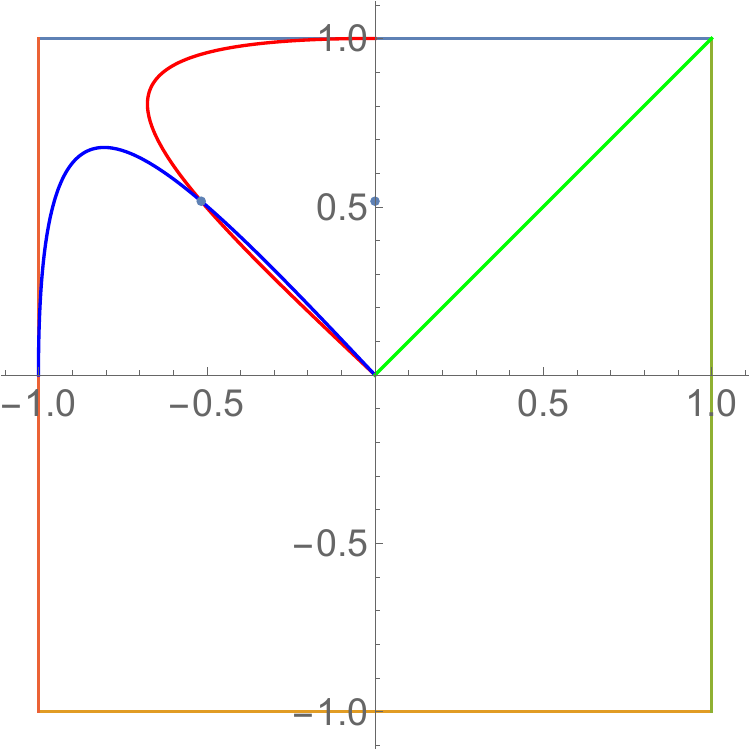}
\quad
\includegraphics[height=2in]{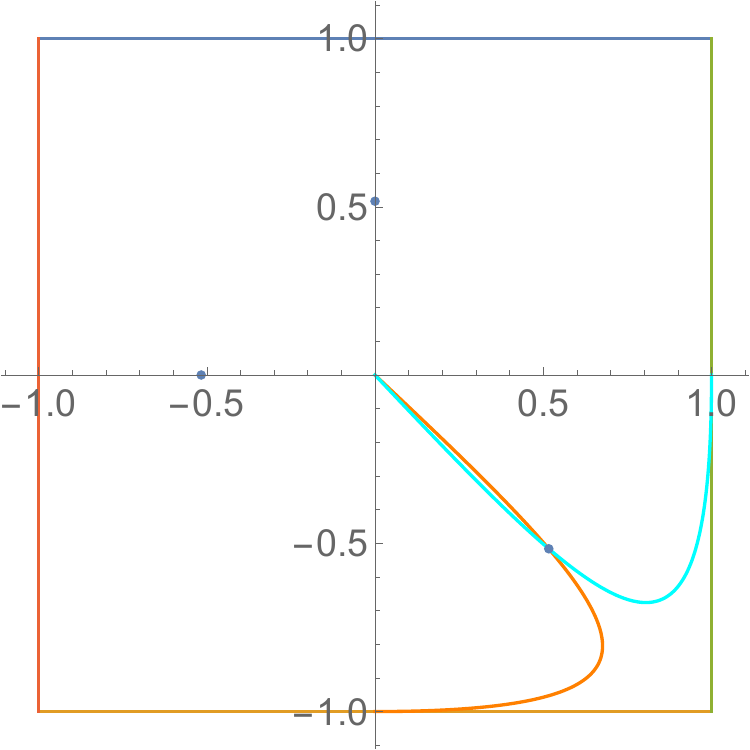}
\caption{On the left: $\eta_1^{+}= \cL_{b}(\{0\}\times [0,1])$ (red), 
$\eta_2^{-}= \cR_{b}([-1,0]\times \{0\})$ (blue).
On the right: $\eta_3^{+}= \cL_{b} \cR_{b}(\{0\}\times [0,1])$ (orange), 
$\eta_4^{-}= \cR_{b} \cL_{b}([-1,0]\times \{0\})$ (cyan).
$b=1.53$ in both plots. } \label{fig.cross}
\end{figure}

\noindent\textbf{Step 3.} 
Suppose $\Gamma^u_{b_0, P_0,+}\cap \{d_\ell=0\}-\{P_0\}\neq\emptyset$. 
Let $T=(0,d_r^\ast)\neq P_0$ be the first intersection of $\Gamma^u_{b_0, P_0,+}$ with $\{d_\ell=0\}$. 
We claim that $d_r^\ast> \big(1-\frac{1}{4b^2-8}\big)^{1/2}$ .

\begin{proof}[Proof of the Claim]
Using the same argument as in Step 2, 
we have that $[P_0,T]^u_{b_0}$ is contained in the wedge $\{0<d_{\ell} < d_r<1\}$.
So there are two possibilities for $d_r^\ast$: 
either $d_r^\ast> \big(1-\frac{1}{4b^2-8}\big)^{1/2}$ or $0<d_r^\ast< \big(1-\frac{1}{4b^2-8}\big)^{1/2}$.
It remains to show that it is impossible to have $0<d_r^\ast< \big(1-\frac{1}{4b^2-8}\big)^{1/2}$.

Suppose on the contrary that $0<d_r^\ast< \big(1-\frac{1}{4b^2-8}\big)^{1/2}$.
Let $\eta_3^{+}=\{\cL_{b} \cR_{b}(0,d_r): d_r\in [0,1]\}$, 
and $\eta_4^{-}=\{\cR_{b} \cL_{b} (d_r, 0): d_r\in [-1,0]\}$, which are symmetric with respect to the line
 $d_{\ell}+d_r=0$.
Moreover,  and $P_2=\cR_{b} \cL_{b}(Q_0)\in \eta_4^{-}$.
it is easy to check that $\eta_{3}^+$ is above the line $d_{\ell}+d_r=0$ in the region 
$-\big(1-\frac{1}{4b^2-8}\big)^{1/2} < d_{r} <0$
and below the line  $d_{\ell}+d_r=0$ in the region 
$-1 < d_{r} < -\big(1-\frac{1}{4b^2-8}\big)^{1/2}$.
The pushforward arc $\cL_{b} \cR_{b}([P_0,T]^u_{b_0})$ is an unstable arc of the hyperbolic periodic
point  $P_2=\cL_{b} \cR_{b}(P_0) \in \eta_3^{+}$,
while its reflection is a stable arc of $P_2$. Moreover, the unstable arc $\cL_{b} \cR_{b}([P_0,T]^u_{b_0})$
connects $P_2$ to $\cL_{b} \cR_{b}(T)$ through the left side of $\eta_3^{+}$ and has to intersect the line
$d_{\ell}+d_r=0$ at a point $T'\neq P_2$. Then the stable arc intersects the line $d_{\ell}+d_r=0$ at the same point $T'$, which makes $T'$ a crossing homoclinic intersection of $P_2$. 
By center symmetry, $-T'$ is a crossing homoclinic intersection of $Q_1$, and hence $\cL_{b}(-T')$ is a crossing homoclinic intersection of $P_0$, contradicting $b\in Z$.   This completes the proof of the claim.
\end{proof}

\noindent\textbf{Step 4.} 
The conclusion $0<d_r^\dagger< \big(1-\frac{1}{4b^2-8}\big)^{1/2}$ in Step 2 is also impossible,
since it would force an intersection of the unstable branch $\Gamma^u_{b_0, P_0,+}$ and the segment $\{d_\ell=0\}$
at some point with $0<d_r^\ast< \big(1-\frac{1}{4b^2-8}\big)^{1/2}$, which has been excluded in Step 3. 
If follows that $\Gamma^s_{b_0, P_0,+}$ and $\{d_\ell=0\}$ do not intersect at all.

\noindent\textbf{Step 5.} 
Similarly, the conclusion $d_r^\ast> \big(1-\frac{1}{4b^2-8}\big)^{1/2}$ in Step 3 is also impossible,
since it would force an intersection of the stable branch $\Gamma^s_{b_0, P_0,+}$ and the segment $\{d_\ell=0\}$ at some point 
with $d_r^\dagger > \big(1-\frac{1}{4b^2-8}\big)^{1/2}$, which has been excluded in Step 2. 
If follows that $\Gamma^u_{b_0, P_0,+}$ and $\{d_\ell=0\}$ do not intersect at all.

\noindent\textbf{Step 6.} We claim that $\Gamma^s_{b,P_0,+}\cap \Delta \neq \emptyset$ 
and $\Gamma^u_{b,P_0,+}\cap \Delta \neq\emptyset$.

\begin{proof}[Proof of the claim]
Suppose on the contrary that the claim is false. Without loss of generality
we assume $\Gamma^s_{b,P_0,+}\cap \Delta = \emptyset$. 
Then the whole branch $\Gamma^s_{b,P_0,+}$ is contained in the open wedge $\{(d_{\ell}, d_{r}): d_r> d_\ell> 0\}$. In the closed wedge $\{(d_{\ell}, d_{r}): d_r \ge d_\ell \ge 0\}$, there are three points that are fixed by $\Theta_b$: the only hyperbolic periodic point   $P_0$,
two elliptic periodic points: the origin $(0,0)$ and the point $E_0$.
It follows from Proposition \ref{pro.no.conn} that
\begin{enumerate}[label = (\roman*)]
\item either  $\Gamma^s_{b,P_0,+}$ is a saddle connection:
in this case, it has to be a homoclinic loop for $P_0$. Let $D$ be the simply connected domain enclosed by the homoclinic loop $\Gamma^s_{b,P_0,+}$. Note that $D$ is $\Theta_b$-invariant and is nonwandering.
It follows from Proposition \ref{pro.franks} that there is an additional $\Theta_b$-fixed point in $D$, contradicting Proposition \ref{pro.no.new.po6}.

\item or $\Gamma^s_{b,P_0,+}$ is recurrent and accumulates on both adjacent branches through the adjacent sectors of $P_0$, contradicting the fact that $\Gamma^s_{b,P_0,+}$
is contained in the wedge  $\{(d_{\ell}, d_{r}): d_r> d_\ell> 0\}$.
\end{enumerate}
We have derived a contradiction for each possible case under the assumption that the claim is false. This completes the proof of the claim.
\end{proof}

Collecting terms, we have that $\Gamma^s_{b,P_0,+}$ intersects the diagonal $\Delta$ at some point $P$, which is also the only intersection point of $\Gamma^u_{b,Q_0,+}$ with the diagonal $\Delta$ due to the 
symmetry of the system. 
It follows from Step 1 that  $\Gamma^s_{b,P_0,+} =\Gamma^u_{b,Q_0,+}$. That is, it is a heteroclinic connection and is  contained in the first quadrant. For convenience, we may denote this connection by either $\Gamma^s_{b,P_0,Q_0}$ or $\Gamma^u_{b,Q_0,P_0}$.
Applying symmetry again, its intersection with the diagonal is perpendicular.

Using a similar argument, we have that $\Gamma^u_{b,P_0,+}=\Gamma^s_{b,Q_0,+}$ is  another heteroclinic connection of $P_0$ and $Q_0$
and  is contained in the first quadrant. Then the union $\Gamma_{b,P_0,Q_0}:=\Gamma^s_{b,P_0,Q_0}\cup \Gamma^u_{b,P_0,Q_0}$ is a simple closed $\Theta_b$-invariant curve which encloses a simply connected $\Theta_b$-invariant domain $D$ in the first quadrant. 
It follows from Proposition \ref{pro.franks} that there is a $\Theta_b$-fixed point  $p\in D$. 
Since there is exactly one such point in the interior of first quadrant-- the elliptic periodic point $E_0$,
we conclude that the $\Theta_b$-fixed point contained in $D$ is exactly the elliptic periodic point $E_0$. 
This completes the proof.
\end{proof}

\begin{corollary}\label{cor.cycle.con}
For every $b\in Z$, the stable and unstable branches of all six hyperbolic periodic points $P_j$ and $Q_j$, $j=0, 1, 2$, form two invariant cycles of heteroclinic connections with its neighboring points, and each of the six elliptic periodic points is contained in a pair of heteroclinic connections
of neighboring hyperbolic periodic points. 
\end{corollary}
\begin{proof}
Let $\Gamma^{s}_{b,P_0, Q_0}$ and $\Gamma^{u}_{b,P_0, Q_0}$ be the pair of heteroclinic connections between $P_0$ and $Q_0$ given by   Proposition \ref{pro.connection.PQ}.
Then we get a pair of  heteroclinic connections between $P_1$ and $Q_1$ (applying $\cL_b$)
and a pair of  heteroclinic connections between $P_2$ and $Q_2$ (applying $\cR_b$).
See Fig.~\ref{fig.phase156} for an illustration. By center symmetry, we get a pair of heteroclinic connections between $P_{j}$ and $Q_{j+1}$, $j=0, 1, 2$, where  the indices are modulo 3. The statements about the elliptic periodic points can be proved in the same way. This completes the proof.
\end{proof}

\begin{remark}\label{rem.app.split}
Despise the appearance, Fig.~\ref{fig.phase156} shouldn't be interpreted as an evidence on the existence of heteroclinic connections.
After zooming in, one can see that the heteroclinic intersections are actually transverse with small angles.
See \cite{GeLa} for a recent survey on the splitting of separatrices. 
\end{remark}

\subsection{Hypothesis for Contradiction} \label{sec.hypo.cont}
Recall that the subset $Z\subset (1.5, 1.5 +\delta_0)$ given in Definition~\ref{def.Z} consists of parameters $b$ such that $P_0(b)$ has no homoclinic point.
We have showed that for each $b\in Z$,
$\Gamma_{b,P_0,+}^s =\Gamma_{b,Q_0,+}^u$ is a heteroclinic connection between $P_0$ and $Q_0$
that is contained in the first quadrant, and intersect the diagonal exactly once, say $T_b=(d_b, d_b)$.
Let $U\subset  (1.5, 1.5 +\delta_0)$ be the set of parameters such that both 
$\Gamma_{b,P_0,+}^s$ and $\Gamma_{b,Q_0,+}^u$ cross the diagonal. 
Clearly $U$ is open.
It follows from Proposition \ref{pro.connection.PQ}  that $U \supset Z$.
Let $U'\subset U$ be the union of the connection components of $U$ that intersect $Z$. 
For each $b\in U'$, let $T_{b,P_0}=(d_{b,P_0}, d_{b,P_0})$ be the first intersection of $\Gamma_{b,P_0,+}^s$
with the diagonal and $T_{b,Q_0}=(d_{b,Q_0}, d_{b,Q_0})$ for $\Gamma_{b,Q_0,+}^u$.
Let $\eta_{b,P_0}:(d_{b,P_0}-\eps, d_{b,P_0}+\eps) \to \bR$ be the function with $\eta(d_{b,P_0})=d_{b,P_0}$ and whose graph is 
a short segment of the stable manifold $\Gamma_{b,P_0,+}^s$ around $T_{b,P_0}$.
Similarly, define  $\eta_{b,Q_0}:(d_{b,Q_0}-\eps, d_{b,Q_0}+\eps) \to \bR$ for $\Gamma_{b,Q_0,+}^u$.
Since both invariant manifolds $\Gamma_{b,P_0, +}^s$ and $\Gamma_{b,Q_0,+}^u$ depend analytically on $b$, 
the difference  $\eta_{b,P_0}-\eta_{b,Q_0}$ also depends analytically for $b\in U$.

\vskip.1in

\noindent\textbf{(HC) Hypothesis for contradiction.} The set $Z$ has a limit point in $(1.5, 1.5 +\delta_0)$.

\vskip.1in

Note that $T_{b,P_0}=T_{b,Q_0}=T_{b}$ and  $\eta_{b,P_0}-\eta_{b,Q_0}=0$ for every $b\in Z$.
For each $b\in U'$, consider the Taylor series of the function $\eta_{b,P_0}-\eta_{b,Q_0}$, say
\begin{align}
\eta_{b,P_0}(z)-\eta_{b,Q_0}(z)=\sum_{n=1}^{\infty}c_n(b)(z-d_{b,P_0})^n.
\end{align}
By analytical dependence on the parameter $b$, we have that each term $c_n(b)$ is an analytic function of $b$ over $U'$ that vanishes on $Z$. It follows that  $c_n(b)=0$ for all $b\in U'$ and for every $n\ge 1$ and hence the function $\eta_{b,P_0}-\eta_{b,Q_0}$ is identically zero for each $b\in U'$.
The crossing condition between the branches of the invariant manifolds of $P_0$ and $Q_0$ and the diagonal is an open condition
while the equation $\eta_{b,P_0}-\eta_{b,Q_0}=0$ defines a closed condition.
It follows that $U'=U=(1.5, 1.5 +\delta_0)$, and hence $\Gamma_{b,P_0, +}^s=\Gamma_{b,Q_0,+}^u$ for every $b\in (1.5, 1.5 +\delta_0)$.
Applying the same argument to the branches $\Gamma_{b,P_0, +}^u$ and $\Gamma_{b,Q_0,+}^s$ and then the argument in Corollary \ref{cor.cycle.con}, we have that
there are two cycles of six heteroclinic connections between the six hyperbolic periodic points for $b\in  (1.5, 1.5 +\delta_0)$. See Fig.~\ref{fig.phase156} for an illustration with the phase portrait for $b=1.56$, and \cite{dOD87} for methods of plotting stable and unstable manifolds of hyperbolic manifolds using normal forms.
\begin{definition}\label{def.HCcycles}
Denote by $\cC^{\text{int}}_{b}$ the interior cycle  of heteroclinic connections containing $\Gamma^u_{b,P_0,Q_0}$ 
and by  $\cC^{\text{ext}}_{b}$ the exterior cycle of heteroclinic connections containing $\Gamma^s_{b,P_0,Q_0}$. 
Moreover, denote by $\cD_{b}^{\text{int}}$ and $\cD_{b}^{\text{ext}}$ the domains
enclosed by $\cC_{b}^{\text{int}}$ and $\cC_{b}^{\text{ext}}$, respectively.
\end{definition} 
Note that $\cD_{b}^{\text{int}} \subset \cD_{b}^{\text{ext}} \subset \cD_b$,
and both $\cD_{b}^{\text{int}}$ and $\cD_{b}^{\text{ext}}$ are invariant domains under both maps $\cL_{b}$ and $\cR_{b}$. See Definition~\ref{def.domain} for the definition of the domain $\cD_b$.

Now we study what happens for $b\in [1.5 +\delta_0, b_{\max})$, where $b_{\max}=1+ 2^{-1/2}$.
By the analytic dependence of the invariant manifolds on the parameter $b$, the two cycles $\cC^{\text{ext}}_{b}$ and $\cC^{\text{int}}_{b}$ will persist as we increase $b$ from $b=1.5 +\delta_0$ till it is broken. 
If no breakdown happens for every $b\in [1.5 +\delta_0, b_{\max})$, then the connection cycles
exist for all such cases. We will deal with this case later. For now we assume that breakdowns of the heteroclinic cycles do happen for some parameters $b\in [1.5 +\delta_0, b_{\max})$, and 
denote by $b_0\in [1.5+ \delta_0, b_{\max})$ be the infimum of such parameters.
Observe that $b_0$ is actually a minimum, since the domain $\cD_b$ is open and the existence of the interior and exterior cycles is an open condition by the analytic dependence.
 
Because of the analytic dependence again, a breakdown of a connection on the interior cycle can only happen when the connection oscillates indefinitely\footnote{Indefinite oscillations are possible since there might be new $\Theta_b$-fixed points that are either degenerate or non-stable elliptic, in which case the prime-end analysis is no long applicable.}, while a breakdown of a connection on the exterior cycle can  happen either when the connection oscillates indefinitely or when a connection falls onto one singularity curve of the domain. The indefinite oscillation breakdown for both interior and exterior cycles are similar. We will focus on the case that the cycle $\cC^{\text{ext}}_{b}$ breaks down at $b_0$, and show necessary changes to the case when only the cycle $\cC^{\text{int}}_{b}$ breaks down at $b_0$. 
We first show that all of conclusions in Proposition \ref{pro.connection.PQ} continue to hold up to the breakpoint.
\begin{proposition}\label{pro.cont.fq}
Under the hypothesis (HC),
both connections $\Gamma^s_{b, P_0,Q_0}$ and $\Gamma^u_{b, P_0,Q_0}$ are contained in the first quadrant $\cM^{+}$, intersect the diagonal exactly once  and the intersection is perpendicular, and their union encloses the  point $E_0$ for every $1.5<b< b_0$ .
\end{proposition}
\begin{proof}
We just need to examine the steps in the proof of Proposition \ref{pro.connection.PQ}.
All of the conclusions in Step 1 and Step 6 hold automatically for the analytic continuation of the heteroclinic connections  for $1.5<b< b_0$. The arguments in Step 2 and Step 3 work the same. The conclusions in Step 4 and Step 5 follow directly from Step 2 and Step 3.
Since both the point $E_0(b)$ and the curve $\Gamma^s_{b, P_0,Q_0}\cup \Gamma^u_{b, P_0,Q_0}$ depend analytically on $b$, and the point $E_0(b)$ is enclosed by the curve for $1.5<b<1.5+\delta_0$, it continues to hold for $1.5+\delta_0 \le b< b_0$.
Therefore, all conclusions in Proposition \ref{pro.connection.PQ} continue to hold for $1.5<b< b_0$.
\end{proof}

Note that for $1.5 < b < b_0$, the cycle $\cC^{\text{ext}}_{b}$ consists of the global stable/unstable manifolds of hyperbolic periodic points and is invariant under both maps $\cL_{b}$
and $\cR_{b}$. 
In the case that $\cC^{\text{ext}}_{b}$ breaks at $b=b_0$, we will denote $\ds \cC^{\text{ext}}_{b_0}=\lim_{b\nearrow b_0} \cC^{\text{ext}}_{b}$, where the limit of a family of subsets $\{E_b \subset M: b < b_0\}$ is defined as
\begin{align}
\lim_{b \nearrow b_0} E_b = \overline{\bigcup_{b< b_0} E_b \times \{b\}} \bigcap (M\times \{b_0\}) \subset M\times \bR_{b_0},
\end{align}
where we identify $M$ with $M\times\{b_0\}$. 
In the same way we can define the limit of the six saddle connections of the cycle $\cC^{\text{ext}}_{b}$. For example, $\ds \Gamma^s_{b_0,P_0, Q_0}=\lim_{b\nearrow b_0} \Gamma^s_{b_0,P_0, Q_0}$.
Note that
$\Gamma^s_{b_0,P_0, Q_0}$ is a continuum containing both points $P_0$ and $Q_0$, is contained in the first quadrant, is symmetric with respect to the diagonal, and is invariant under both $\Phi_{b}$ and $\Psi_{b}$.
For convenience, we will call $\Gamma^s_{b_0,P_0, Q_0}$ a continuum connection. 
The continuum connection  for $\Gamma^u_{b_0,P_0, Q_0}$ (if it breaks down) can be defined in the same way and it has the same property.
\begin{lemma}\label{lem.break.all}
If $\cC^{\text{ext}}_{b_0} \cap \pa \cD_{b_0} \neq \emptyset$,
then $\Gamma^s_{b_0,P_0, Q_0}\cap \pa \cD_{b_0} \neq \emptyset$.
\end{lemma}
\begin{proof}
Suppose $\cC^{\text{ext}}_{b_0} \cap \pa \cD_{b_0} \neq \emptyset$.
Then some of the six connections intersect the boundary $\cD_{b_0}$.
Because of the center symmetry, we only need to consider two cases:
\begin{enumerate}
\item  $\Gamma^{s}_{b_0,P_1, Q_1} \cap \pa \cD_{b_0} \neq \emptyset$:
note that $\Gamma^{s}_{b,P_1, Q_1}=\cL_{b}(\Gamma^{s}_{b,P_0, Q_0})$
is contained in the second quadrant. By taking limit, we see that $\Gamma^{s}_{b_0,P_1, Q_1}$
is also contained in the second quadrant and may intersect  three pieces of the boundary
(see Fig.~\ref{fig.phase156}):
\begin{enumerate}
\item $\Gamma^{s}_{b_0,P_1, Q_1} \cap L_{b_0} \neq \emptyset$: applying $\cL_{b_0}$,
we get that $\Gamma^{s}_{b_0,P_0, Q_0} \cap \cS_{1,+} \neq \emptyset$;

\item $\Gamma^{s}_{b_0,P_1, Q_1} \cap \cS_{1,+} \neq \emptyset$: applying $\cL_{b_0}$,
we get that $\Gamma^{s}_{b_0,P_0, Q_0} \cap L_{b_0}\neq \emptyset$ (an empty case since $\Gamma^{s}_{b_0,P_0, Q_0}$ is contained in the first quadrant);

\item $\Gamma^{s}_{b_0,P_1, Q_1} \cap \cS_{2,-} \neq \emptyset$: applying $\cR_{b_0}$,
we get that $\Gamma^{s}_{b_0,P_1, Q_1} \cap L_{b_0} \neq \emptyset$. It reduces to Case (a).
\end{enumerate}

\item $\Gamma^{s}_{b,P_2, Q_2} \cap \pa \cD_{b_0} \neq \emptyset$: note that $\Gamma^{s}_{b,P_2, Q_2}=\cR_{b}(\Gamma^{s}_{b,P_0, Q_0})$
is contained in the fourth quadrant. By taking limit, we see that $\Gamma^{s}_{b_0,P_2, Q_2}$
is also contained in the fourth quadrant . The discussion is similar to Case (1) and is omitted.
\end{enumerate}
Collecting terms, we conclude that $\Gamma^s_{b_0,P_0, Q_0}\cap \pa \cD_{b_0} \neq \emptyset$. This completes the proof.
\end{proof}

Although not needed, it is easy to see that Lemma \ref{lem.break.all} can be strengthened such that all six connections of the cycle $\cC^{\text{ext}}_{b_0}$ intersect the boundary $\pa \cD_{b_0}$ simultaneously.
It follows from Lemma \ref{lem.break.all} that $b_0$ is the first parameter that either the heteroclinic connection $\Gamma^s_{b,P_0,Q_0}$ or the connection $\Gamma^u_{b,P_0,Q_0}$ breaks down (maybe both). 
In the following we will discuss possible scenarios when the breakdown of the heteroclinic
connection $\Gamma^s_{b,P_0, Q_0}$ or $\Gamma^u_{b,P_0,Q_0}$ happens. We will exclude these possibilities and show that the two cycles of heteroclinic
connections continues to exist for all $b< b_{\max}=1+2^{-1/2}$.
Then we will show that the heteroclinic connection $\Gamma^s_{b,P_0, Q_0}$ must have been broken for any $b$
sufficiently close to $b_{\max}$ due to geometric reasons.
This contradiction shows that the hypothesis (HC) is wrong that $Z$ has a limit point in $(1.5, 1.5 +\delta_0)$.  Combining with Proposition \ref{pro.homo.entr}, we complete the proof of Theorem~\ref{thm.main}.

\begin{definition}\label{def.rho}
Let $b\in (1.5, b_0)$,
$\rho_b: t\in [-1,1]\hookrightarrow \Gamma^s_{b,P_0, Q_0}$ be a parametrization for the connection $\Gamma^s_{b,P_0, Q_0}$ with $\rho_b(-1)=Q_0$, $\rho_b(1)=P_0$ and $\rho_b(-t)=\cI(\rho_b(t))$. 
\end{definition}
Note that $\rho_b(0) =T_b$ is the unique intersection of $\Gamma^s_{b,P_0, Q_0}$ with the diagonal.   
Similar parametrization exists for the connection $\Gamma^u_{b,P_0, Q_0}$ with  $\hat\rho_b(-1)=P_0$, $\hat\rho_b(1)=Q_0$ and $\hat\rho_b(-t)=\cI(\hat\rho_b(t))$.

\begin{lemma}\label{rhob}
Let $1.5<b<b_0$.
Then  $\rho_b^{-1}(\Phi_{b}(\rho_b(t)))<-t$ and $\rho_b^{-1}(\Psi_{b}(\rho_b(-t)))>t$ for any $t\in [0,1)$. 
Similarly,  $\hat\rho_b^{-1}(\Phi_{b}(\hat\rho_b(t)))<-t$ and $\hat\rho_b^{-1}(\Psi_{b}(\hat\rho_b(-t)))>t$ for any $t\in [0,1)$. 
\end{lemma}

\begin{proof}
Let $1.5<b< b_0$ be fixed.
We only need to prove the statement about $\rho_b$, since the same argument works for $\hat\rho_{b}$.
Moreover, we only need to prove the inequality about $\Phi_{b}$ since $\Gamma^s_{b,P_0, Q_0}$ is invariant under $\cI$ and the parametrization is symmetric.
Let $A_b=\{t\in [0,1): \rho_b^{-1}(\Phi_{b}(\rho_b(t)))<-t\}$, which is an open subset
of $[0,1)$ for each $1.5<b< b_0$. 
It suffices to show that $A_b$ is nonempty and closed. Then we have $A_b=[0,1)$.

\noindent{\bf Step 1.} $0\in A_b$: We start with the case  $b\in (1.5, 1.5+\delta_0)$.
It follows from Lemma \ref{lem.cphi.sym} that a point $(d,d)\in \cM^{+}$ satisfies $\Phi_{b}(d,d)\in \{d_\ell=d_r\}$ if and only if  $(d,d)=E_0$. 
Applying Lemma \ref{lem.mono}, we have 
\begin{align*}
\Phi_{b}(d,d)\begin{cases}&\in \{d_\ell<d_r\}, \text{ for }d<\sqrt{1-\frac{1}{4(b-1)^2}},\\
&=(d,d), \text{ for }d=\sqrt{1-\frac{1}{4(b-1)^2}},\\
&\in \{d_\ell >d_r\}, \text{ for } d> \sqrt{1-\frac{1}{4(b-1)^2}}. 
\end{cases}
\end{align*} 
It is proved in Proposition \ref{pro.connection.PQ} that $\Gamma^s_{b,P_0, Q_0}$  goes over the elliptic periodic point $E_0$ through the northeastern corner. 
This implies that $\rho_b^{-1}(\Phi_{b}(\rho_b(0)))<0$, or equally, $0\in A_b$ for each $b\in (1.5, 1.5+\delta_0)$. 

Now assume that there exists $\tilde{b}\in [1.5+\delta_0,  b_0)$ such that $\rho_{\tilde{b}}^{-1}(\Phi_{b}(\rho_{\tilde{b}}(0)))=0$, then $\rho_{\tilde{b}}(0)=T_{\tilde{b}}$ satisfies $\Phi_{\tilde{b}}(T_{\tilde{b}})=T_{\tilde{b}}$. In particular, $T_{\tilde{b}}$ satisfies \eqref{eq.cphi}
and hence $T_{\tilde{b}} \in \cC_{\Phi_{\tilde{b}}}\cap \{d_{\ell}=d_{r}\}$, which is $\{S_n\}$ for $\tilde{b}\le b_{\crit}$ and $\{ E_0(\tilde{b}), E_{1}(\tilde{b}), E_{2}(\tilde{b})\}$ for $b>b_{\crit}$. In either case, only $E_0(\tilde{b})$ is fixed by $\Phi_{\tilde{b}}$. It follows that $T_{\tilde{b}} = E_0(\tilde{b})$.
However, there cannot be a fixed/periodic point on any stable/unstable manifold, a contradiction.
It follows that $0\in A_b$ holds for each $b\in (1.5, \min\{ b_{\crit}, b_0\})$.

\noindent{\bf Step 2.} $A_b$ is a closed subset of $[0,1)$: Let $t_n\in A_b \to t_{\ast}\in (0,1)$.
Denote $(d_{\ell}^{\ast}, d_{r}^{\ast})=\rho_b(t_{\ast})$. Then $\rho_b^{-1}(\Phi_{b}(\rho_b(t_{\ast})))\le -t_{\ast}$.
If $\rho_b^{-1}(\Phi_{b}(\rho_b(t_{\ast})))=-t_{\ast}$, then $\Phi_{b}(\rho_b(t_{\ast}))=\rho_b(-t_{\ast})$ is symmetric with respect to the diagonal.
That is,  $\Phi_{b}(d_{\ell}^{\ast}, d_{r}^{\ast})=(d_{r}^{\ast}, d_{\ell}^{\ast})$.
It follows that $\rho_b(t_{\ast}) \in \cC_{\Phi_{b}}$. Combining with Lemma \ref{lem.mono}, we have that either $\rho_b(t_{\ast})=P_0$ or $\rho_b(t_{\ast})=E_0$, contradicting the hypothesis that $t_{\ast}\in (0,1)$.

Therefore,  $A_b$ is nonempty and closed, and hence  $A_b=[0,1)$.
This completes the proof.
\end{proof}

In the following we will consider the union of  the two curves  $\cC_{\Phi_b}$ and $\cC_{\Psi_b}$ defined in \eqref{eq.cphi} and \eqref{eq.cpsi}, respectively. We need to chop off some parts of these curves since there are multi-intersections of both curves with the diagonal for $b_{\crit} < b \le b_{\max}$. More precisely,  let  $\gamma_1(b) \subset \cC_{\Phi_b}$ be the part from $P_0$ to $E_1(b)=(\sin\beta, \sin\beta)$, $\gamma_2(b) \subset \cC_{\Phi_b}$  the part from $Q_0$ to $E_2(b)=(\sin\alpha, \sin\alpha)$, where $(\alpha, \beta)\in D' \cap \cJ$ is the unique pair with $b=\fb(\alpha, \beta)$. Note that this construction also works for $1.5<b\le b_{\crit}$, just with both $E_1(b)$ and $E_2(b)$ collapsing onto $E_0(b)$. Define $\gamma_3(b) :=\cI(\gamma_1(b)) \subset \cC_{\Psi_b}$ and $\gamma_4(b) :=\cI(\gamma_2(b)) \subset \cC_{\Psi_b}$. 
It follows from Lemma \ref{lem.per.bif} that $\Phi_{b}(\gamma_1(b))=\gamma_2(b)$
and $\Psi_{b}(\gamma_3(b))=\gamma_4(b)$.

\begin{remark}
Denote by $L_c:=\{d_r-d_{\ell}=c\}$  the line in the $(d_{\ell}, d_r)$ plane with slope 1 and intercept $c$. Then for each point $(d_{\ell}, d_r) \in \gamma_1(b) \cap L_c$, it follows from \eqref{char.prl} that the point $(d_{\ell}', d_r') = \Phi_b(d_{\ell}, d_r) \in \gamma_2(b) \cap L_{-c}$.
Combining with Lemma~\ref{lemma: dl leq dr}, we see that $(d_{\ell}', d_r')$ lies on the southwest side of the point $(d_r, d_{\ell})=\cI(d_{\ell}, d_r) \in \cI(\gamma_1(b)) =\gamma_3(b)$. It follows that $\gamma_2(b)$ is contained in the interior of the domain bounded by $\gamma_1(b)$, $\gamma_3(b)$ and the two axes. In particular, the curves $\gamma_i(b)$, $1\le i\le 4$, are mutually disjoint (except the endpoints). See Fig.~\ref{fig.prl}.
\end{remark} 

Next we introduce the curve
\begin{align}
\cC_{\prl, b} = \bigcup_{1\le j \le 4}\gamma_j(b). \label{eq.C.prl}
\end{align}
See Fig.~\ref{fig.prl}. 
Note that $\cC_{\prl, b} = \cC_{\Phi_b} \cup \cC_{\Psi_b}$  for $1.5 < b \le b_{\crit}$.
It follows from the definition that $\cC_{\prl, b}$ is symmetric with respect to the diagonal. 
Moreover, it follows from Lemma \ref{lem.foc.infty} that the curve $\gamma_1(b) $ is transverse to 
the lines $L_c$, $c\in \bR$. So $\gamma_1(b)$ intersects each line $L_c$ exactly once and is a simple curve. 
Therefore, $\cC_{\prl, b}$ is 
\begin{enumerate}
\item  a closed  curve with one simple self-intersection
at $E_0$ for $1.5 < b \le  b_{cirt}$;

\item a simple closed curve enclosing a domain containing $E_0$ for $b_{\crit} < b \le b_{\max}$.
\end{enumerate}

\begin{lemma}\label{lem.prl}
For each $1.5 < b \le b_{0}$, both connections $\Gamma^s_{b,P_0, Q_0}$  
and $\Gamma^u_{b, P_0, Q_0}$
do not intersect the curve $\cC_{\prl, b}$.
\end{lemma}
\begin{proof}
We only need to give the proof for the connection  $\Gamma^s_{b,P_0, Q_0}$,
since the same argument works for the  connection $\Gamma^u_{b, P_0,Q_0}$.
See Fig.~\ref{fig.prl} for plots of the curve $\cC_{\prl, b}$.

\noindent{\bf Step 1.} 
Suppose $\Gamma^s_{b,P_0, Q_0}\cap \cC_{\prl, b}\neq\emptyset$ for some 
$1.5 < b < b_{0}$.
It follows from Lemma \ref{lem.tan.cprl} that $\Gamma^s_{b,P_0, Q_0}$ starts at the exterior component of the complement of $\cC_{\prl, b}$.
Let $A_0$ be the first intersection of $\Gamma^s_{b,P_0, Q_0}$ with $\cC_{\prl, b}$ after leaving $P_0$. Then $A_0\in\gamma_1$.
Since there is no periodic point on $\Gamma^s_{b,P_0, Q_0}$, we have $A_0\neq E_0$ for $1.5<b \le b_{\crit}$
and $A_0 \notin\{E_1(b), E_2(b)\}$ for $b_{\crit} < b \le b_0$.
Since the point $A_0$ is not on the diagonal, $A_0 =\rho_b(t_a)$ for some $0<t_a \le 1$, where $\rho_b: [-1,1] \to \Gamma^s_{b,P_0, Q_0}$
is a symmetric parametrization given in Definition \ref{def.rho}.  
Let $B_0=\Phi_{b}(A_0) \in \gamma_2$.
Since $\Phi_{b}(\Gamma^s_{b,P_0, Q_0})=\Gamma^s_{b,P_0, Q_0}$, $B_0$ is also an intersection point of $\Gamma^s_{b,P_0, Q_0}$ with $\cC_{\prl, b}$. 
It follows from Lemma \ref{rhob} that $B_0=\rho_b(t_b)$ for some $-1<t_b < -t_a$.
So $B_0$ is an intersection of point of $\Gamma^s_{b,P_0, Q_0}$ with $\cC_{\prl, b}$
after leaving $Q_0$ that happens before the symmetric point $\cI(A_0)$.
This contradicts that fact that both $\Gamma^s_{b,P_0, Q_0}$ and $\cC_{\prl, b}$ are symmetric
with respect to the diagonal. Therefore, $\Gamma^s_{b,P_0, Q_0}\cap \cC_{\prl, b}=\emptyset$ for every
$1.5 < b < b_{0}$.

\noindent{\bf Step 2.}  Suppose $\Gamma^s_{b_0,P_0, Q_0}\cap \cC_{\prl, b_0}\neq\emptyset$. Note that $b_{0}>1.5 + \delta_0$.
By continuity, it follows from Step 1 that the continuum connection $\Gamma^s_{b_0,P_0, Q_0}$ lies at the exterior component of the complement of $\cC_{\prl, b_0}$.  We will divide the remaining discussion into three subcases:

(2a). $b_0 >b_{\crit}$. Denote by $A_0$  the first intersection of $\Gamma^s_{b_0,P_0, Q_0}$ and $\cC_{\prl, b_0}$
Note that $A_0\in \gamma_1$, and $A_0 \neq E_0$, since $E_0$ is enclosed in the interior of $\cC_{\prl, b_0}$.
It follows that $\Gamma^s_{b_0,P_0, Q_0}$ and $\cC_{\prl, b_0}$ also intersect at $B_0=\Phi_{b_0}(A_0) \in \gamma_2$,
which means $\Gamma^s_{b_0,P_0, Q_0}$ have crossed $\gamma_3$. It follows that the heteroclinic connection $\Gamma^s_{b,P_0, Q_0}$ has to cross both  curves $\cC_{\Phi_{b}}$ and $\cC_{\Psi_{b}}$ for $b< b_0$ sufficiently close to $b_0$, contradicting Step 1.

(2b). $b_0< b_{\crit}$.
As in Case (2a), it suffices to show that $E_0 \notin \Gamma^s_{b_0, P_0, Q_0}$.
Suppose on the contrary that $E_0 \in \Gamma^s_{b_0, P_0, Q_0}$.   
Recall that the point $E_0$ is an elliptic $\Theta_{b_0}$-fixed point.
Since $\Gamma^s_{b_0,P_0, Q_0}$ is $\Theta_{b_0}$-invariant,
it has to accumulate on $E_0$ through both sides of any smooth curve passing through $E_0$. 
This again implies that the continuum connection $\Gamma^s_{b_0,P_0, Q_0}$ has crossed both  curves $\cC_{\Phi_{b_0}}$ and $\cC_{\Psi_{b_0}}$, contradicting Step 1 as in Case (2a).

(2c) $b_{0}=b_{\crit}$. It suffices to show that $E_0 \notin \Gamma^s_{b_0, P_0, Q_0}$. 
Suppose on the contrary that $E_0 \in \Gamma^s_{b_0, P_0, Q_0}$.  
Since $\Gamma^s_{b_0, P_0, Q_0}$ is a continuum, there are points $T_n \in \Gamma^s_{b_0, P_0, Q_0} \to E_0$ through the cusp between
$\gamma_1$ and $\gamma_3$, see the middle plot in Fig.~\ref{fig.prl}.
Note that $E_0$ is a parabolic fixed point for $\Phi_{b_0}$ with
eigenvalues $\lambda_1=-1$ and $\lambda_2=1$, and the eigenvector for $\lambda_1=-1$
is $v_1=(1,1)$. This also follows from the fact that $\Phi_{b}(\gamma_1)=\gamma_2$.
Then the points $\Phi_{b_0}(T_n) \in \Phi_{b_0}(\Gamma^s_{b_0, P_0, Q_0})=\Gamma^s_{b_0, P_0, Q_0}$ get flipped to somewhere near the cusp between $\gamma_2=\Phi_{b_0}(\gamma_1)$ and $\gamma_4=\Phi_{b_0}(\gamma_3)$. This again contradicts Step 1 as in Case (2a).

Collecting terms, we have $\Gamma^s_{b_0,P_0, Q_0}\cap \cC_{\prl, b_0}= \emptyset$. This completes the proof of the lemma.
\end{proof}

\begin{proposition}\label{pro.alt}
If $\Gamma^s_{b_0,P_0, Q_0}$ is a continuum connection, then one of the following holds: 
\begin{enumerate}
\item $\Gamma^s_{b_0,P_0, Q_0} \cap  \pa \cD_{b_0} \neq\emptyset$ (see Definition \ref{def.domain} for the definition of the domain $\cD_{b}$);

\item there are $\Theta_{b_0}$-fixed points on $\Gamma^s_{b_0,P_0, Q_0}$ other than $P_0$
and $Q_0$. 
\end{enumerate}
Similarly, if $\Gamma^u_{b_0,P_0, Q_0}$ is a continuum connection, then there are $\Theta_{b_0}$-fixed points on $\Gamma^u_{b_0,P_0, Q_0}$ other than $P_0$
and $Q_0$
\end{proposition}
\begin{proof}
We only need to consider the case that  $\Gamma^s_{b_0,P_0, Q_0}$ is a continuum connection.
We argue by contradiction. Suppose on the contrary that $\Gamma^s_{b_0,P_0, Q_0} \cap  \pa \cD_{b_0}=\emptyset$ and there is no other $\Theta_{b_0}$-fixed point on $\Gamma^s_{b_0,P_0, Q_0}$  besides $P_0$ and $Q_0$. 
It follows from Lemma~\ref{lem.break.all} that the limiting exterior cycle $\cC^{\text{ext}}_{b_0} \cap \pa \cD_{b_0} =\emptyset$. That is, $\cC^{\text{ext}}_{b_0}$ is a continuum that is invariant under both maps $\cL_{b_0}$ and $\cR_{b_0}$.
It follows that the limiting connection $\Gamma^s_{b_0,P_0, Q_0}$ is a continuum that is invariant under the map $\Theta_{b_0}$.
Applying Proposition~\ref{pro.sec.end} with $f=\Theta_{b_0}$, $A=\Gamma^s_{b_0,P_0, Q_0}$, $e=\hat{P}_0$ and $q =P_0$, we get that the prime end $\hat{P}_0$ is a sector end. The remaining of the proof follows closely Mather's argument \cite[\S 10]{Mat81}
and is included for completion.
Let $\Sigma$ be the sector of $P_0$ to which the prime end $\hat{P}_0$ is associated,
$\Gamma_1$ and $\Gamma_2$ be the two branches of $P_0$ that border $\Sigma$.
Then the corresponding arc  $\hat\Gamma_1$ of $\Gamma_1$ is a $\hat\Theta_{b_0}$-invariant  arc on the added circle $S^1$ of prime ends . One of the two endpoints of $\hat\Gamma_1$ is $e$. Denote the other endpoint by $e_1$. It is possible that $e_1 =e$. Note that $e_1$ is fixed by $\hat\Theta_{b_0}$. Let $q_1 \in \Gamma^s_{b_0,P_0, Q_0}$ be a principal point of $e_1$, which must be a fixed point of $\Theta_{b_0}$. 
It follows from our hypothesis that $q_1$ is either $P_0$ or $Q_0$.
Applying Proposition~\ref{pro.sec.end} again to $e_1$, we see that $e_1$ is also a sector end.
Let $\Sigma'$ be the sector of $q_1$ to which the prime end $e_1$ is associated,
$\Gamma_1'$ and $\Gamma_2'$ be the two branches of $q_1$ that border $\Sigma'$.
Then one of these two branches must be $\Gamma_1$. That is, $\Gamma_1$ is a saddle connection from $P_0$ to $q_1$. We divide the remaining discussion into two cases according to the point $q_1$:
\begin{enumerate}
\item $q_1 = Q_0$: then $\Gamma_1$ is a saddle connection between $P_0$ and $Q_0$
and hence $\Gamma^s_{b_0,P_0, Q_0} = \Gamma_1$ is a regular connection, contradicting the choice of the parameter $b_0$; 

\item $q_1 =P_0$, then $\Gamma_1$ is a homoclinic loop of $P_0$ contained in the first quadrant. By symmetry, $\cI(\Gamma_1)$   is a homoclinic loop of $Q_0$ contained in the first quadrant. Note that the union $\Gamma_1 \cup \cI(\Gamma_1)$ is not connected and must be disjoint with the diagonal. Such a set cannot be the limit of the saddle connections $\Gamma^s_{b,P_0, Q_0}$ that are connected and cross the diagonal.
\end{enumerate}
Collecting terms, we get a contradiction with the working hypothesis. This completes the proof. 
\end{proof}

Next we will show that none of  the two alternatives for the continuum connection $\Gamma^s_{b_0,P_0, Q_0}$  in Proposition \ref{pro.alt} can happen.
Similarly, there is no new fixed point on $\Gamma^u_{b_0,P_0, Q_0}$.
This leads to a contradiction and hence the working hypothesis (HC) must be wrong. 

\begin{lemma}\label{lem.no.cross}
The continuum connection $\Gamma^s_{b_0,P_0, Q_0}$ does not intersect $\pa \cD_{b_0}$.
\end{lemma}
\begin{proof}
Suppose on the contrary that $\Gamma^s_{b_0,P_0, Q_0}\cap \pa \cD_{b_0} \neq\emptyset$.
Then one of the following holds:
\begin{enumerate}
\item either the stable manifold $\Gamma^s_{b_0,P_0, +}$ is a finite curve that ends at a point on $\pa \cD_{b_0}$: denote such a point by $R_0(d_{\ell,0},d_{r,0})$. 
Then $R_0\in \cS_{1,+}$, and $\Gamma^s_{b_0,P_0, R_0}$ is contained in $\cD_{b_0}$
and intersects $\pa \cD_{b_0}$ tangentially at $R_0$;

\item or $\Gamma^s_{b_0,P_0, +}$ is an infinite curve with some accumulation points on $\pa \cD_{b_0}$: 
pick the accumulation point with smallest $d_{\ell}$-coordinate and denote it by $R_0(d_{\ell,0},d_{r,0})$. 
Then $R_0\in \cS_{1,+}$.
\end{enumerate}
 
Note that $R_0 \in \cS_{1,+}$ implies $d_{\ell,0} \le 2^{-1/2} \le d_{r,0}$. 
Let $T_0=(d_{r,0},d_{\ell,0})$, which is a tangential intersection point of $\Gamma^u_{b_0,Q_0,+}$ with $\cS_{2,+}$
in Case (1), and is an accumulation point of $\Gamma^u_{b_0,Q_0,+}$ on $\cS_{2,+}$ in Case (2). Since our arguments below work the same for both cases, we will not distinguish these two cases, denote the corresponding branches by $\Gamma^s_{b_0,P_0, \ast}$ and $\Gamma^u_{b_0,Q_0, \ast}$ and  call $R_0$ an intersection from now on. Denote
\begin{enumerate}
\item $R_1(d_{\ell,1},d_{r,0}):=\cL_{b}(R_0)$, which is the intersection of $\Gamma^u_{b_0, Q_1, \ast}=\cL_{b}\Gamma^s_{b_0, P_0, \ast}$
with $\cL_{b}\cS_{1,+}\subset L_{b}$;

\item $R_2(d_{\ell,1},d_{r,1}):=\cR_{b}(R_1)$, which is the intersection of $\Gamma^s_{b_0, P_1, \ast}=\cR_{b}\Gamma^u_{b_0, Q_1, \ast}$
with $ \cS_{2,-}\subset \cR_{b}(L_{b})$;

\item $T_1(d_{\ell,2},d_{\ell,0}):=\cL_{b}(T_0)$, which is the intersection of $\Gamma^u_{b_0, Q_0, \ast}=\cL_{b}\Gamma^s_{b_0, P_1, \ast}$
with $\cL_{b} \cS_{2,+}$. 
\end{enumerate}

We will derive a contradiction in three steps, each step excluding a possibility for $b_0$.
Putting them together, we conclude that the hypothesis  $\Gamma^s_{b_0,P_0, Q_0}\cap \pa \cD_{b_0} \neq\emptyset$ does not hold. Recall the point $S_0=(2^{-1/2},2^{-1/2})$ defined right before Definition \ref{def.domain}.

\noindent{\bf Step 1.}
Denote $\cL_{b}(S_0)=(d_{\ell,1},2^{-1/2})$ and $\cR_{b}\cL_{b}S_0=(d_{\ell,1}, d_{r,1})$, where 
$d_{\ell,1} = 2^{-1/2}- b$, and $d_{r,1}=d_{\ell,1}-b(1-2d_{\ell,1}^2)$.
Let $d_{\ell}(b)= -\frac{1+\sqrt{8 b^2+1}}{4 b}$ be the solution of equation $d_{\ell}-b(1-2d_{\ell}^2)=0$, and $b_2\approx 1.58885$ be the positive solution of $d_{\ell}(b)=2^{-1/2}- b$, or equally, $b -\frac{1+\sqrt{8 b^2+1}}{4 b}=2^{-1/2}$, see Fig.~\ref{fig.no.boundary}.
Then we have  $d_{r,1} =d_{\ell,1}-b(1-2d_{\ell,1}^2) < 0$ for $1.5 < b< b_2$.

It follows from Proposition \ref{pro.cont.fq} that the heteroclinic connection $\Gamma^s_{b, P_0, Q_0}$ is contained in the first quadrant for $1.5<b <b_0$. By taking limit $b\to b_0$, we conclude that  the continuum connection $\Gamma^s_{b_0, P_0, Q_0}$ is contained in the first quadrant, which implies 
 $\cL_{b}(\Gamma^s_{b_0, P_0, Q_0})=\Gamma^u_{b_0, Q_1, P_1}$ is contained in the  second quadrant. Since $d_{r,1} =d_{\ell,1}-b(1-2d_{\ell,1}^2) < 0$ for $1.5 < b< b_2$, it is impossible for $\Gamma^u_{b_0, Q_1, P_1}$ to have an intersection at $R_0 \in \cS_{1,+}$ if $b_0< b_2$.
This completes Step 1. \\

\noindent{\bf Step 2.}
Let $(d_{\ell,1}, d_{r,1})=\cR_{b} \cL_{b}(S_0)$, and 
$b_3\approx 1.62326$ be the parameter $b$ with $d_{\ell,3}:=(\Phi_{b}(d_{\ell,1}, d_{r,1}))_1=0$, see Fig.~\ref{fig.no.boundary}.
Then $d_{\ell,3} <0$ for $b_2 \le b <  b_3$.
Let $\hat d_r:=b-\frac{1+\sqrt{8 b^2+1}}{4 b} \in (2^{-1/2}, 1)$,
which satisfies $\cR_{b}  (\hat d_r-b, \hat d_r) =(\hat d_r-b, 0)$.
Denote $\eta_{1,+}=\{(d_{r}+b(1-2d_r^2), d_r): 2^{-1/2} < d_r < \hat d_r\} \subset \cS_{1,+}$, which is the part satisfying that $\eta_{2,-}:=\cR_{b}\cL_{b}(\eta_{1,+})$ is contained in the second quadrant.
In the following we will work with a slightly smaller bound $\hat b_3= 1.622 < b_3$. 

\begin{figure}[htbp]
\begin{overpic}[height=2in, unit=0.1in]{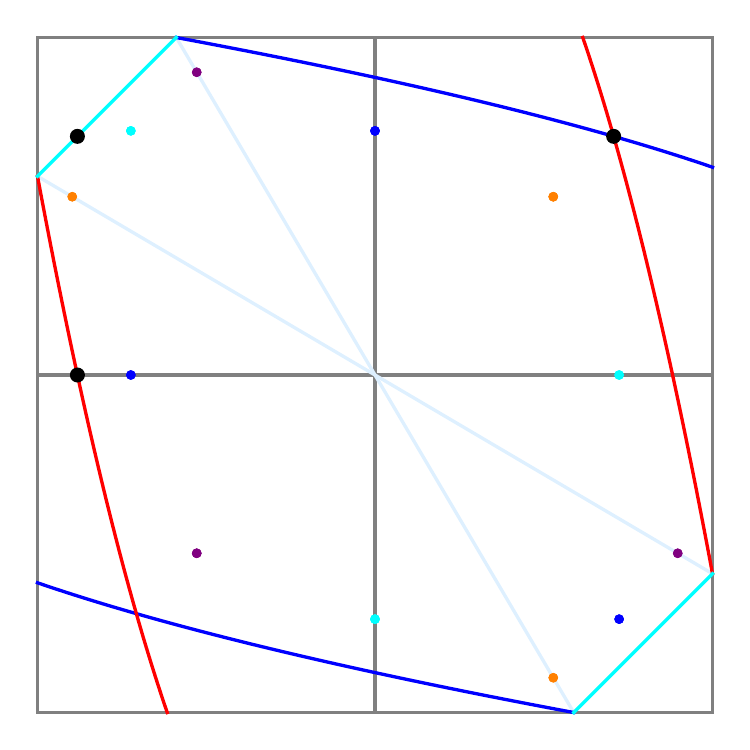}
\put (16.7, 16.7) {$S_0$} %
\end{overpic}
\begin{overpic}[height=2in, unit=0.1in]{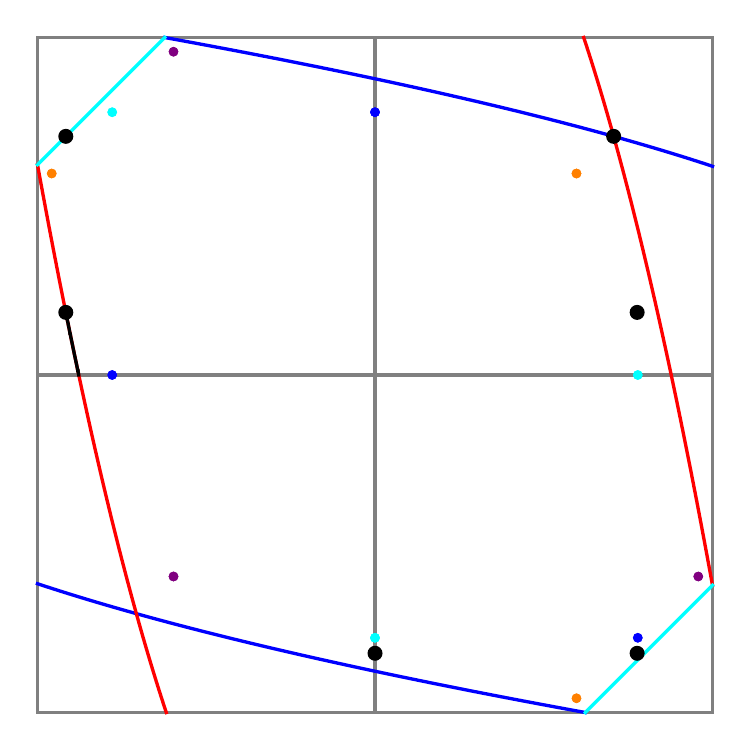}
\put (16.7, 16.7) {$S_0$} %
\put (2, 11) {$\eta_{2,-}$} %
\end{overpic}
\begin{overpic}[height=2in, unit=0.1in]{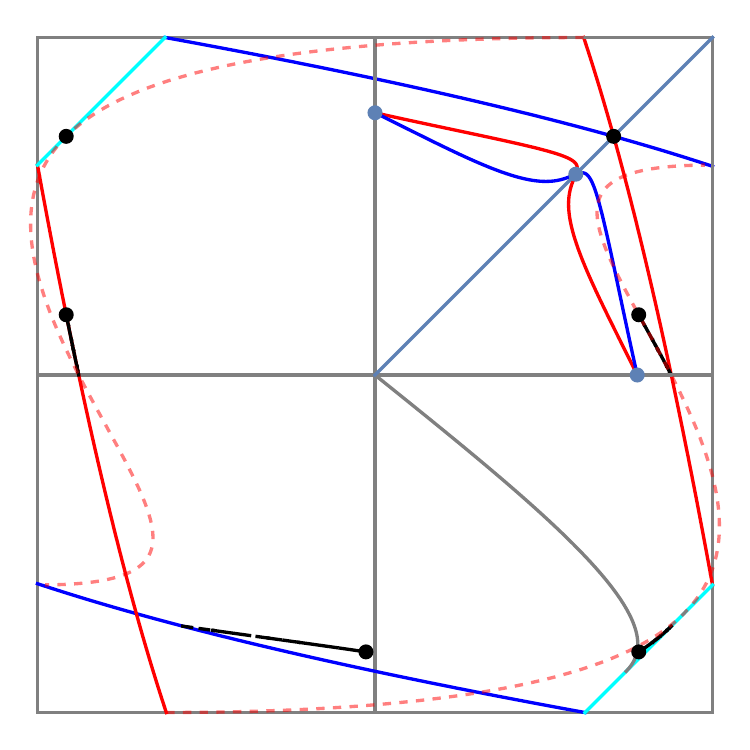}
\put (16.7, 16.7) {$S_0$} %
\put (2, 11) {$\eta_{2,-}$} %
\put (6, 3.5) {${\scriptstyle \Phi_b\eta_{2,-}}$} %
\end{overpic}
\caption{Some plots of the orbit of $S_0$. Left: $b=1.58885$. Middle: $b=1.62326$. Right: $b=1.622$. } \label{fig.no.boundary}
\end{figure}
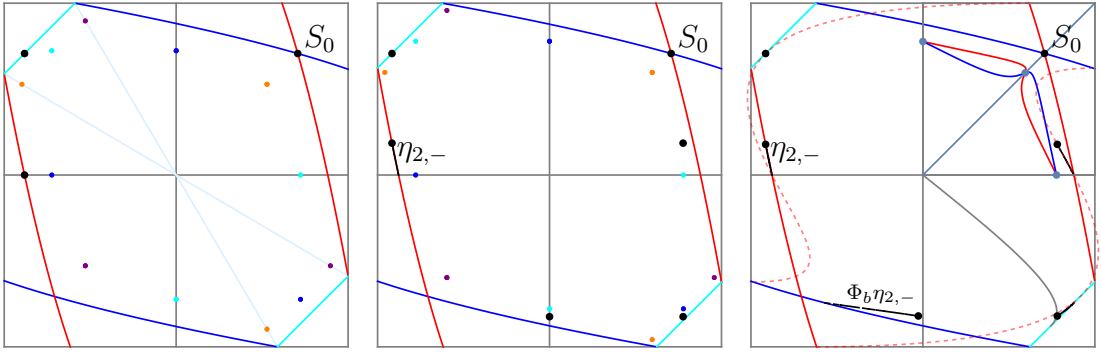

\noindent{\bf Claim.} For $b_2 \le b \le \hat b_3=1.622$, the curve $\Phi_{b} (\eta_{2,-})$
is contained in the third quadrant.

\vskip.1in

\begin{proof}[Proof of the Claim]
Note that two endpoints of $\eta_{2,-}=\cR_{b}\cL_{b}(\eta_{1,+})$ are given by
\begin{align*}
&\cR_{b} \cL_{b}(S_0) =\cR_{b} (2^{-1/2} -b, 2^{-1/2})
= (2^{-1/2}-b, 2^{-1/2}-b - 2\sqrt{2}b^2 + 2b^3); \\
&\cR_{b} \cL_{b}(\hat d_r+b(1-2\hat d_r^2), \hat d_r) =\cR_{b}  (\hat d_r-b, \hat d_r)
=(\hat d_r-b, 0).
\end{align*}
Note that the quantity $2^{-1/2}-b - 2\sqrt{2}b^2 + 2b^3$ is monotone increasing with respect to 
$b$ on the interval $b_2 \le b \le \hat b_3$, on which it satisfies
$0\le 2^{-1/2}-b - 2\sqrt{2}b^2 + 2b^3 < 0.18$.
It is clear that $\cL_{b}(\hat d_r-b, 0)=(b- \hat d_r, 0)$ and $\cR_{b} (b- \hat d_r, 0)=(b- \hat d_r, - \hat d_r)$.
Moreover, the point $\cL_{b}(b- \hat d_r, - \hat d_r)$ is in the third quadrant. By the choice of $b_3$ and the fact that $\hat b_3 < b_3$, 
we have that $\Phi_{b}(2^{-1/2}-b, 2^{-1/2}-b - 2\sqrt{2}b^2 + 2b^3)$ is in the third quadrant
when  $b_2 \le b \le \hat b_3$.
So to prove the claim, it suffices to show that $\Phi_{b}(\eta_{2,-})$ 
does not intersect the line $d_{\ell}=0$, which is equivalent to 
that $\cR_{b} \cL_{b}(\eta_{2,-})$ does not intersect the curve 
$\{\cL_{b}(0,d_r):-\frac{1+\sqrt{8 b^2+1}}{4 b}< d_r < 0\}$ (see the gray curve in the four quadrant of the plot $b=1.622$ in Fig.~\ref{fig.no.boundary}).
We just need to compare the first coordinates of these two curves. 
Since the map $\cR_{b}$ does not change the first coordinate, 
we can even use $\cL_{b}(\eta_{2,-})$ instead of $\cR_{b} \cL_{b}(\eta_{2,-})$.

(1) Let $(d_{\ell},d_r)\in \eta_{2,-}$, where $d_{\ell}=-\frac{1+\sqrt{8 b^2+8 b d_r+1}}{4 b}$
and $0\le  d_r \le 2^{-1/2}-b - 2\sqrt{2}b^2 + 2b^3$. 
Then the component $\cL_{b}(d_{\ell}, d_r)_1$ consists of two parts $I+ I\!I$, where
\begin{align}
I&=d_r+(1-2 d_r^2) \Big(\frac{(8 b^2+8 b d_r+1)^{1/2}+1}{4 b}+d_r \Big), \label{eq.fl1I}\\
I\!I&=-2 d_r(1-d_r^2)^{1/2} \Big(b^2 - \Big(\frac{(8 b^2+8 b d_r+1)^{1/2}+1}{4 b}+d_r\Big)^2\Big)^{1/2}.  \label{eq.fl1II}
\end{align}
Note that for fixed $d_r$, it is easy to see that $I$ is monotone decreasing with respect to $b$
(as long as $0< d_r < 2^{-1/2}$).
For $I\!I$, it is also monotone decreasing since $b^2 - \Big(\frac{(8 b^2+8 b d_r+1)^{1/2}+1}{4 b}+d_r\Big)^2$ is monotone increasing. 
Therefore, the minimum of  $\cL_{b}(\eta_{2,-})_1$ is achieved when  $b=\hat b_3$.
Plugging in  $b=\hat b_3$, we have that $\cL_{b}(d_{\ell}, d_r)_1 \ge 0.78143$ for every $(d_{\ell},d_r)\in \eta_{2,-}$.

(2) It is easy to see that the component $\cL_{b}(0,d_r)_1=-2 d_r(\sqrt{(1-d_r^2)(b^2-d_r^2)}+d_r^2-1)$ is monotone increasing with respect $b$ on the domain $-1< d_r< 0$. Therefore, the maximum of  $\cL_{b}(0,d_r)_1$ is achieved when  $b=\hat b_3$. At the parameter $b=\hat b_3$, we find that $\cL_{b}(0,d_r)_1\le 0.778575$ for every $-1\le d_r \le 0$.

Collecting terms, we conclude that the max of $\cL_{b}(\{0\}\times [-1,0])_1$ is smaller than the minimum of $\cL_{b}(d_{\ell}, d_{\ell}- b(1- 2d_{\ell}^2))_1$ for $b_2 \le b \le \hat b_3$. It follows that these two curves do not intersect. This completes the proof of the claim.
\end{proof}

Note that an intersection of $\Gamma^s_{b_0,P_0, Q_0}$ with $\cS_{1,+}$, if happens,  lies in the curve $\eta_{1,+}$.
Moreover, $\Phi_{b} \cR_{b} \cL_{b}(\Gamma^s_{b_0,P_0, Q_0})$ is a continuum connection from $Q_2$ to $P_2$
and hence contained in the fourth quadrant.
Since $\Gamma^s_{b,P_0, Q_0}$ and $\cS_{1,+}$ do not intersect for $b_2<b \le \hat b_3$, it is impossible to have $b_2 \le b_0 \le \hat b_3$. This completes Step 2.

\vskip.1in

\noindent{\bf Step 3.} Next we show that it is impossible to have $\hat b_3 \le b_0 \le b_{\max}$. We claim that  the curve $\cL_{b}(\cS_{2,-}(b))$ intersects the curve $\cC_{\prl, b}$ 
for $\hat b_3 \le b \le b_{\max}$. See the plot $b=1.622$ in Fig.~\ref{fig.no.boundary}, in which the dashed curves are $\cL_{b}(\cS_{2, \pm}(b))$.

\vskip.1in

\begin{proof}[Proof of the claim]
Consider the point $(d_{\ell}, d_r)$ on $\cS_{2,-}(b)$ with $0 \le d_r \le b_{\max}-1$.
As we have seen in Eq.~\eqref{eq.fl1I} and \eqref{eq.fl1II} in Step 2, $\cL_{b}(d_{\ell}, d_r)_1$ is monotone decreasing with respect to $b$  for each fixed $d_r$. To prove the claim, it suffices to consider the case that $b=\hat b_3$.

The part $\gamma_3 \subset \cC_{\prl, b}$ 
is given by the reflected version of Eq.~\eqref{abdldr} (since Eq.~\eqref{abdldr} itself corresponds to $\gamma_1$), in which the coordinate $(\alpha, \beta)$ satisfies
\begin{align}
\fb^2=
1+\frac{\sin \alpha+\sin \beta}{\sin (\alpha+\beta)}
+\frac{\sin \alpha \sin \beta}{\sin^2(\alpha+\beta)}
+\frac{(\sin \alpha-\sin \beta)^2}{\sin^2(2 \alpha+2 \beta)} = 1.622^2. \label{bsq1622}
\end{align}
See \eqref{bsquare}.
For concreteness, we set $d_r=0.45$. Then it follows from \eqref{abdldr} and \eqref{bsq1622} that 
$\sin\alpha=0.37315$ and $\hat d_{\ell}=\sin\beta=0.679152$.

On the other hand, for $(d_{\ell},d_r) \in \cS_{2,-}$ with $d_r=0.45$, we have $d_{\ell}=-\frac{1+\sqrt{8b^2+8b d_{r} +1}}{4b} \approx -0.968056$ and $\cL_{b}(d_{\ell},d_r)_1 \approx 0.660888 < \hat d_{\ell}=0.679152$.
It follows that $\cL_{b}(\cS_{2,-}(b))$ already crosses the curve $\cC_{\prl, b}$ before reaching $d_r=0.45$ for all $\hat b_3 \le b \le b_{\max}$. This completes the proof of the claim.
\end{proof}

Note that  the unstable manifold $\Gamma^u_{b_0, Q_0, \ast}$ cannot cross the curve  $\cL_{b}(\cS_{2,-}(b_0))$,
since such a crossing would imply that $\Gamma^s_{b, P_1,Q_1}$ crosses the boundary $\cS_{2,-}(b)$ for any $b$ sufficiently close to $b_0$,
contradicting our choice of the parameter $b_0$.
If $\hat b_3 \le b_0 \le b_{\max}$,
then $\Gamma^u_{b_0,Q_0, \ast}$ has to follow $\cL_{b}(\cS_{2,-})$ and  intersect the curve $\cC_{\prl, b_0}$  in order for for it to reach the boundary $\pa \cD_{b_0}$, which contradicts Lemma \ref{lem.prl}. Therefore, it is impossible to have $\hat b_3 \le b_0 \le b_{\max}$. This completes Step 3.

Collecting the conclusions from these three steps, we have exhausted all possibilities for $1.5<b_0 \le b_{\max}$ under the hypothesis
that $\Gamma^s_{b_0, P_0, Q_0} \cap \pa \cD_{b_0} \neq\emptyset$.
This completes the proof of the lemma.
\end{proof}

\begin{lemma}\label{lem.no.pp}
There is no other $\Theta_{b_0}$-fixed point on $\Gamma^s_{b_0,P_0, Q_0}$ or  $\Gamma^u_{b_0,P_0, Q_0}$ besides $P_0$ and $Q_0$.
\end{lemma}
\begin{proof}
We only need to consider the case $\Gamma^s_{b_0,P_0, Q_0}$ is a continuum connection.
Suppose on the contrary that there were other $\Theta_{b_0}$-fixed points on $\Gamma^s_{b_0,P_0, Q_0}$. Pick one of them, say $T_0(d_{\ell,0}, d_{r,0})$. Since $\Gamma^s_{b_0,P_0, Q_0}$ is symmetric with respect to the diagonal, 
we assume $d_{\ell,0} \le d_{r,0}$ without losing any generality.
Let $T_1(d_{\ell,1}, d_{r,1}) := \Phi_{b}(T_0)$. Since $\Gamma^s_{b_0,P_0, Q_0}$ lies at the exterior of $\cC_{\prl, b_0}$, it follows from Lemma \ref{lem.mono} that
\begin{align}
d_{\ell,1}-d_{r,1}  > d_{r,0} -d_{\ell,0} \ge 0. \label{dlr10}
\end{align}
Note that $T_1=\Phi_{b}(T_0) \in \Phi_{b}(\Gamma^s_{b_0,P_0, Q_0}) =\Gamma^s_{b_0,P_0, Q_0}$,
which also  lies at the exterior of $\cC_{\prl, b_0}$.
Applying Lemma \ref{lem.mono3} to $T_2(d_{\ell,2}, d_{r,2}) := \Psi_{b}(T_1)$, we have
\begin{align}
d_{r,2} -d_{\ell,2} > d_{\ell,1}-d_{r,1}. \label{dlr21}
\end{align}
However, since $T_0$ is a periodic point of period $6$, we have $T_2=\Psi_{b}(T_1)=\Psi_{b}\Phi_{b}(T_0)=T_0$, and hence $d_{r,2} -d_{\ell,2}= d_{r,0} -d_{\ell,0}$,
contradicting \eqref{dlr10} and \eqref{dlr21}.
\end{proof}

\begin{proof}[Proof of Theorem \ref{thm.main}]
Let $Z$ be the set of parameters $b\in (1.5, 1.5+\delta_0)$ such that $P_0(b)$ has no homoclinic point. Suppose (HC) holds that  the set $Z$ has a limiting point in $(1.5, 1.5+\delta_0)$.
Collecting the results from Proposition \ref{pro.alt}, Lemma \ref{lem.no.cross} and Lemma~\ref{lem.no.pp}, we conclude that none of the heteroclinic connections
$\Gamma^s_{b,P_0, Q_0}$ or $\Gamma^u_{b,P_0, Q_0}$ breaks down at any $1.5+\delta< b < b_{\max}=1+2^{-1/2}$. 
So $b_0 \ge b_{\max}$. 
On the other hand,
as $b$ increases from $b_{\crit}$ to $b_{\max}$, the exterior corner point $E_1(b)\in \cC_{\prl, b}$ slides along the diagonal and already passes through the point $S_0$ at $b=b_{\sing}\approx 1.67892$.
It follows that the exterior part of the domain $\cD_{b}\cap \cM^{+}\backslash \cC_{\prl, b}$ breaks into two connected components containing $P_0$ and $Q_0$, respectively for every
$b\ge b_{\sing}$. Since $\Gamma^s_{b,P_0, Q_0}$ connects $P_0$ and $Q_0$ through the exterior part of the domain $\cD_{b}\cap \cM^{+}\backslash \cC_{\prl, b}$, we must have $b_0 \le b_{\sing}$, contradicting the fact that $b_0 \ge b_{\max}$.
Therefore, the  hypothesis (HC) does not hold. Combining with Proposition \ref{pro.homo.entr}, we complete the proof. 
\end{proof}

\begin{proof}[An alternative proof of Theorem \ref{thm.main}]
It suffices to show that it is impossible to have $b_0 \ge b_{\max}$.
Note that $\Gamma^s_{b,P_0, Q_0}$ does not intersect the set $\cC_{\prl, b} \cup\pa \cD_{b}$  for every $1.5 <b< b_{\max}$. It follows that the whole segment $I=\{(x,x):  0\le x\le 2^{-1/2}\}$ is contained in the domain $U$ that is invariant under both $\cL_{b}$ and $\cR_{b}$. In particular, $\cL_{b}(I) \subset U$. On the other hand, $\cL_{b}(x,x)=(x-2bx(1-x^2)^{1/2},x)$, and $\cL_{b_{\max}}(0.6,0.6)=(-1.03882, 0.6)$ already exceeds the domain $U$, contradicting the fact that $\cL_{b_{\max}}(I) \subset U$.	
\end{proof}

\begin{remark}
There are many dynamical systems with persistent existence of saddle connections. A classical example is the elliptic billiards \cite[\S~4]{Tab05}, where the stable and unstable manifolds of the periodic orbit along the major axis form a cycle of saddle connections. McMillan \cite{Mcm71} discovered a class of integrable planer mappings of the form $(q, p) \mapsto (p, -q+f(q))$, where the force function $f(q)$ is the quotient of two quadratic polynomials. Similar results can be found in \cite{Sur89}. See \cite[\S~2]{ZKN} for phase portraits of their mappings. 
For comparison, the strategy of the proof of Theorem \ref{thm.main} is to derive a contradiction from the hypothesis (HC) on the persistent existence of saddle connections. 
As we have indicated in Remark~\ref{rem.app.split}, such connections do not actually exist.
\end{remark}

\section*{Acknowledgment}
The authors extend their sincere gratitude to the anonymous referees for their invaluable comments and insightful suggestions. Their dedicated efforts have significantly enhanced the clarity and presentation of this paper.

\end{document}